\documentclass{article}
\usepackage{fullpage}

\usepackage{epsfig}  
\usepackage{graphicx}  
\usepackage{amsmath}  
\usepackage{amssymb}  
\usepackage{amsthm}
\usepackage{url}  
\usepackage{lscape}  
\usepackage{bm} 
\usepackage{bbm}
\usepackage{algorithm,algorithmic} 
\usepackage{epstopdf} 
\usepackage[T1]{fontenc} 
\usepackage{float}
\usepackage{cite}
\usepackage{mathrsfs}
\usepackage{enumerate}
\usepackage[hypcap=true]{caption}
\usepackage{color}

\usepackage{xr-hyper}
\usepackage{hyperref}
\usepackage{cleveref}
\usepackage{enumitem}

\usepackage[font=footnotesize]{subfig}
\usepackage{amssymb}
\usepackage{etoolbox}

\usepackage{tikz}

\usetikzlibrary{calc,trees,positioning,arrows,chains,shapes.geometric,%
decorations.pathreplacing,decorations.pathmorphing,shapes,%
matrix,shapes.symbols,positioning,fit}

\usepackage{tabularx}  

\usepackage{mathptmx}

\makeatletter
\def\namedlabel#1#2{\begingroup
#2%
\def\@currentlabel{#2}%
\phantomsection\label{#1}\endgroup
}

\tikzstyle{decision} = [diamond, draw, fill=blue!20, 
text width=4.5em, text badly centered, node distance=3cm, inner sep=0pt]
\tikzstyle{block} = [rectangle, draw, fill=none, 
text width=6em, text centered, rounded corners, minimum height=4em]
\tikzstyle{line} = [draw, -latex']
\tikzstyle{cloud} = [draw, ellipse,fill=red!20, node distance=3cm,
minimum height=2em]

\newcounter{dir_est_ineq_cnt}
\newcounter{K_rho_unknown}

\newcolumntype{L}[1]{>{\raggedright\arraybackslash}p{#1}}
\newcolumntype{C}[1]{>{\centering\arraybackslash}p{#1}}
\newcolumntype{R}[1]{>{\raggedleft\arraybackslash}p{#1}}

\def\cw#1{#1}
\def\an#1{#1}

\newtheorem{lem}{Lemma}
\newtheorem{thm}{Theorem}

\def\argmin{\mathop{\rm \text{arg \hspace{-0.5mm}min}}}

\def\bx{\bm{x}}

\def\bz{\bm{z}}
\def\bw{\bm{w}}

\def\btheta{\bm{\theta}}

\def\fRV{\bm{\eta}}

\def\xSp{\mathcal{X}}

\newcommand{\inprod}[2]{\left\langle #1 , #2 \right\rangle }

\newcommand{\ipm}[3]{ \gamma_{#1}(#2,#3) }

\def\numIter{K}
\def\iterIndex{k}
\def\lossFunc{\ell}

\newcommand{\sfunc}[3]{\hat{f}_{#3}\left(#1,#2\right)}
\newcommand{\sgrad}[3]{\bm{g}_{#3}\left(#1,#2\right)}
\newcommand{\shess}[3]{\bm{g}^{(2)}_{#3}\left(#1,#2\right)}

\newtoggle{useMeanGap}
\toggletrue{useMeanGap}

\iftoggle{useMeanGap}{
\def\meangap{mean criterion}

\def\MeanGap{Mean Criterion}

}{
\def\meangap{excess risk}

\def\MeanGap{Excess Risk}

}

\newtoggle{useArxiv}
\toggletrue{useArxiv}

\newtoggle{useTwoColumn}
\togglefalse{useTwoColumn}

\iftoggle{useTwoColumn}{
\newcommand{\myfiguresize}{3.45 in}

}{
\newcommand{\myfiguresize}{5 in}
}

\title{Adaptive Sequential Stochastic Optimization}

\author{
  Craig Wilson\thanks{This work was supported by the US National Science Foundation under award CCF 1111342 and NSF DMS 1312907, and by the US Army Research Laboratory under cooperative agreement W911NF-17-2-0196,  through the University of Illinois at Urbana-Champaign. Parts of this work were presented at CDC 2014 \cite{Wilson2014} and ICASSP2016 \cite{Wilson2015a}.}\\
  Google\\
  wilson60@illinois.edu
  \and
  Venugopal Veeravalli\\
  University of Illinois at Urbana-Champaign\\
  vvv@illinois.edu
  \and
  Angelia Nedi\'c\\
  Arizona State University\\
  Angelia.Nedich@asu.edu
}

\date{}

\begin{document}

\maketitle

\begin{abstract}
	A framework is introduced for sequentially solving convex stochastic minimization problems, where the objective functions change slowly, in the sense that the distance between successive minimizers is bounded. 
	The minimization problems are solved by sequentially applying a selected optimization algorithm, such as stochastic gradient descent (SGD), based on drawing a number of samples in order to carry the iterations. 
	Two tracking criteria are introduced to evaluate approximate minimizer quality: one based on being accurate with respect to the mean trajectory, and the other based on being accurate in high probability (IHP). An estimate of a bound on the minimizers' change, combined with properties of the chosen optimization algorithm, is used to select the number of samples needed to meet the desired tracking criterion. A technique to estimate the change in minimizers is provided along with analysis to show that eventually the estimate upper bounds the change in minimizers. This estimate of the change in minimizers provides sample size selection rules that guarantee that the tracking criterion is met for sufficiently large number of time steps. Simulations are used to confirm that the estimation approach provides the desired tracking accuracy in practice, while being efficient in terms of number of samples used in each time step.
\end{abstract}


\section{Introduction}

Problems involving optimizing a sequence of functions that slowly vary over time naturally arise in many different contexts including channel estimation, parameter tracking, and sequential learning. To describe and analyze such problems, we consider solving a sequence of stochastic convex optimization problems
\begin{equation}
\label{prob_form:f_def}
\min_{x \in \xSp} \left\{  f_{n}(\bx) \triangleq \mathbb{E}_{\bz_{n}} \left[ \lossFunc_{n}(\bx,\bz_{n})  \right]  \right\}
\end{equation}
with $\ell_{n}$ being an appropriate loss function, $\bz_{n}$ representing the randomness in the loss at time $n$, and $\bx \in \xSp \subset \mathbb{R}^{d}$ being a nonempty, closed and convex set. 
	We will assume that problem~\eqref{prob_form:f_def} has a unique solution, denoted by  $\bx_{n}^{*}$, at every instance $n$, i.e.,
	\[
	\bx_{n}^{*} = \argmin_{\bx \in \mathcal{X}} f_{n}(\bx)\qquad \hbox{for all }n\ge 1.
	\]
To capture the idea that the sequence of functions in \eqref{prob_form:f_def} is changing slowly, we assume
that there is a bound $\rho>0$ on the optimal solutions of the form:
\begin{equation}
\label{prob_form:opt_change_L2}
\|\bx_{n+1}^{*}-\bx_{n}^{*}\|  \leq \rho
\end{equation}
where $\|\bx\|$ is the Euclidean norm. Rather than using a Markov chain model or other Bayesian models for the changes in 
$\{\bx_{n}^{*}\}_{n=1}^{\infty}$, we only use the bound \eqref{prob_form:opt_change_L2} on $\rho$ in our analysis.

Given a sequence of slowly varying functions $\{f_{n}(\bx)\}_{n=1}^{\infty}$, we want to efficiently, sequentially minimize each of the functions to within a desired accuracy. We look at solving this problem by applying an optimization algorithm that uses $\numIter_{n}$ samples of $\bz_{n}$ such as 
stochastic gradient descent (SGD). 
%
%
We want to understand the trade-off between the solution accuracy and the complexity, represented by the number of samples $\numIter_{n}$. In effect, 
We want to understand how many samples $\{\bz_{n}(\iterIndex)\}_{\iterIndex=1}^{\numIter_{n}}$ are necessary to achieve a desired level of accuracy.

We introduce two different types of tracking criteria to characterize approximate minimizers of \eqref{prob_form:f_def}, denoted $\bx_{n}$ for each $n$. First, we define a \emph{mean tracking criterion}
\begin{equation}
\label{prob_form:L2_crit}
\mathbb{E}\left[ f_{n}(\bx_{n})\right] - f_{n}(\bx_{n}^{*})  \leq \epsilon
\end{equation}
and second, we define an \emph{in high probability (IHP) tracking criterion}
\begin{equation}
\label{prob_form:IHP_crit}
\mathbb{P}\left\{  f_{n}(\bx_{n}) - f_{n}(\bx_{n}^{*}) > t  \right\} \leq r
\end{equation}
with the expectation and probability taken over the samples $\{\bz_{n}(\iterIndex)\}_{\iterIndex=1}^{\numIter_{n}}$.

The remainder of this paper is organized as follows. In Section~\ref{probForm}, we introduce our problem. In Section~\ref{tracking_crit_rho_known}, we study the problem of selecting the number of samples $\numIter_{n}$ to achieve \cw{the mean criterion in \eqref{prob_form:L2_crit}}. We find a relationship between $\numIter_{n}$ and $\epsilon$ for the mean tracking criterion with the change in the minimizers, \cw{$\rho$ in \eqref{prob_form:opt_change_L2}}, known. This relationships allows us to select $\numIter_{n}$ in order to satisfy the mean criterion \cw{for sufficiently large $n$}. In Section~\ref{tracking_est_rho}, we introduce an estimate for the change in the minimizers, $\rho$, from \eqref{prob_form:opt_change_L2}. We provide theoretical guarantees that the introduced estimate eventually upper bounds the change in the minimizers. In Section~\ref{withRhoUnknown}, we combine the $\rho$ estimate of Section~\ref{tracking_est_rho} with the analysis of the case with $\rho$ known in Section~\ref{tracking_crit_rho_known} to provide rules to select $\numIter_{n}$ in order to meet the desired tracking criterion. We provide guarantees that for $n$ large enough, we meet our desired tracking criterion almost surely. Finally, we carry out simulation experiments to test our $\rho$ estimation and $\numIter_{n}$ selection rules.

\subsection{Related Work} 

There has been some work on similar problems, but general optimization theory tools to deal with time-varying optimization problems under \eqref{prob_form:opt_change_L2} have yet to be developed. 

In \cite{Zhu16}, the authors independently studied a time-varying optimization problem similar to ours. They imposed a bound on the change in the minimizers as in \eqref{prob_form:opt_change_L2} and studied SGD with a constant step size to develop finite-sample bounds suitable for guaranteeing that the mean criterion is satisfied when $\rho$ is known. The significant difference between our paper and the work in   \cite{Zhu16} is that we consider the case where the change in minimizers in unknown, and we develop a more general framework to handle any optimization algorithm that fits our assumptions.

Another relevant approach is {\em online optimization} in which a sequence of functions arrive, and in general no knowledge is available about the incoming functions other than that all the functions come from a specified class of functions, i.e., linear or convex functions with uniformly bounded gradients. 
Online optimization models do not include the notion of a desired tracking accuracy at each time instant such as \eqref{prob_form:L2_crit} and \eqref{prob_form:IHP_crit}. Instead, only bounds on the worst case performance of the best estimators are investigated through regret formulations~\cite{Cesa2006,Duchi2011,Duchi2009,Hazan2007,Bartlett2008,Shwartz2009,Shwartz2006,Shwartz2007,Xiao2010,Zinkevich2003}. 

For the problem of online optimization, the idea of controlling the variation of the sequence of functions has been studied in~\cite{RakhlinSridharan2012} and~\cite{Yang2012}. In~\cite{Yang2012}, regret is minimized subject to a bound, say $G_{b}$, on the total variation of the gradients over a time interval $T$ of interest, i.e.,
\begin{equation}
\label{intro:yang_var}
\sum_{n=2}^{T} \max_{\bx \in \xSp} \| \nabla f_{n}(\bx) -  \nabla f_{n-1}(\bx) \|^{2} \leq G_{b}.
\end{equation}
If all the functions $\{f_{n}(x)\}$ are strongly convex with the same parameter $m$, then by the optimality conditions (see Theorem 2F.10 in~\cite{Dontchev2009}) relation~\eqref{intro:yang_var} implies that
\[
\sum_{n=2}^{T} \|\bx_{n}^{*} - \bx_{n-1}^{*}\|^{2} \leq \tilde{G}_{b}
\]
with $\tilde{G}_{b}$ a function of $G_{b}$. Therefore, the work in~\cite{Yang2012} can be seen as studying the regret while controlling the total variation in the optimal solutions over $T$ time instants. In contrast, we control the variation of the optimal solutions at each time instant with \eqref{prob_form:opt_change_L2} and then seek to maintain a tracking criterion such as \eqref{prob_form:L2_crit} and \eqref{prob_form:IHP_crit} at each time instant.

Additionally, there is other work that has some of the ingredients of our proposed problem formulation. In~\cite{TYS2010}, a sequence of quadratic functions is considered and treated within the domain of estimation theory; however, the authors only examine the Least Mean Squares (LMS) algorithm (corresponding to $\numIter_n=1$ for all $n$). 
The work in~\cite{YO2005,SYO2006} considers a sequence $\{f_n\}$ of convex objective functions converging to some limit function $f$, where all the functions $f_n$ have the same set of possible minima. However, aside from considering time-varying objective functions, these works have nothing else in common with the work described here. There has also been work 
in~\cite{Kushner1994} considering the limit as the rate of change of the functions goes to zero and for the least means squares (LMS) algorithm in~\cite{Solo1995}. The results in~\cite{Kushner1994} and~\cite{Solo1995} both require a Bayesian model for the changes in the function sequence, which we do not require. 

If we have a quadratic loss centered at $\bx_{n}^{*}$ and a linear state space evolution for the optimal solution $\bx_{n}^{*}$, then we could apply the Kalman filter \cite{Sayed2008}. If the function we seek to optimize is non-linear, another approach we can consider under a Bayesian framework is particle filtering \cite{Doucet2009}. For particle filtering, it is harder to provide exact guarantees on performance similar to those given in \eqref{prob_form:L2_crit} and \eqref{prob_form:IHP_crit}.

To conclude, there are no existing approaches within optimization theory or estimation theory that allow us to solve a sequence of time-varying problems, subject to abiding to a pre-specified tracking error criterion such as \eqref{prob_form:L2_crit} or \eqref{prob_form:IHP_crit} under only \eqref{prob_form:opt_change_L2}. In this work, we fill in this gap and provide methods to solve such problems.

\section{Problem Formulation}
\label{probForm}

\subsection{Assumptions}
We make several assumptions to proceed. First, let $\xSp$ be closed and convex with $\text{diam}(\xSp) < + \infty$. Define the $\sigma$-algebra
\begin{equation}
\label{withRhoUnknown:FSigAlg}
\mathcal{F}_{i} \triangleq \sigma\left(  \left\{ \bz_{j}(\iterIndex): \;  j = 1, \ldots, i; \; \iterIndex=1, \ldots, \numIter_{j} \right\} \right)
\end{equation}
which is the smallest $\sigma$-algebra such that the random variables in the set $\left\{ \bz_{j}(\iterIndex): \;  j = 1, \ldots, i; \;\iterIndex=1, \ldots, \numIter_{j} \right\}$ are measurable. By convention $\mathcal{F}_{0}$ is the trivial $\sigma$-algebra.

We suppose that the following conditions hold:
\begin{description}
	\item\namedlabel{probState:assump1}{A.1}
	For each $n$, $f_{n}(\bx)$ is twice continuously differentiable with respect to $\bx$.
	\item\namedlabel{probState:assump2}{A.2} 
	For each $n$, $f_{n}(\bx)$ is strongly convex with a parameter $m>0$, i.e.,
	\begin{equation}
	\label{prob_form:strong_convex_cond}
	f_{n}(\tilde{\bx}) \geq f_{n}(\bx) + \inprod{\nabla_{\bx}f_{n}(\bx)}{\tilde{\bx} - \bx} + \frac{1}{2}m \| \tilde{\bx} - \bx \|^{2}.
	\end{equation}
	where $\inprod{\bx}{\tilde{\bx}}$ is the Euclidean inner product between $\bx$ and $\tilde{\bx}$.
	\item\namedlabel{probState:assump3}{A.3}  For each $n$, we can draw stochastic gradients $\sgrad{\bx}{\bz_{n}}{n}$ such that the following holds:
	\begin{equation}
\label{prob_form:stochGradDef}
\mathbb{E} [\sgrad{\bx}{\bz_{n}}{n} ] = \nabla f_{n}(\bx)
\qquad\hbox{for all\ } \bx.
\end{equation}
	
	\item\namedlabel{probState:assump4}{A.4} Given an optimization algorithm that generates an approximate minimizer $\bx_{n}$ 	using $\numIter_{n}$ samples $\{\bz_{n}(\iterIndex)\}_{\iterIndex=1}^{\numIter_{n}}$, there exists a function $b(d_{0},\numIter_{n})$ such that the following conditions hold:
	\begin{enumerate}
		\item If $\numIter_{n}$ and $d_{0}$ are both $\mathcal{F}_{n-1}$-measurable random variables, it holds that
		\iftoggle{useTwoColumn}{
		\begin{align}
		&\| \bx_{n-1} - \bx_{n}^{*}\|^{2} \leq d_{0}^{2} \nonumber \\
		\label{probState:bBoundRand}
		&\;\;\;\;\;\;\; \Rightarrow \mathbb{E}[f_{n}(\bx_{n}) \;|\; \mathcal{F}_{n-1}] - f_{n}(\bx_{n}^{*}) \leq b(d_{0},\numIter_{n}).
		\end{align}
		}{
		\begin{equation}
		\label{probState:bBoundRand}
		\| \bx_{n-1} - \bx_{n}^{*}\|^{2} \leq d_{0}^{2}  \;\;\Rightarrow\;\; \mathbb{E}[f_{n}(\bx_{n}) \;|\; \mathcal{F}_{n-1}] - f_{n}(\bx_{n}^{*}) \leq b(d_{0},\numIter_{n}).
		\end{equation}
		}
		\item If $\tilde{\numIter}_{n}$ and $\gamma$ are constants, it holds that
		\iftoggle{useTwoColumn}{
		\begin{align}
		&\mathbb{E}\| \bx_{n-1} - \bx_{n}^{*}\|^{2} \leq \gamma^{2} \nonumber \\
		\label{probState:bBoundDeterm}
		&\;\;\;\;\;\;\; \Rightarrow \mathbb{E}[f_{n}(\bx_{n})] - f_{n}(\bx_{n}^{*}) \leq b(\gamma,\tilde{\numIter}_{n}).
		\end{align}
		}{
		\begin{equation}
		\label{probState:bBoundDeterm}
		\| \bx_{n-1} - \bx_{n}^{*}\|^{2} \leq \gamma^{2} \;\;\Rightarrow\;\; \mathbb{E}[f_{n}(\bx_{n}) \;|\; \mathcal{F}_{n-1}] - f_{n}(\bx_{n}^{*}) \leq b(\gamma,\tilde{\numIter}_{n}).
		\end{equation}
		}
		\item The bound $b(d_{0},\numIter_{n})$ is non-decreasing in $d_{0}$ and non-increasing in $\numIter_{n}$.
	\end{enumerate} 
	\item\namedlabel{probState:assump5}{A.5} There exist constants $A,B \geq 0$ such that
	\begin{equation}
	\label{prob_form:L2_ngrad}
	\mathbb{E}\left[ \| \sgrad{\bx}{\bz_{n}}{n} \|^{2} \;|\; \mathcal{F}_{n-1} \right] \leq A + B \|\bx - \bx_{n}^{*}\|^{2}
	\end{equation}
	\item\namedlabel{probState:assump6}{A.6} Initial approximate minimizers $\bx_{1}$ and $\bx_{2}$ satisfy 
	\[
	f_{i}(\bx_{i}) - f_{i}(\bx_{i}^{*}) \leq \epsilon_{i} \;\;\;\;\; i=1,2
	\]
	with $\epsilon_{1}$ and $\epsilon_{2}$ known.
\end{description}

\cw{For Assumption~\ref{probState:assump2}, an example of a strongly convex function is a quadratic $f(\bx)=\frac{1}{2}\bx^{\top}\bm{Q}\bx$ where the smallest eigenvalue of $\bm{Q}$ satisfies $\lambda_{\textrm{min}}(\bm{Q}) \geq m$.}

For Assumption~\ref{probState:assump4}, we consider
SGD
\begin{eqnarray}
\label{prob_form:sa_alg}
\bx_{n}(\iterIndex) \!\!\! &=& \!\!\! \Pi_{\xSp}[\bx_{n}(\iterIndex - 1) - \mu_{n}(\iterIndex) \sgrad{\bx}{\bz_{n}(\iterIndex)}{n}]  \\
\bx_{n}(0) \!\!\! &\triangleq& \!\!\! \bx_{n-1} \nonumber
\end{eqnarray}
with \cw{$\iterIndex=1,\ldots,\numIter_{n}$, and $\Pi_{\xSp}$ denoting projection on to the set $\xSp$.}  We choose $\bx_{n}$ as a convex combination of the iterates $\{\bx_{n}(\iterIndex)\}_{\iterIndex=0}^{\numIter_{n}}$ generated by SGD
\[
\bx_{n} = \sum_{\iterIndex = 0}^{\numIter_{n}} \lambda_{n}(\iterIndex) \bx_{n}(\iterIndex).
\]
One simple choice is setting $\bx_{n} = \bx_{n}(\numIter_{n})$, which corresponds to setting $\lambda_{n}(\numIter_{n}) = 1$ and $\lambda_{n}(0) = \cdots = \lambda_{n}(\numIter_{n}-1) = 0$.

\cw{Section~\ref{bBounds} discusses several applicable bounds $b(d_{0},\numIter)$ for SGD and choices of convex combinations $\{\lambda_{n}(\iterIndex)\}$. We need \eqref{probState:bBoundRand} to handle the case when $\rho$ must be estimated, and \eqref{probState:bBoundDeterm} of Assumption~\ref{probState:assump4} to handle the case when $\rho$ is known. In fact, if the bound $b(d_{0},\numIter_{n})$ factors as
\begin{equation}
\label{probState:factorAssump}
b(d_{0},\numIter_{n}) = \alpha(\numIter_{n})d_{0}^{2} + \beta(\numIter_{n})
\end{equation}
then \eqref{probState:bBoundRand} implies \eqref{probState:bBoundDeterm} as well. To see this set $d_{0}=\| \bx_{n-1} - \bx_{n}^{*}\|$ and suppose that $\mathbb{E}[d_{0}^{2}] \leq \gamma^{2}$. The bound in \eqref{probState:bBoundRand} implies that 
\[
\mathbb{E}[f_{n}(\bx_{n}) ] - f_{n}(\bx_{n}^{*}) \leq \mathbb{E}_{d_{0},\numIter_{n}}[b(d_{0},\tilde{\numIter}_{n})].
\]
Applying \eqref{probState:factorAssump} yields
\begin{eqnarray}
\mathbb{E}_{d_{0},\numIter_{n}}[b(d_{0},\tilde{\numIter}_{n})] &=& \mathbb{E}[\alpha(\tilde{\numIter}_{n}) d_{0}^{2} + \beta(\tilde{\numIter}_{n})] \nonumber \\
&=& \alpha(\tilde{\numIter}_{n}) \mathbb{E}[d_{0}^{2}] + \beta(\tilde{\numIter}_{n}) \nonumber \\
&=& b\left( \sqrt{\mathbb{E}[d_{0}^{2}]}, \tilde{\numIter}_{n}  \right) \nonumber \\
&\leq& b(\gamma, \tilde{\numIter}_{n}). \nonumber
\end{eqnarray}
} 

In practice, we may not know the parameters such as the strong convexity parameter $m$ from Assumption~\ref{probState:assump2} and the gradient parameters $A$ and $B$ from Assumption~\ref{probState:assump5}. \cw{Section~\ref{parameterEstimation} introduces several techniques to estimate these parameters using the stochastic gradients in Assumption~\ref{probState:assump3}.}

In our assumptions, we condition on the $\sigma$-algebra $\mathcal{F}_{n-1}$, since this captures all of the information available at the beginning of time $n$. In later sections, we will select $\numIter_{n}$ as a function of the samples $\{\bz_{i}(\iterIndex)\}_{\iterIndex=1}^{\numIter_{i}}$ for $i=1,\ldots,n-1$. This implies that $\numIter_{n}$ is $\mathcal{F}_{n-1}$ measurable. In this case, where $\numIter_{n}$ is itself a random variable, Assumption~\ref{probState:assump4} is crucial to our analysis.

Finally, for Assumption ~\ref{probState:assump6}, we generally must select $\numIter_{1}$ and $\numIter_{2}$ blindly in the sense that we have no information about $\rho$ defined in \eqref{prob_form:opt_change_L2}. We can only make a choice such as
\[
\numIter_{i} = \min\left\{ \numIter \geq 1 \;\big|\; b\left( \text{diam}(\xSp) , \numIter  \right) \leq \epsilon \right\} \;\;\;\; i=1,2
\]
or fixed initial choices for $\numIter_{1}$ and $\numIter_{2}$. Regardless of our choice of $\numIter_{1}$ and $\numIter_{2}$, we can set $\epsilon_{i} \triangleq b\left( \text{diam}(\xSp) , \numIter  \right)$ for $i=1,2$. In order to have $\epsilon_{i} \leq \epsilon$ for $i=1,2$, we may need to draw significantly more samples up front to find points $\bx_{1}$ and $\bx$ due to using $\text{diam}(\xSp)$.

\subsection{Constructing a Bound On the Change in Minimizers}

We look at the justification behind our choice of controlling the change in functions through the minimizers $\bx_{n}^{*}$ by showing that several other reasonable ways to control how the functions change can be reduced to a bound on the change in minimizers. In Section~\ref{tracking_crit_rho_known}, we show that bounds on the change in the minimizer can be used to select the number of samples $\numIter_{n}$.

\subsubsection{Change in $f$}
Suppose that we instead bound the change in the optimal function values, in the following manner:
\[
f_{n}(\bx_{n-1}^{*}) - f_{n}(\bx_{n}^{*}) \leq \tilde{\rho}.
\]
This bounds the loss incurred as a result of using the minimizer of the previous function $f_{n-1}$ as an approximate minimizer of the current function $f_{n}$. 
By the strong convexity Assumption~\ref{probState:assump2}, it holds that
\begin{equation}
\| \bx_{n}^{*} - \bx_{n-1}^{*} \| \leq \sqrt{\frac{2}{m} \left( f_{n}(\bx_{n-1}^{*}) - f_{n}(\bx_{n}^{*}) \right)} \leq \sqrt{\frac{2}{m} \tilde{\rho}}. \nonumber
\end{equation}
Therefore, a bound on the optimal function values can be translated into a bound on the change in the minimizers.

\subsubsection{Change in Distribution}
For machine learning problems, we can generally write our functions as an expectation of a loss function $\lossFunc(\bx,\bz)$, i.e., 
\[
f_{n}(\bx) = \mathbb{E}_{\bz_{n} \sim p_{n}}[\lossFunc(\bx,\bz_{n})].
\]
Therefore, the source of change in this problem is the model distributions $p_{n}$. 
We can control the change by making an assumption on how $p_{n}$ changes through an appropriate probability metric or pseudo-metric. Given a class of functions $\mathcal{F}$ mapping from $\mathscr{Z} \to \mathbb{R}$, an integral probability metric\cite{Mueller1997} between two distributions $p$ and $q$ on $\mathcal{Z}$ is defined as
\[
\ipm{\mathcal{F}}{p}{q} \triangleq \sup_{h \in \mathcal{F}} | \mathbb{E}_{\bz \sim p}[h(\bz)] - \mathbb{E}_{\tilde{\bz} \sim q}[h(\tilde{\bz})]|.
\]
The following lemma shows that under an inclusion condition on the loss function $\lossFunc(\bx,\bz)$, the
integral probability metric bounds can lead to bounds on the change in minimizers.
\begin{lem} 
	If the class $\left\{ \lossFunc(\bx,\cdot) \;|\; \bx \in \xSp  \right\} \subset \mathcal{F}$
	of loss functions is such that \\$\ipm{\mathcal{F}}{p_{n}}{p_{n-1}} \leq \tilde{\rho}$ for all $n\ge1$, then it holds that
	\[
	\| \bx_{n}^{*} - \bx_{n-1}^{*} \| \leq \sqrt{ \frac{2}{m} \tilde{\rho}}\qquad\hbox{for all }n\ge1.
	\]
\end{lem}
\begin{proof}
	Applying the strong convexity Assumption~\ref{probState:assump2} to $f_{n}(\bx)$ and $f_{n-1}(\bx)$, for the solutions
		$\bx_{n}^{*}$ and $\bx_{n-1}^{*}$, we obtain
	\begin{eqnarray}
	f_{n}(\bx_{n-1}^{*}) &\geq& f_{n}(\bx_{n}^{*}) + \frac{1}{2} m \| \bx_{n}^{*} - \bx_{n-1}^{*}\|^{2} \nonumber \\
	f_{n-1}(\bx_{n}^{*}) &\geq& f_{n-1}(\bx_{n-1}^{*}) + \frac{1}{2} m \| \bx_{n}^{*} - \bx_{n-1}^{*}\|^{2}. \nonumber
	\end{eqnarray}
	By adding these two inequalities and rearranging, it holds that
	\iftoggle{useTwoColumn}{
	\begin{align}
	m& \| \bx_{n}^{*} - \bx_{n-1}^{*} \|^{2} \nonumber \\
	&\leq \left( f_{n}(\bx_{n-1}^{*}) - f_{n}(\bx_{n}^{*}) \right) + \left( f_{n-1}(\bx_{n}^{*}) - f_{n-1}(\bx_{n-1}^{*}) \right) \nonumber \\
	&= \left( f_{n}(\bx_{n-1}^{*}) - f_{n-1}(\bx_{n-1}^{*}) \right) + \left( f_{n-1}(\bx_{n}^{*}) - f_{n}(\bx_{n}^{*}) \right). \nonumber
	\end{align}
	}{
	\begin{eqnarray}
	m \| \bx_{n}^{*} - \bx_{n-1}^{*} \|^{2} &\leq& \left( f_{n}(\bx_{n-1}^{*}) - f_{n}(\bx_{n}^{*}) \right) + \left( f_{n-1}(\bx_{n}^{*}) - f_{n-1}(\bx_{n-1}^{*}) \right) \nonumber \\
	&=& \left( f_{n}(\bx_{n-1}^{*}) - f_{n-1}(\bx_{n-1}^{*}) \right) + \left( f_{n-1}(\bx_{n}^{*}) - f_{n}(\bx_{n}^{*}) \right). \nonumber
	\end{eqnarray}
	}
Now, examine the term $f_{n}(\bx_{n-1}^{*}) - f_{n-1}(\bx_{n-1}^{*})$.  
	We have
	\iftoggle{useTwoColumn}{
	\begin{align}
	f_{n}&(\bx_{n-1}^{*}) - f_{n-1}(\bx_{n-1}^{*}) \nonumber \\
	&\leq | \mathbb{E}_{\bz_{n} \sim p_{n}} \left[  \lossFunc(\bx_{n-1}^{*},\bz_{n}) \right] - \mathbb{E}_{\bz_{n-1} \sim p_{n-1}} \left[  \lossFunc(\bx_{n-1}^{*},\bz_{n-1}) \right] | \nonumber \\
	&\leq \sup_{\ell \in \mathcal{F}} | \mathbb{E}_{\bz_{n} \sim p_{n}} \left[ \lossFunc(\bx_{n-1}^{*},\bz_{n}) \right] - \mathbb{E}_{\bz_{n-1} \sim p_{n-1}} \left[  \lossFunc(\bx_{n-1}^{*},\bz_{n-1}) \right] | \nonumber \\
	&= \ipm{\mathcal{F}}{p_{n}}{p_{n-1}}. \nonumber
	\end{align}
	}{
	\begin{eqnarray}
	f_{n}(\bx_{n-1}^{*}) - f_{n-1}(\bx_{n-1}^{*}) &\leq& | \mathbb{E}_{\bz_{n} \sim p_{n}} \left[  \lossFunc(\bx_{n-1}^{*},\bz_{n}) \right] - \mathbb{E}_{\bz_{n-1} \sim p_{n-1}} \left[  \lossFunc(\bx_{n-1}^{*},\bz_{n-1}) \right] | \nonumber \\
	&\leq& \sup_{f\ell \in \mathcal{F}} | \mathbb{E}_{\bz_{n} \sim p_{n}} \left[ \lossFunc(\bx_{n-1}^{*},\bz_{n}) \right] - \mathbb{E}_{\bz_{n-1} \sim p_{n-1}} \left[  \lossFunc(\bx_{n-1}^{*},\bz_{n-1}) \right] | \nonumber \\
	&=& \ipm{\mathcal{F}}{p_{n}}{p_{n-1}}. \nonumber
	\end{eqnarray}
	}
Similarly, we can see that the same estimate holds for the term $f_{n-1}(\bx_{n}^{*}) - f_{n}(\bx_{n}^{*}) $.
Therefore, it holds that
\begin{equation*}
\| \bx_{n}^{*} - \bx_{n-1}^{*} \| \leq \sqrt{ \frac{2}{m} \ipm{\mathcal{F}}{p_{n}}{p_{n-1}}} \leq \sqrt{ \frac{2}{m} \tilde{\rho}}.
\end{equation*}
\end{proof}
Thus, we see that we can translate a bound on the change in distributions through an integral probability metric
to a bound on the change in minimizers.

\subsubsection{Parameterized Functions}
Finally, we examine the case in which the functions $\{f_{n}(\bx)\}$ come from a parameterized class of functions $f(\bx,\theta)$, i.e.,
\[
f_{n}(\bx) = f(\bx,\btheta_{n}).
\]
Furthermore, we assume that the parameters themselves change slowly
\[
\| \btheta_{n} - \btheta_{n-1} \| \leq \delta.
\]
With appropriate assumptions on the function $f(\bx,\btheta)$, we can apply the implicit function theorem \cite{Rudin1964} to yield a bound of the form
\[
\| \bx_{n+1}^{*} - \bx_{n}^{*}\| \leq G \| \btheta_{n+1} - \btheta_{n}\| \leq G \delta
\]
for an appropriately chosen $G$.

\section{Tracking Analysis with Change in Minimizers Known}
\label{tracking_crit_rho_known}

In this section we combine the bound $b(d_{0},\numIter)$ in assumption~\ref{probState:assump4} with our model for the changes in functions in \cref{prob_form:opt_change_L2} to choose the number of stochastic gradients $\numIter$ needed to achieve desired mean criterion $\epsilon$ in \cref{prob_form:L2_crit}. The IHP criterion in \cref{prob_form:IHP_crit} is analyzed\iftoggle{useArxiv}{}{in an extended version of this paper \cite{Wilson2016a}} in \Cref{ihp_track_error_anaysis}. In this section, we assume that $\rho$ is known. In \Cref{withRhoUnknown}, we will consider the case when $\rho$ is unknown.

\subsection{Mean Criterion Analysis}
\label{tracking_crit_rho_known:meanAnalysis}

We show how to choose $\numIter$ to achieve a target mean criterion $\epsilon$ for all $n$. The idea behind the analysis is to proceed by induction using Assumption~\ref{probState:assump6} as a base case. Suppose that 
\[
\mathbb{E}[f_{n-1}(\bx_{n-1})] - f_{n-1}(\bx_{n-1}^{*}) \leq \epsilon.
\]
Denote the distance from the initial point $\bx_{n-1}$ to the minimizer $\bx_{n}^{*}$ by $d_{n}(0)$, i.e.,
\begin{equation}
\label{withRhoKnown:dn0Def}
d_{n}^{2}(0) = \| \bx_{n-1} - \bx_{n}^{*}\|^{2} ~\text{(almost surely)}.
\end{equation}
To bound $\mathbb{E}[d_{n}^{2}(0)]$ we first use the triangle inequality and $\rho$ from \eqref{prob_form:opt_change_L2} to get
\begin{eqnarray}
\sqrt{\mathbb{E}[d_{n}^{2}(0)]}  &\leq&  \|\bx_{n-1} - \bx_{n-1}^{*}\|_{L_{2}} + \|\bx_{n-1}^{*} - \bx_{n}^{*}\|_{L_{2}} \nonumber \\
&\leq&  \|\bx_{n-1} - \bx_{n-1}^{*}\|_{L_{2}} + \rho. \nonumber
\end{eqnarray}
where $\|\bx\|_{L_{2}} = \sqrt{\mathbb{E}\|\bx\|^{2}}$ is the $L_{2}$-norm.
By the strong convexity Assumption~\ref{probState:assump2}, we have
\begin{equation}
\label{analysis:strong_L2_to_f}
\frac{m}{2}\|\bx_{n-1} - \bx_{n-1}^{*}\|^{2} \leq f_{n}(\bx_{n-1}) - f_{n}(\bx_{n-1}^{*})
\end{equation}
yielding
\begin{equation}
\mathbb{E}\|\bx_{n-1} - \bx_{n-1}^{*}\|^{2} \leq \frac{2}{m} \left( \mathbb{E}[f_{n}(\bx_{n-1})] - f_{n}(\bx_{n-1}^{*})\right) \leq \frac{2 \epsilon}{m}. \nonumber
\end{equation}
Putting everything together we have
\begin{equation}
\label{sol_sens:L2_d0}
\mathbb{E}[d_{n}^{2}(0)] \leq \left(  \sqrt{\frac{2\epsilon}{m}} + \rho \right)^{2}
\end{equation}
according to Assumption~\ref{probState:assump4}.

Therefore, we have the bound
\[
\mathbb{E}[f_{n}(\bx_{n})] - f_{n}(\bx_{n}^{*}) \leq b\left(   \sqrt{\frac{2\epsilon}{m}} + \rho,\numIter \right)
\]
and we can set
\begin{equation}
\label{K_with_rho_known}
\numIter^{*} = \min\left\{ \numIter \geq 1 \;\Bigg|\; b\left(   \sqrt{\frac{2\epsilon}{m}} + \rho,\numIter \right) \leq \epsilon \right\}
\end{equation}
to ensure that
\[
\mathbb{E}[f_{n}(\bx_{n})] - f_{n}(\bx_{n}^{*}) \leq \epsilon \;\;\;\; \forall n \geq 1.
\]

\subsection{IHP Tracking Error Analysis}
\label{ihp_track_error_anaysis}

For the IHP criterion, we assume that assumptions \ref{probState:assump1}-\ref{probState:assump6} hold. We seek an upper bound $r(t,\numIter)$ such that
\begin{equation}
\label{ihp_f_ub}
\mathbb{P} \left\{ f_{n}(\bx_{n}) - f_{n}(\bx_{n}^{*}) > t \right\} \leq r(t,\numIter) \;\;\; \forall n \geq 1.
\end{equation}
Using the mean criterion bounds of the previous section, we know that for all $n$
\[
\mathbb{E}[f_{n}(\bx_{n})] - f_{n}(\bx_{n}^{*}) \leq \epsilon.
\]
Then by Markov's inequality, it holds that
\begin{equation}
\label{withRhoKnown:ihpMarkovBound}
\mathbb{P} \left\{ f_{n}(\bx_{n}) - f_{n}(\bx_{n}^{*}) > t \right\} \leq \frac{\epsilon}{t}.
\end{equation}
Although this bound always holds, we look at a way to tighten this bound. As before, we proceed by induction. As a base case, we can set
\[
\mathbb{P}\left\{ f_{1}(\bx_{1}) - f_{1}(\bx_{1}^{*}) > t \right\} \leq \frac{\epsilon}{t}.
\]
Now, suppose that
\[
\mathbb{P}\left\{ f_{n-1}(\bx_{n-1}) - f_{n-1}(\bx_{n-1}^{*}) > t \right\} \leq r_{n-1}(t)
\]
and we want to construct a bound $r_{n}(t)$ on $\mathbb{P}\left\{ f_{n}(\bx_{n}) - f_{n}(\bx_{n}^{*}) > t \right\}$. We proceed by conditioning on $\{d_{n}(0) \leq \delta\}$ and $\{d_{n}(0) > \delta\}$ with $d_{n}(0)$ defined in \cref{withRhoKnown:dn0Def} using the law of total probability:
\iftoggle{useTwoColumn}{
\begin{align}
\mathbb{P}& \left\{ f_{n}(\bx_{n}) - f_{n}(\bx_{n}^{*}) > t \right\} \nonumber \\
&= \mathbb{P} \left\{ f_{n}(\bx_{n}) - f_{n}(\bx_{n}^{*}) > t \;|\; d_{n}(0) \leq \delta \right\} \mathbb{P}\left\{ d_{n}(0) \leq \delta \right\} \nonumber \\
&\;\;\;\;\;\;\;\;\;\; + \mathbb{P} \left\{ f_{n}(\bx_{n}) - f_{n}(\bx_{n}^{*}) > t \;|\; d_{n}(0) > \delta \right\} \mathbb{P}\left\{ d_{n}(0) > \delta \right\}. \nonumber
\end{align}
}{
\begin{eqnarray}
\mathbb{P} \left\{ f_{n}(\bx_{n}) - f_{n}(\bx_{n}^{*}) > t \right\} &=& \mathbb{P} \left\{ f_{n}(\bx_{n}) - f_{n}(\bx_{n}^{*}) > t \;|\; d_{n}(0) \leq \delta \right\} \mathbb{P}\left\{ d_{n}(0) \leq \delta \right\} \nonumber \\
&&\;\;\;\;\;\;\;\;\;\; + \mathbb{P} \left\{ f_{n}(\bx_{n}) - f_{n}(\bx_{n}^{*}) > t \;|\; d_{n}(0) > \delta \right\} \mathbb{P}\left\{ d_{n}(0) > \delta \right\}. \nonumber
\end{eqnarray}
}
For the first term, it holds that
\begin{equation}
\label{withRhoKnown:ihpDn0Bounded}
\mathbb{P} \left\{ f_{n}(\bx_{n}) - f_{n}(\bx_{n}^{*}) > t \;|\; d_{n}(0) \leq \delta \right\} \leq \frac{1}{t} b(\delta,\numIter) \triangleq \psi(t,\delta)
\end{equation}
and
\[
\mathbb{P}\left\{ d_{n}(0) \leq \delta \right\} \leq 1.
\]
For the second term, it holds that 
\[
\mathbb{P} \left\{ f_{n}(\bx_{n}) - f_{n}(\bx_{n}^{*}) > t \;|\; d_{n}(0) > \delta \right\} \leq \psi(t,\text{diam}(\xSp))
\]
and
\iftoggle{useTwoColumn}{
\begin{align}
\mathbb{P}&\left\{ d_{n}(0) > \delta \right\} \nonumber \\
&= \mathbb{P}\left\{ \| \bx_{n-1} - \bx_{n}^{*}\|_{2}^{2} > \delta \right\}  \nonumber \\
&= \mathbb{P}\left\{ \| \bx_{n-1} - \bx_{n}^{*}\|_{2} > \sqrt{\delta} \right\} \nonumber \\
&\leq \mathbb{P}\left\{ \| \bx_{n-1} - \bx_{n-1}^{*}\|_{2} + \rho > \sqrt{\delta} \right\} \nonumber \\
&\leq \mathbb{P}\left\{ \| \bx_{n-1} - \bx_{n-1}^{*}\|_{2} > \left( \sqrt{\delta} - \rho \right)_{+} \right\} \nonumber \\
&\leq \mathbb{P}\left\{ \sqrt{f_{n-1}(\bx_{n-1}) - f_{n-1}(\bx_{n-1}^{*})} > \sqrt{\frac{2}{m}}\left( \sqrt{\delta} - \rho \right)_{+} \right\} \nonumber \\
&\leq r_{n-1}\left( \frac{2}{m}\left( \sqrt{\delta} - \rho \right)_{+}^{2} \right) \nonumber 
\end{align}
}{
\begin{eqnarray}
\mathbb{P}\left\{ d_{n}(0) > \delta \right\} &=& \mathbb{P}\left\{ \| \bx_{n-1} - \bx_{n}^{*}\|_{2}^{2} > \delta \right\}  \nonumber \\
&=& \mathbb{P}\left\{ \| \bx_{n-1} - \bx_{n}^{*}\|_{2} > \sqrt{\delta} \right\} \nonumber \\
&\leq& \mathbb{P}\left\{ \| \bx_{n-1} - \bx_{n-1}^{*}\|_{2} + \rho > \sqrt{\delta} \right\} \nonumber \\
&\leq& \mathbb{P}\left\{ \| \bx_{n-1} - \bx_{n-1}^{*}\|_{2} > \left( \sqrt{\delta} - \rho \right)_{+} \right\} \nonumber \\
&\leq& \mathbb{P}\left\{ \sqrt{f_{n-1}(\bx_{n-1}) - f_{n-1}(\bx_{n-1}^{*})} > \sqrt{\frac{2}{m}}\left( \sqrt{\delta} - \rho \right)_{+} \right\} \nonumber \\
&\leq& r_{n-1}\left( \frac{2}{m}\left( \sqrt{\delta} - \rho \right)_{+}^{2} \right) \nonumber 
\end{eqnarray}
}
where $(x)_{+} = \max\{x,0\}$. Combining these bounds yields an overall bound
\iftoggle{useTwoColumn}{
\begin{align}
\mathbb{P}& \left\{ f_{n}(\bx_{n}) - f_{n}(\bx_{n}^{*}) > t \right\} \nonumber \\
&\qquad \leq \psi(t,\delta) + \psi(t,\text{diam}(\xSp)) r_{n-1}\left( \frac{2}{m}\left( \sqrt{\delta} - \rho \right)_{+}^{2}  \right). \nonumber
\end{align}
}{
\begin{eqnarray}
\mathbb{P} \left\{ f_{n}(\bx_{n}) - f_{n}(\bx_{n}^{*}) > t \right\} &\leq& \psi(t,\delta) + \psi(t,\text{diam}(\xSp)) r_{n-1}\left( \frac{2}{m}\left( \sqrt{\delta} - \rho \right)_{+}^{2}  \right). \nonumber
\end{eqnarray}
}

We can optimize this bound over $\delta$ to yield the bound
\iftoggle{useTwoColumn}{
\begin{align}
\label{withRhoKnown:optExactIHPBound}
\mathbb{P} &\left\{ f_{n}(\bx_{n}) - f_{n}(\bx_{n}^{*}) > t \right\} \nonumber \\
&\qquad \leq \inf_{0 < \delta \leq \text{diam}(\xSp)} \left\{ \psi(t,\delta) \right. \nonumber \\
&\left. \qquad \qquad \qquad + \psi(t,\text{diam}(\xSp)) r_{n-1}\left( \frac{2}{m}\left( \sqrt{\delta} - \rho \right)_{+}^{2}  \right) \right\}.
\end{align}
}{
\begin{align}
\label{withRhoKnown:optExactIHPBound}
\mathbb{P} &\left\{ f_{n}(\bx_{n}) - f_{n}(\bx_{n}^{*}) > t \right\} \nonumber \\
&\;\;\;\;\;\;\;\;\;\; \leq \inf_{0 < \delta \leq \text{diam}(\xSp)} \left\{ \psi(t,\delta) + \psi(t,\text{diam}(\xSp)) r_{n-1}\left( \frac{2}{m}\left( \sqrt{\delta} - \rho \right)_{+}^{2}  \right) \right\}.
\end{align}
}
The quantity $\psi(t,\delta)$ defined in \cref{withRhoKnown:ihpDn0Bounded} can be replaced by any bound that also satisfies the inequality in \cref{withRhoKnown:ihpDn0Bounded}. Therefore, we can set
\iftoggle{useTwoColumn}{
\begin{align}
r_{n}(t) &= \inf_{0 < \delta \leq \text{diam}(\xSp)} \left\{ \psi(t,\delta) \right. \nonumber \\
\label{withRhoKnown:ihpExactBound}
&\left. \qquad \qquad + \psi(t,\text{diam}(\xSp)) r\left( \frac{2}{m}\left( \sqrt{\delta} - \rho \right)_{+}^{2}  \right) \right\} 
\end{align}
}{
\begin{align}
\label{withRhoKnown:ihpExactBound}
r_{n}(t) = \inf_{0 < \delta \leq \text{diam}(\xSp)} \left\{ \psi(t,\delta) + \psi(t,\text{diam}(\xSp)) r\left( \frac{2}{m}\left( \sqrt{\delta} - \rho \right)_{+}^{2}  \right) \right\} 
\end{align}
}
resulting in $\mathbb{P}\left\{ f_{n}(\bx_{n}) - f_{n}(\bx_{n}^{*}) > t \right\} \leq r_{n}(t)$. In \Cref{ihpFiniteBound}, we develop an approach to compute the IHP bound. With this approach, we can set $\numIter_{n}$ to achieve a desired IHP criterion $\epsilon$ and $r$.

\subsubsection{Bound at a Finite Number of Points}
\label{ihpFiniteBound}

The bound of the preceding section is exact but difficult to compute. In this section, we introduce a computationally simpler bound. Computing the entire sequence of functions $r_{n}(t)$ is generally difficult, so we look at bounding
\[
\mathbb{P}\left\{ f_{n}(\bx_{n}) - f_{n}(\bx_{n}^{*}) > t \right\}
\]
at a finite number of points $t^{(1)},\ldots,t^{(N)}$ ordered in increasing order. We want to compute bound\\ $r_{n}(1),\ldots,r_{n}(N)$ such that
\[
\mathbb{P}\left\{ f_{n}(\bx_{n}) - f_{n}(\bx_{n}^{*}) > t^{(i)} \right\} \leq r_{n}(i) \;\;\;\; i=1,\ldots,N.
\] 
We define an initial bound
\[
r_{1}(i) = \frac{\epsilon}{t^{(i)}}.
\]
Suppose that
\[
\mathbb{P}\left\{ f_{n-1}(\bx_{n-1}) - f_{n-1}(\bx_{n-1}^{*}) > t^{(i)} \right\} \leq r_{n-1}(i).
\]
Then as in \cref{withRhoKnown:ihpExactBound}, it follows that
\iftoggle{useTwoColumn}{
	\begin{align}
	&\mathbb{P}\left\{ f_{n}(\bx_{n}) - f_{n}(\bx_{n}^{*}) > t^{(i)} \right\} \nonumber \\
	&\; \leq \inf_{0 < \delta \leq \text{diam}(\xSp)} \left\{ \psi(t^{(i)},\delta) + \psi(t^{(i)},\text{diam}(\xSp)) \mathbb{P}\{d_{n}(0) > \delta \} \right\}.
	\end{align}
}{
\begin{align}
\mathbb{P}\left\{ f_{n}(\bx_{n}) - f_{n}(\bx_{n}^{*}) > t^{(i)} \right\} \leq \inf_{0 < \delta \leq \text{diam}(\xSp)} \left\{ \psi(t^{(i)},\delta) + \psi(t^{(i)},\text{diam}(\xSp)) \mathbb{P}\{d_{n}(0) > \delta \} \right\}.
\end{align}	
}
The key then is to bound $\mathbb{P}\{d_{n}(0) > \delta \}$ in terms of $\{r_{n-1}(i)\}_{i=1}^{N}$. Define the function
\[
t(\delta) = \max\{t^{(i)} \;|\; t^{(i)} \leq \delta  \}
\]
to be point $t^{(i)}$ closest to $\delta$ but not greater. Provided that $\frac{2}{m}\left(\sqrt{\delta} - \rho\right)_{+}^{2} \geq t^{(1)}$ and $t^{(i)} = t\left(\frac{2}{m}\left(\sqrt{\delta} - \rho\right)_{+}^{2}\right)$ it holds that
\iftoggle{useTwoColumn}{
	\begin{align}
	\mathbb{P}&\left\{ d_{n}(0) > \delta \right\} \nonumber \\
	&\leq \mathbb{P}\left\{ f_{n-1}(\bx_{n-1}) - f_{n-1}(\bx_{n-1}^{*}) > \frac{2}{m}\left(\sqrt{\delta} - \rho\right)_{+}^{2} \right\} \nonumber \\
	&\leq \mathbb{P}\left\{ f_{n-1}(\bx_{n-1}) - f_{n-1}(\bx_{n-1}^{*}) > t^{(i)} \right\} \nonumber \\
	&\leq r_{n-1}(i).
	\end{align}
}{
\begin{eqnarray}
\mathbb{P}\left\{ d_{n}(0) > \delta \right\} &\leq& \mathbb{P}\left\{ f_{n-1}(\bx_{n-1}) - f_{n-1}(\bx_{n-1}^{*}) > \frac{2}{m}\left(\sqrt{\delta} - \rho\right)_{+}^{2} \right\} \nonumber \\
&\leq& \mathbb{P}\left\{ f_{n-1}(\bx_{n-1}) - f_{n-1}(\bx_{n-1}^{*}) > t^{(i)} \right\} \nonumber \\
&\leq& r_{n-1}(i).
\end{eqnarray}
}
Otherwise, if $\frac{2}{m}\left(\sqrt{\delta} - \rho\right)_{+}^{2} < t^{(1)}$, then
\[
\mathbb{P}\left\{ d_{n}(0) > \delta \right\} \leq \frac{\epsilon}{\frac{2}{m}\left(\sqrt{\delta} - \rho \right)_{+}}.
\]
Define the overall bound for the term $\mathbb{P}\left\{ d_{n}(0) > \delta \right\}$ as follows:
\begin{equation}
\phi_{n}(\delta) \triangleq \begin{cases}
r_{n-1}\left(t\left( \frac{2}{m} \left( \sqrt{\delta} - \rho \right)_{+}^{2}  \right)\right), & t^{(1)} \leq \frac{2}{m} \left( \sqrt{\delta} - \rho \right)_{+}^{2} \\
\frac{\epsilon}{\frac{2}{m} \left( \sqrt{\delta} - \rho \right)_{+}^{2}}, & \text{else}
\end{cases}.
\end{equation}
Then we can set
\[
r_{n}^{(i)} = \inf_{\delta > 0}\left\{ \psi(t^{(i)},\delta) + \psi(t^{(i)},\text{diam}(\xSp)) \phi_{n}(\delta) \right\}.
\]
This algorithm is summarized in Algorithm~\cref{ihp_alg}. In practice, once the bound $\psi(t^{(i)},\text{diam}(\xSp))$ is less than one, then the gains are significant. \figurename{}~\cref{fig:ihp_alg_plot} plots a comparison of the bound produced by Algorithm~\cref{ihp_alg} against the Markov inequality bound from \cref{withRhoKnown:ihpMarkovBound} applied to the problem in \Cref{experiment} with $\numIter_{n} = 300$.
\begin{algorithm}
	\caption{Calculate IHP bounds}
	\label{ihp_alg}
	\begin{algorithmic} 
		\REQUIRE Points $t^{(1)},\ldots,t^{(N)}$
		\STATE 1. Set
		\[
		r_{1}^{(i)} = \frac{\epsilon}{t^{(i)}} \;\;\;\; i=1,\ldots,N.
		\]
		\STATE 2. Compute
		\[
		r_{n}^{(i)} = \min_{0 < \delta \leq \text{diam}(\xSp)} \left\{ \psi(t^{(i)},\delta) + \psi(t^{(i)},\text{diam}(\xSp)) \phi_{n}(\delta) \right\}
		\]
		for $i=1,\ldots,N$.
		\STATE 3. $n\leftarrow n+1$ and go back to step 2.
	\end{algorithmic}
\end{algorithm}

\begin{figure}[!ht]
	\centering
	\includegraphics[width=\myfiguresize]{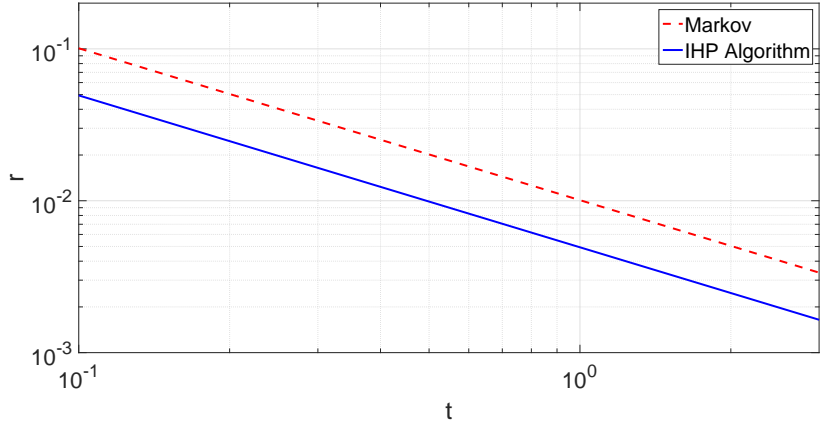}
	\caption{IHP Algorithm Plot}
	\label{fig:ihp_alg_plot} \medskip
\end{figure}

Either the Markov bound of \cref{withRhoKnown:ihpMarkovBound} or Algorithm~\cref{ihp_alg} will produce valid upper bounds of the form
\[
\mathbb{P}\left\{ f_{n}(\bx_{n}) - f_{n}(\bx_{n}^{*}) > t^{(i)} \right\} \leq r_{n}(i) \;\;\;\; i=1,\ldots,N.
\]
Suppose that the set $\{t^{(1)},\ldots,t^{(N)}\}$ contains $t$ at index $i^{*}$. These bounds can in turn be used to select $\numIter_{n}$ to achieve a target $(t,r)$ pair by selecting the smallest $\numIter_{n}$ such that
\[
r_{n}(i^{*}) \leq r.
\]

\section{Estimating the Change in Minimizers}
\label{tracking_est_rho}

In practice, we do not know $\rho$, so we must construct an estimate $\hat{\rho}_{n}$ using the samples $\{z_{i}(k)\}_{\iterIndex=1}^{\numIter_{i}}$, $i=1,\ldots, n$. First, we construct estimates $\tilde{\rho}_{i}$ for the one step changes $\|\bx_{i}^{*} - \bx_{i-1}^{*}\|_{2}$ for each $i$. Next, we combine the one step estimates to construct an overall estimate $\hat{\rho}_{n}$ for $\rho$. As an intermediary step, we also look at a special case in which either
\begin{equation}
\label{prob_form:opt_change_eq}
\|\bx_{n+1}^{*}-\bx_{n}^{*}\|_{2} = \rho
\end{equation}
or
\begin{equation}
\label{prob_form:opt_change_eq_L2}
\|\bx_{n+1}^{*}-\bx_{n}^{*}\|_{L_{2}} = \rho.
\end{equation}
We show that for our estimate $\hat{\rho}_{n}$ and appropriately chosen sequences $\{t_{n}\}$, for all $n$ large enough, $\hat{\rho}_{n} + t_{n} \geq \rho$ almost surely. With this property, analysis similar to that in \Cref{tracking_crit_rho_known} holds.

\subsection{One Step Changes}
We construct an estimate $\tilde{\rho}_{i}$ for the one-step changes $\|\bx_{i}^{*} - \bx_{i-1}^{*}\|$. As a consequence of the strong convexity of $f_{i}(\bx)$, we have the following lemma
\begin{lem}
	\label{estRho:dirEstJustLemma}
	It holds that
	$\| \bx - \bx_{i}^{*} \|_{2} \leq \frac{1}{m} \| \nabla f_{i}(\bx) \|_{2} \;\;\;\;\; \forall i \geq 1,
	\an{\quad \forall \bx\in\xSp }$.
\end{lem}
\begin{proof}
Since our functions $f_{i}(\bx)$ are convex, it holds that
$\inprod{\nabla f_{i}(\bx_{i}^{*})}{\bx - \bx_{i}^{*}} \geq 0,  \forall \bx\in\xSp$.
By the strong \an{monotonicity} of the gradient, a consequence of strong convexity \cite{Boyd2004}, it holds that
$\inprod{\nabla f_{i}(\bx)-\nabla f_{i}(\tilde{\bx})}{\bx - \tilde{\bx}} \geq m \| \bx - \tilde{\bx} \|_{2}^{2} \forall \bx, \tilde{\bx}\in\xSp$. Plugging in \an{$\tilde{\bx}=\bx_{i}^{*}$} yields
\[
\inprod{\nabla f_{i}(\an{\bx}) }{\bx - \bx_{i}^{*}} \geq m \| \bx - \bx_{i}^{*}\|_{2}^{2}.
\]
Applying \an{the Cauchy-Schwarz inequality} yields the result.
\end{proof}
This in turn by way of the triangle inequality proves that
\iftoggle{useTwoColumn}{
\begin{align}
&\| \bx_{i}^{*} - \bx_{i-1}^{*} \|_{2} \nonumber \\
&\;\;\; \leq \|\bx_{i} - \bx_{i-1} \|_{2} + \|\bx_{i} - \bx_{i}^{*} \|_{2}+ \|\bx_{i-1} - \bx_{i-1}^{*} \|_{2} \nonumber \\
&\;\;\; \leq \label{dir_est_deriv} \|\bx_{i} - \bx_{i-1} \|_{2} + \frac{1}{m}\| \nabla f_{i}(\bx_{i}) \|_{2} + \frac{1}{m}\| \nabla f_{i-1}(\bx_{i-1}) \|_{2}.
\end{align}
}{
\begin{eqnarray}
\| \bx_{i}^{*} - \bx_{i-1}^{*} \|_{2} &\leq& \|\bx_{i} - \bx_{i-1} \|_{2} + \|\bx_{i} - \bx_{i}^{*} \|_{2}+ \|\bx_{i-1} - \bx_{i-1}^{*} \|_{2} \nonumber \\
&\leq& \label{dir_est_deriv} \|\bx_{i} - \bx_{i-1} \|_{2} + \frac{1}{m}\| \nabla f_{i}(\bx_{i}) \|_{2} + \frac{1}{m}\| \nabla f_{i-1}(\bx_{i-1}) \|_{2}.
\end{eqnarray}
}
Motivated by this bound, we define the following estimate denoted the \emph{direct estimate} by approximating the gradients
\begin{equation}
\label{dir_est_def}
\tilde{\rho}_{i} \triangleq \|\bx_{i} - \bx_{i-1} \|_{2} + \frac{1}{m}\| \hat{G}_{i} \|_{2} + \frac{1}{m}\| \hat{G}_{i-1} \|_{2}
\end{equation}
where
\[
\hat{G}_{i} = \frac{1}{\numIter_{i}} \sum_{\iterIndex=1}^{\numIter_{i}} \sgrad{\bx_{i}}{\bz_{i}(\iterIndex)}{i}.
\]

\subsection{Combining with Constant Change of Minimizers}
As a special case, we look at combining the one step estimates when either \cref{prob_form:opt_change_eq} or \cref{prob_form:opt_change_eq_L2} holds.
\subsubsection{Euclidean Norm Condition}
Under \cref{prob_form:opt_change_eq}, we construct an estimate by averaging the one step estimates
\begin{equation}
\label{estRho:combEq:euclid}
\hat{\rho}_{n} \triangleq \frac{1}{n-1} \sum_{i=2}^{n} \tilde{\rho}_{i}.
\end{equation}
We want to show that for an appropriate sequence $\{t_{n}\}$, described in \cref{rho_conc_eq} and \cref{rho_conc_ineq} below, and for all $n$ large enough $\hat{\rho}_{n} + t_{n} \geq \rho$ almost surely under \cref{prob_form:opt_change_eq} or \cref{prob_form:opt_change_eq_L2}. The difficulty in actually proving this condition for \cref{dir_est_def} is that when we compute
\[
\hat{G}_{i} = \frac{1}{\numIter_{i}} \sum_{\iterIndex=1}^{\numIter_{i}} \sgrad{\bx_{i}}{\bz_{i}(\iterIndex)}{i}
\]
$\bx_{i}$ and $\{\bz_{i}(\iterIndex)\}_{\iterIndex=1}^{\numIter_{i}}$ are dependent. To get around this issue, we consider performing a second independent draw of samples $\{\tilde{\bz}_{i}(\iterIndex)\}_{\iterIndex=1}^{\numIter_{i}}$. Note that we do not need to actually draw new independent samples; this is purely for the sake of analysis. We start from $\bx_{i-1}$ and produce $\tilde{\bx}_{i}$ using these new samples. For example, with SGD, we have
\begin{equation}
\label{estRho:secondSampleSGD}
\begin{aligned}
\bx_{i}(\iterIndex) &= \Pi_{\xSp}\left[ \bx_{i}(\iterIndex-1) - \mu(\iterIndex) \sgrad{\bx_{i}(\iterIndex-1)}{\bz_{i}(\iterIndex)}{i}  \right]  \\
\tilde{\bx}_{i}(\iterIndex) &= \Pi_{\xSp}\left[ \tilde{\bx}_{i}(\iterIndex-1) - \mu(\iterIndex) \sgrad{\bx_{i}(\iterIndex-1)}{\tilde{\bz}_{i}(\iterIndex)}{i}  \right] 
\end{aligned}
\end{equation}
for $\iterIndex=1,\ldots,\numIter_{i}$ with $\bx_{i}(0) = \tilde{\bx}_{i}(0) = \bx_{i-1}$. Then we copy the form of the direct estimate using $\tilde{\bx}_{i}$ in place of $\bx_{i}$ by defining
\begin{equation}
\label{estRho:rhoTwoEstimate}
\tilde{\rho}_{i}^{(2)} = \|\tilde{\bx}_{i} - \tilde{\bx}_{i-1}\|_{2} + \frac{1}{m} \| \tilde{G}_{i} \|_{2} + \frac{1}{m} \| \tilde{G}_{i-1} \|_{2}
\end{equation}
with
\[
\tilde{G}_{i} = \frac{1}{\numIter_{i}} \sum_{\iterIndex=1}^{\numIter_{i}} \sgrad{\tilde{\bx}_{i}}{\bz_{i}(\iterIndex)}{i}.
\]
In this case, $\tilde{\bx}_{i}$ and $\{\bz_{i}(\iterIndex)\}_{\iterIndex=1}^{\numIter_{i}}$ are independent, so $\mathbb{E}[\tilde{\rho}_{i}^{(2)}] \geq \rho$ by \cref{estRho:dirEstJustLemma}. Under \cref{prob_form:opt_change_eq}, using a dependent sub-Gaussian concentration inequality from \cite{Antonini2005} similar to Hoeffding's inequality, we then argue that $\hat{\rho}_{n}$ from \cref{estRho:combEq:euclid} is close to 
\[
\hat{\rho}_{n}^{(2)} = \frac{1}{n-1} \sum_{i=2}^{n} \tilde{\rho}_{i}^{(2)}
\]
which in turn upper bounds $\rho$ for all $n$ large enough almost surely. Similarly, under \cref{prob_form:opt_change_eq_L2}, we show that $\hat{\rho}_{n}^{2}$ from \cref{estRho:combEq:L2} is close to $(\hat{\rho}_{n}^{(2)})^{2}$, which in turn upper bounds $\rho^{2}$ for all $n$ large enough almost surely.

To proceed with our analysis, suppose that the following conditions hold:
\begin{description}
	\item\namedlabel{probState:assumpB1}{B.1} Suppose there exist functions $C_{i}(\numIter_{i})$ such that
	$\mathbb{E}\left[ \|\bx_{i} - \tilde{\bx}_{i}\|_{2}^{2} \;|\; \mathcal{F}_{i-1} \right] \leq C_{i}^{2}(\numIter_{i})$ with the $\sigma$-algebra defined in \eqref{withRhoUnknown:FSigAlg}.
	\item\namedlabel{probState:assumpB2}{B.2} Suppose that $\forall \bx,\tilde{\bx} \in \xSp$ it holds that
	\[
	\mathbb{E}\left[ \|\sgrad{\bx}{\bz_{i}}{i} - \sgrad{\tilde{\bx}}{\bz_{i}}{i}\|_{2} \;|\; \mathcal{F}_{i-1}  \right] \leq M \mathbb{E}\left[ \| \bx - \tilde{\bx} \|_{2} \;|\; \mathcal{F}_{i-1} \right]
	\]
	and
	\[
	\mathbb{E}\left[ \|\sgrad{\bx}{\bz_{i}}{i} - \nabla f_{i}(\bx) \|_{2} \;|\; \mathcal{F}_{i-1}  \right] \leq \sigma \;\; \forall \bx \in \xSp.
	\]
	\item\namedlabel{probState:assumpB3}{B.3} Suppose the gradients are bounded in the sense that
	$\| \sgrad{\bx}{\bz}{n} \|_{2} \leq G \;\;\;\; \forall \bx \in \xSp , \bz \in \mathcal{Z} $.
\end{description}
Assumption~\ref{probState:assumpB1} is a bound on the difference in how far apart two independent outputs of the optimization algorithm $\bx_{i}$ and $\tilde{\bx}_{i}$ starting from $\bx_{i-1}$ are. \cw{Due to Assumptions~\ref{probState:assump1} and \ref{probState:assump4}, we always have the following choice of $C(\numIter_{i})$:
\begin{align}
\mathbb{E}&\left[ \|\bx_{i} - \tilde{\bx}_{i}\|^{2} \;|\; \mathcal{F}_{i-1} \right] \nonumber \\
 &\leq 2 \mathbb{E}\left[ \|\bx_{i} - \bx_{i}^{*}\|^{2} \;|\; \mathcal{F}_{i-1} \right] + 2 \mathbb{E}\left[ \|\tilde{\bx}_{i} - \bx_{i}^{*}\|^{2} \;|\; \mathcal{F}_{i-1} \right] \nonumber \\
 &\leq \frac{4}{m} \left( \mathbb{E}\left[ f(\bx_{i}) - f(\bx_{i}^{*}) \;|\; \mathcal{F}_{i-1} \right] +  \mathbb{E}\left[ f(\tilde{\bx}_{i}) - f(\bx_{i}^{*}) \;|\; \mathcal{F}_{i-1} \right]  \right) \nonumber \\
 & \leq \frac{4}{m} b(\text{diam}(\xSp), \numIter_{i}) \triangleq C(\numIter_{i}). \nonumber
\end{align}
By a more sophisticated analysis specific to the particular chosen optimization algorithm, it is possible to get tighter $C(\numIter_{i})$ bounds \cite{Wilson2016a}.} Assumption~\ref{probState:assumpB2} controls first, how the noisy gradient changes when it is evaluated at two points $\bx$ and $\tilde{\bx}$ and second, the amount of noise in the noisy gradient. \cw{The first condition is similar to a Lipschitz gradient assumption except imposed on the noisy gradient. For the second part of this assumption, by applying Jensen's inequality and the linearity of the trace of a matrix, it follows that
\begin{align}
\mathbb{E}&\left[ \|\sgrad{\bx}{\bz_{i}}{i} - \nabla f_{i}(\bx) \| \;|\; \mathcal{F}_{i-1}  \right]^{2} \nonumber \\
&\leq \mathbb{E}\left[ \|\sgrad{\bx}{\bz_{i}}{i} - \nabla f_{i}(\bx) \|^{2} \;|\; \mathcal{F}_{i-1}  \right] \nonumber \\
&= \textrm{tr}\left\{ \mathbb{E}\left[ \left( \sgrad{\bx}{\bz_{i}}{i} - \nabla f_{i}(\bx) \right) \left( \sgrad{\bx}{\bz_{i}}{i} - \nabla f_{i}(\bx) \right)^{\top} \;|\; \mathcal{F}_{i-1}  \right] \right\} \nonumber \\
&= \textrm{tr}\left\{ \textrm{Cov}_{\bz_{i}}(\sgrad{\bx}{\bz_{i}}{i})  \right\} \nonumber
\end{align}
where $\textrm{Cov}_{\bz_{i}}(\sgrad{\bx}{\bz_{i}}{i})$ is the covariance matrix of the noisy gradients. Provided there is a uniform bound on the trace of the covariance matrix over $\bx \in \xSp$ and $i \geq 1$, this assumption holds.} Finally, Assumption~\ref{probState:assumpB3} is reasonable if the space $\mathcal{Z}$ that contains the $\bz_{n}$ has finite diameter \cw{and $\sgrad{\bx}{\bz}{n}$ is continuous in the pair $(\bx,\bz)$. In this case, it holds that
\[
\sup_{\bx \in \xSp,\bz \in \mathcal{Z}} \| \sgrad{\bx}{\bz}{n}\| < +\infty
\]
and the assumption is satisfied.}

In the following theorem, we establish that the direct estimate from \cref{estRho:combEq:euclid} (i.e., the Euclidean norm condition) upper bounds $\rho$ from \cref{prob_form:opt_change_eq} eventually.
\begin{thm}
\label{rho_conc_eq}
Provided that assumptions~\ref{probState:assumpB1}-\ref{probState:assumpB3} hold and our sequence $\{t_{n}\}$\footnote{Note that a choice of $t_{n}$ that is no greater than $1/\sqrt{n-1}$ works here.} satisfies
\[
\sum_{n=2}^{\infty} \left( \exp\left\{ - \frac{(n-1)t_{n}^{2}}{18 \text{diam}^{2}(\xSp)}\right\} + 2\exp\left\{ - \frac{m^2(n-1)t_{n}^{2}}{72 G^{2}} \right\} \right) < +\infty
\]
it holds that for all $n$ large enough 
$\hat{\rho}_{n} + D_{n} + t_{n} \geq \rho$
almost surely with
\iftoggle{useTwoColumn}{
$D_{n}$ defined in \eqref{rho_conc_eq:Dn}
\begin{figure*}[!t]
\normalsize
\begin{align}
\label{rho_conc_eq:Dn}
D_{n} &= \frac{1}{n-1} \left[ \left(1 + \frac{M}{m}\right) C_{1} + \sqrt{\frac{\sigma}{\numIter_{1}}}  + 2 \sum_{i=2}^{n-1} \left( \left(1 + \frac{M}{m}\right) C_{i} + \sqrt{\frac{\sigma}{\numIter_{i}}} \right) + \left(1 + \frac{M}{m}\right) C_{n} + \sqrt{\frac{\sigma}{\numIter_{n}}}  \right]
\end{align}
\hrulefill
\vspace*{4pt}
\end{figure*}
}{
\begin{equation}
\label{rho_conc_eq:Dn}
D_{n} = \frac{1}{n-1} \left[ \left(1 + \frac{M}{m}\right) C_{1} + \sqrt{\frac{\sigma}{\numIter_{1}}}  + 2 \sum_{i=2}^{n-1} \left( \left(1 + \frac{M}{m}\right) C_{i} + \sqrt{\frac{\sigma}{\numIter_{i}}} \right) + \left(1 + \frac{M}{m}\right) C_{n} + \sqrt{\frac{\sigma}{\numIter_{n}}}  \right].
\end{equation}
}
\end{thm}
\begin{proof}
See \Cref{usefulConcIneq} .
\end{proof}
From now on, for notational convenience, we absorb $D_{n}$ into the $t_{n}$ term and refer only to $\hat{\rho}_{n} + t_{n}$.

\subsubsection{$L_{2}$ Norm Condition}
Under \cref{prob_form:opt_change_eq_L2}, we construct an estimate by averaging the squares of the one step estimates and taking a square root
\begin{equation}
\label{estRho:combEq:L2}
\hat{\rho}_{n} \triangleq \sqrt{\frac{1}{n-1} \sum_{i=2}^{n} \tilde{\rho}_{i}^{2}}.
\end{equation}

In the following theorem, we establish that the direct estimate from \cref{estRho:combEq:L2} upper bounds $\rho$ from \cref{prob_form:opt_change_eq_L2} eventually.
\begin{thm}
	\label{rho_conc_eq_sq}
	Provided that assumptions~\ref{probState:assumpB1}-\ref{probState:assumpB3} hold and our sequence $\{t_{n}\}$ satisfies
	\[
	\sum_{n=2}^{\infty} \left( \exp\left\{ - \frac{(n-1)t_{n}^{2}}{18 \text{diam}^{2}(\xSp)}\right\} + 2\exp\left\{ - \frac{m^2(n-1)t_{n}^{2}}{72 G^{2}} \right\} \right) < +\infty
	\]
	it holds that for all $n$ large enough 
	$\sqrt{\left( \hat{\rho}_{n} \right)^{2} + \tilde{D}_{n} + t_{n}} \geq \rho$
	almost surely with\\ $\tilde{D}_{n} = 2 \text{diam}(\xSp) D_{n}$.
\end{thm}
\begin{proof}
See \Cref{usefulConcIneq}.
\end{proof}

\subsection{Combining with Bounded Changes of Minimizers}
\label{tracking_est_rho:boundedChange}

We examine estimating $\rho$ in the case that \cref{prob_form:opt_change_L2} holds.
We denote the exact one step time changes by $\rho_{i} \triangleq \| \bx_{i}^{*} - \bx_{i-1}^{*} \|$. The simplest way to combine the estimates from \cref{dir_est_def} would be to set
\[
\hat{\rho}_{n} = \max\{\tilde{\rho}_{2},\ldots,\tilde{\rho}_{n}\}.
\]
For the sake of argument, suppose that $\tilde{\rho}_{i} = \rho_{i} + e_{i}$ with independent $e_{i} \sim \mathcal{N}(0,\sigma^{2})$. Then it follows that \cite[Ex. 10.5.3 on p~.302]{David03}
\[
\mathbb{E}[\hat{\rho}_{n}] \geq \mathbb{E}[\max\{e_{2},\ldots,e_{n}\}].
\]
For independent Gaussian random variables, it holds that $\mathbb{E}[\max\{e,\ldots,e_{n}\}] \to \infty$ as $n \to \infty$ \cw{\cite[Ex. 10.5.3 on p~.302]{David03}}, and therefore this estimate goes to infinity as $n \to \infty$. We do produce an upper bound, but it increases to the trivial bound $\text{diam}^{2}(\xSp)$. 
Next, we examine how to avoid this issue.

\subsubsection{Euclidean Norm Condition}
Suppose that the following conditions hold.
\begin{description}
	\item\namedlabel{probState:assumpB4}{B.4} We have estimates $\hat{h}_{W}: \mathbb{R}^{W} \to \mathbb{R}$ that are non-decreasing in their arguments such that
	\[
	\mathbb{E}[ \hat{h}_{W}(\rho_{j},\ldots,\rho_{j-W+1}) ] \geq \rho.
	\]
	\item\namedlabel{probState:assumpB5}{B.5} 
	There exists absolute constants $\{a_{i}\}_{i=1}^{W}$ for any fixed 
	\an{$W\ge 1$} such that $\forall \bm{p},\bm{q} \in \mathbb{R}^{W}_{\geq 0}$
	\[
	|\hat{h}_{W}(p_{1},\ldots,p_{W}) - \hat{h}_{W}(q_{1},\ldots,q_{W})| \leq \sum_{i=1}^{W} a_{i} |p_{i} - q_{i}|.
	\]
\end{description}

For example, if $\rho_{i} \overset{\text{iid}}{\sim} \text{Unif}[0,\rho]$, then
\begin{equation}
\label{estRho:euclid:hFunc}
\hat{h}_{W}\left( \rho_{i},\rho_{i+1},\ldots,\rho_{i+W-1} \right) = \frac{W+1}{W} \max\{ \rho_{i},\rho_{i+1},\ldots,\rho_{i+W-1} \}
\end{equation}
is an estimate of $\rho$ with the required properties with $a_{i} = 1 + \frac{1}{W}$. In this case, we compute the maximum in \eqref{estRho:euclid:hFunc} over a sliding window and then average the maxima. This estimate will not blow up and will eventually upper bound $\rho$ as we will see in \cref{rho_conc_ineq} below.

Given an estimate satisfying assumptions~\ref{probState:assumpB4}-\ref{probState:assumpB5}, we compute
\[
\bar{\rho}^{(i)} = \hat{h}_{W}(\tilde{\rho}_{i},\tilde{\rho}_{i-1},\ldots,\tilde{\rho}_{i-W+1})
\]
and produce an estimate $\hat{\rho}_{n}$ that is an average of these estimates
\iftoggle{useTwoColumn}{
\begin{align}
\hat{\rho}_{n} &= \frac{1}{n-1} \sum_{i=2}^{n} \bar{\rho}^{(i)} \nonumber \\
\label{ineqCond:basicEst}
&= \frac{1}{n-1} \sum_{i=2}^{n} \hat{h}_{\min\{W,i-1\}}(\tilde{\rho}_{i},\tilde{\rho}_{i-1},\ldots,\tilde{\rho}_{\max\{i-W+1,2\}}).
\end{align}
}{
\begin{equation}
\label{ineqCond:basicEst}
\hat{\rho}_{n} = \frac{1}{n-1} \sum_{i=2}^{n} \bar{\rho}^{(i)} = \frac{1}{n-1} \sum_{i=2}^{n} \hat{h}_{\min\{W,i-1\}}(\tilde{\rho}_{i},\tilde{\rho}_{i-1},\ldots,\tilde{\rho}_{\max\{i-W+1,2\}}).
\end{equation}
}
In the following theorem, we establish that the direct estimate from \cref{ineqCond:basicEst} (i.e., the Euclidean norm condition) upper bounds $\rho$ from \cref{prob_form:opt_change_L2} eventually.
\setcounter{dir_est_ineq_cnt}{\value{thm}}
\begin{thm}
	\label{rho_conc_ineq}
	Provided that assumptions~\ref{probState:assumpB1}-\ref{probState:assumpB5} hold and our sequence $\{t_{n}\}$ is chosen such that
	\iftoggle{useTwoColumn}{
	\begin{align}
	\sum_{n=2}^{\infty} &\left( \exp\left\{ - \frac{(n-W)^{2}t_{n}^{2}}{18(n-1) \text{diam}^{2}(\xSp)\left( \sum_{j=1}^{W} a_{j} \right)^{2}}\right\} \right. \nonumber \\
	&\;\;\;\;\;\;\;\;\;\;\;\;\;\; \left. + 2\exp\left\{ - \frac{m^2(n-W)^2 t_{n}^{2}}{72 (n-1) G^{2}\left( \sum_{j=1}^{W} a_{j} \right)^{2}} \right\} \right) \nonumber
	\end{align}
	}{
	\[
	\sum_{n=2}^{\infty} \left( \exp\left\{ - \frac{(n-W)^{2}t_{n}^{2}}{18(n-1) \text{diam}^{2}(\xSp)\left( \sum_{j=1}^{W} a_{j} \right)^{2}}\right\} + 2\exp\left\{ - \frac{m^2(n-W)^2 t_{n}^{2}}{72 (n-1) G^{2}\left( \sum_{j=1}^{W} a_{j} \right)^{2}} \right\} \right)
	\]	
	}
	is finite,
it holds that for all $n$ large enough 
\[
\hat{\rho}_{n} + \left( \frac{n-1}{n-W} \sum_{j = 1}^{W} a_{j} \right) D_{n} + t_{n} \geq \rho
\]
with $D_{n}$ from \cref{rho_conc_eq}.
\end{thm}
\begin{proof}
	The proof in this case is similar to the proof for the equality assumption on $\rho$ in \cref{prob_form:opt_change_eq} and is provided in \Cref{usefulConcIneq}.
\end{proof}
As before, we will absorb $\left( \frac{n-1}{n-W} \sum_{j = 1}^{W} a_{j} \right) D_{n}$ into $t_{n}$.

\subsubsection{$L_{2}$ Norm Condition}
Suppose that the following conditions hold, which are\\ analogs of assumptions~\ref{probState:assumpB4}-\ref{probState:assumpB5}
\begin{description}
	\item\namedlabel{probState:assumpB6}{B.6} 
	We have estimates $\hat{h}_{W}: \mathbb{R}^{W} \to \mathbb{R}$ that are non-decreasing in their arguments such that
	\[
	\mathbb{E}[ \hat{h}_{W}(\rho_{j}^{2},\ldots,\rho_{j-W+1}^{2}) ] \geq \rho^{2}.
	\]
	\item\namedlabel{probState:assumpB7}{B.7} 
	There exists absolute constants $\{a_{i}\}_{i=1}^{W}$ for any fixed \an{$W\ge 1$} such that $\forall \bm{p},\bm{q} \in \mathbb{R}^{W}_{\geq 0}$
	\[
	|\hat{h}_{W}(p_{1}^{2},\ldots,p_{W}^{2}) - \hat{h}_{W}(q_{1}^{2},\ldots,q_{W}^{2})| \leq \sum_{i=1}^{W} a_{i} |p_{i}^{2} - q_{i}^{2}|.
	\]
\end{description}
For example, if $\rho_{i} \overset{\text{iid}}{\sim} \text{Unif}[0,\rho]$, then
\[
\hat{h}_{W}\left( \rho_{i}^{2},\rho_{i+1}^{2},\ldots,\rho_{i+W-1}^{2} \right) = \frac{W+2}{W} \max\{ \rho_{i}^{2},\rho_{i+1}^{2},\ldots,\rho_{i+W-1}^{2} \}
\]
is an estimate of $\rho$ with the required properties with $b_{i} = \frac{W+2}{W}$. In this case, we compute the max over a sliding window and then average the maximums. This estimate will not blow up but will eventually upper bound $\rho$ as we will see later.

Given an estimate satisfying assumptions~\ref{probState:assumpB4}-\ref{probState:assumpB5}, we compute
\[
\bar{\rho}^{(i)} = \sqrt{\hat{h}_{W}(\tilde{\rho}_{i}^{2},\tilde{\rho}_{i-1}^{2},\ldots,\tilde{\rho}_{i-W+1}^{2})}.
\]
Under assumptions~\ref{probState:assumpB1}-\ref{probState:assumpB3} and assumptions~\ref{probState:assumpB6}-\ref{probState:assumpB7}, we can then show that
\begin{equation}
\label{estRho:combIneq:L2}
\hat{\rho}_{n} = \sqrt{\frac{1}{n-1} \sum_{i=2}^{n} \left( \bar{\rho}^{(i)} \right)^{2} }
\end{equation}
eventually upper bounds $\rho$.

\begin{thm}
	\label{rho_conc_ineq_sq}
	Provided that assumptions~\ref{probState:assumpB1}-\ref{probState:assumpB3} and \ref{probState:assumpB6}-\ref{probState:assumpB7} hold and our sequence $\{t_{n}\}$ is chosen such that the sum
	\iftoggle{useTwoColumn}{
	\begin{align}
	\sum_{n=2}^{\infty} &\left( \exp\left\{ - \frac{(n-W)^{2}t_{n}^{2}}{18(n-1) \text{diam}^{2}(\xSp)\left( \sum_{j=1}^{W} a_{j} \right)^{2}}\right\} \right. \nonumber \\
	&\left. \;\;\;\;\;\;\;\;\;\;\; + 2\exp\left\{ - \frac{m^2(n-W)^2 t_{n}^{2}}{72 (n-1) G^{2}\left( \sum_{j=1}^{W} a_{j} \right)^{2}} \right\} \right) \nonumber
	\end{align}
	}{
	\[
	\sum_{n=2}^{\infty} \left( \exp\left\{ - \frac{(n-W)^{2}t_{n}^{2}}{18(n-1) \text{diam}^{2}(\xSp)\left( \sum_{j=1}^{W} a_{j} \right)^{2}}\right\} + 2\exp\left\{ - \frac{m^2(n-W)^2 t_{n}^{2}}{72 (n-1) G^{2}\left( \sum_{j=1}^{W} a_{j} \right)^{2}} \right\} \right)
	\]
	}
	is finite, it holds that for all $n$ large enough 
\[
\sqrt{\left(\hat{\rho}_{n}\right)^{2} + \left( \frac{n-1}{n-W} \sum_{j = 1}^{W} a_{j} \right) \tilde{D}_{n} + t_{n}} \geq \rho
\]
with $\tilde{D}_{n}$ from \cref{rho_conc_eq_sq}.
\end{thm}
\begin{proof}
	The proof in this case is similar to the proof for the equality assumption on $\rho$ in \cref{prob_form:opt_change_eq} and is provided in \Cref{usefulConcIneq}.
\end{proof}

\section{Tracking Analysis with Change in Minimizers Unknown}
\label{withRhoUnknown}

We now examine the case with $\rho$ unknown. We extend the work of \Cref{tracking_crit_rho_known} using the estimate of $\rho$ in \Cref{tracking_est_rho}. Our analysis depends on the following crucial assumptions:
\begin{description}
\item\namedlabel{probState:assumpC1}{C.1} For appropriate sequences $\{t_{n}\}$, for all $n$ sufficiently large it holds that $\hat{\rho}_{n} + t_{n} \geq \rho$ almost surely.
\item\namedlabel{probState:assumpC2}{C.2} 
The bound $b(d_{0},\numIter_{n})$ \an{defined in Assumption~\ref{probState:assump4}} factors as $b(d_{0},\numIter_{n}) = \alpha(\numIter_{n}) d_{0}^{2} + \beta(\numIter_{n})$.
\end{description}
\cw{We have demonstrated that Assumption~\ref{probState:assumpC1} holds for the direct estimate of $\rho$ in Section~\ref{tracking_est_rho}. Section~\ref{bBounds} has some examples of $b(d_{0},\numIter)$ which factor as in C.2. For many variants of SGD, the bound $b(d_{0},\numIter)$ has one term $\alpha(\numIter_{n})d_{0}^{2}$ that controls how the optimization algorithms forgets its initial condition and another term $\beta(\numIter_{n})$ that controls the asymptotic performance.}

In this section, we assume that either the constant change in minimizers condition, \cref{prob_form:opt_change_L2}, or the bounded change in minimizers condition, \cref{prob_form:opt_change_eq_L2}, holds. Our analysis is not affected by which one is true. We use the following result, proved in \Cref{appendix_proofs_rho_unknown}, to derive rules to pick $\numIter_{n}$ when $\rho$ is unknown:
\setcounter{K_rho_unknown}{\value{thm}}
\begin{thm}
	\label{withRhoUnknown:meanGapRhoKnownLemma}
	Under assumptions~\ref{probState:assumpC1}-\ref{probState:assumpC2}, with $\numIter_{n} \geq \numIter^{*}$ for all $n$ large enough where $\numIter^{*}$ is defined in \eqref{K_with_rho_known} we have  
	\[
	\limsup_{n \to \infty} \left( \mathbb{E}[f_{n}(\bx_{n})] - f_{n}(\bx_{n}^{*})  \right) \leq \epsilon
	\]
	almost surely
\end{thm}
\begin{proof}
See \Cref{appendix_proofs_rho_unknown}.
\end{proof}

\subsection{Update Past \MeanGap{} Bounds}
\label{withRhoUnknown:updatePast}

We first consider updating all past \meangap{} bounds as we go. At time $n$, we plug-in $\hat{\rho}_{n-1} + t_{n-1}$ in place of $\rho$ and follow the analysis of \Cref{tracking_crit_rho_known}. Define
\begin{eqnarray}
\hat{\epsilon}_{i}^{(n)} &=& b\left(   \sqrt{\frac{2}{m} \hat{\epsilon}_{i-1}^{(n)}}  + \hat{\rho}_{n-1} + t_{n-1} ,\numIter_{i} \right) \;\;\;\; i=1,\ldots,n. \nonumber
\end{eqnarray}
If it holds that $\hat{\rho}_{n-1} + t_{n-1} \geq \rho$, then ${\mathbb{E}\left[ f_{n}(\bx_{n})  \right] - f_{n}(\bx_{n}^{*}) \leq \hat{\epsilon}_{n}^{(i)}}$ for 
${i=1,\ldots,n}$. Assumption~\ref{probState:assumpC1} guarantees that this holds for all $n$ large enough almost surely. We can thus set $\numIter_{n}$ equal to
\[
\numIter_{n} = \min\left\{ \numIter \;\Bigg|\;  b\left(  \sqrt{ \frac{2}{m} \max\{\hat{\epsilon}^{(n-1)}_{n-1},\epsilon\} } + \hat{\rho}_{n-1} + t_{n-1}  , \numIter  \right) \leq \epsilon \right\}
\]
for all $n \geq 3$ to achieve \meangap{} $\epsilon$. The maximum in this definition ensures that when $\hat{\rho}_{n-1} + t_{n-1} \geq \rho$, $\numIter_{n} \geq \numIter^{*}$ with $\numIter^{*}$ from \cref{K_with_rho_known}. We can therefore apply \cref{withRhoUnknown:meanGapRhoKnownLemma}.

\subsection{Do Not Update Past \MeanGap{} Bounds}
\label{withRhoUnknown:doNotUpdatePast}

Updating all past estimates of the \meangap{} bounds from time $1$ up to $n$ imposes a computational and memory  burden. Suppose that instead for all $n \geq 3$ we set
\begin{equation}
\label{withRhoUnknown:KnChoice}
\numIter_{n} = \min \left\{ \numIter \geq 1 \;\Bigg|\; b\left(  \sqrt{\frac{2\epsilon}{m}} + \hat{\rho}_{n-1} + t_{n-1} , \numIter  \right) \leq \epsilon \right\}.
\end{equation}
This is the same form as the choice in \cref{K_with_rho_known} with $\hat{\rho}_{n-1} + t_{n-1}$ in place of $\rho$. Due to assumption~\ref{probState:assumpC1}, for all $n$ large enough it holds that $\hat{\rho}_{n} + t_{n} \geq \rho$ almost surely. Then by the monotonicity assumption in \ref{probState:assump4}, for all $n$ large enough we would pick $\numIter_{n} \geq \numIter^{*}$ almost surely. We can therefore apply \cref{withRhoUnknown:meanGapRhoKnownLemma}.

\subsection{In High Probability Bounds}
\label{ihpRhoUnknownAnalysis}

We can adopt the same approach as with $\rho$ known by substituting $\hat{\rho}_{n-1} + t_{n-1}$ in place of $\rho$. As soon as $\rho_{n-1} + t_{n-1} \geq \rho$, \cref{ihp_alg} will produce upper bounds of the form
\[
\mathbb{P}\left\{ f_{n}(\bx_{n}) - f_{n}(\bx_{n}^{*}) > t^{(i)} \right\} \leq r_{n}(i) \;\;\;\; i=1,\ldots,N.
\]
Suppose that the set $\{t^{(1)},\ldots,t^{(N)}\}$ contains $t$ at index $i^{*}$. These bounds can in turn be used to select $\numIter_{n}$ to achieve a target $(t,r)$ pair by selecting the smallest $\numIter_{n}$ such that
\[
r_{n}(i^{*}) \leq r.
\]

\section{Experiment}
\label{experiment}

We apply our framework to a  mean-squared vector estimation problem \footnote{In \cite{Wilson2015a}, we apply the framework developed in this paper to a variety of machine learning problems using real data.}. We fix the following signal model:
\[
y_{n} = \fRV_{n}^{\top} \bw_{n} + e_{n}.
\]
Our goal is to estimate $\fRV_{n}$. We consider minimizing the following functions to estimate $\fRV_{n}$
\begin{equation}
\label{experiment:fDef}
f_{n}(\bx) = \mathbb{E}_{\bz_{n} \sim p_{n}}\left[ \frac{1}{2}(y_{n} - \bx^{\top}\bw_{n})^{2}  \right].
\end{equation}
By simple algebraic manipulation, it holds that
\begin{equation}
\label{experiment:fSimpDef}
f_{n}(\bx) = \frac{1}{2}(\bx - \fRV_{n})^{\top} \mathbb{E}[\bw_{n}\bw_{n}^{\top}] (\bx - \fRV_{n}) + \frac{1}{2} \mathbb{E}[e_{n}^{2}].
\end{equation}
It is easy to see then that $\bx_{n}^{*}=\fRV_{n}$. Set
\begin{equation}
\label{experiment:zNDef}
\bz_{n} \triangleq [\bw_{n}^{\top} \; e_{n}]^{\top}
\end{equation}
and define the stochastic gradients
\[
\sgrad{\bx}{\bz_{n}}{n} \triangleq -(y_{n} - \bx^{\top}\bw_{n}) \bw_{n}
\]
which satisfy the required condition in \cref{prob_form:stochGradDef}.  To find an approximation to $\bx_{n}^{*}$, we apply SGD using the inverse step size averaging technique discussed in \Cref{bBounds}\iftoggle{useArxiv}{}{ of the extended version of this paper \cite{Wilson2016a}}.

Let $\bw_{n} \sim \mathcal{N}\left(\bm{0},\frac{\sigma_{\bw}^{2}}{d}\bm{I}\right)$ where $\bw_{n} \in \mathbb{R}^{d}$ and $e_{n} \sim \mathcal{N}(0,\sigma_{e}^{2})$. We assume $\fRV_{n}$ is a deterministic sequence satisfying 
\begin{equation}
\label{experiment:rhoDef}
\|\fRV_{n+1} - \fRV_{n} \|_{2} \leq \rho.
\end{equation}
Since $\bx_{n}^{*} = \fRV_{n}$, the minimizer change condition in \cref{prob_form:opt_change_L2} is satisfied. We use the $\rho$ estimate in \cref{estRho:combEq:euclid}. Note that since $\{\fRV_{n}\}$ is deterministic, we cannot apply a Kalman filter. Furthermore, we suppose that the collection of all $\bw_{n}$, $e_{n}$, and $\fRV_{n}$ over all time instants are independent.

With this choice of model combined with the form of the functions in \cref{experiment:fSimpDef}, it is clear that the functions $f_{n}(\bx)$ are strongly convex with $m = \sigma_{\bw}^{2}/d$ satisfying assumption~\ref{probState:assump2}. By applying the inequality $(a+b) \leq 2a^2 + 2b^{2}$, it follows that
\iftoggle{useTwoColumn}{
\begin{align}
\mathbb{E}&\| \sgrad{\bx}{\bz_{n}}{n} \|_{2}^{2} \nonumber \\
&=  \mathbb{E}\| \sgrad{\bx^{*}}{\bz_{n}}{n} + \left(  \sgrad{\bx}{\bz_{n}}{n} - \sgrad{\bx^{*}}{\bz_{n}}{n} \right) \|_{2}^{2} \nonumber \\
&\leq 2 \mathbb{E}\| \sgrad{\bx^{*}}{\bz_{n}}{n} \|_{2}^{2} + 2 \mathbb{E}\|  \sgrad{\bx}{\bz_{n}}{n} - \sgrad{\bx^{*}}{\bz_{n}}{n} \|_{2}^{2}. \nonumber
\end{align}
}{
\begin{eqnarray}
\mathbb{E}\| \sgrad{\bx}{\bz_{n}}{n} \|_{2}^{2} &=&  \mathbb{E}\| \sgrad{\bx^{*}}{\bz_{n}}{n} + \left(  \sgrad{\bx}{\bz_{n}}{n} - \sgrad{\bx^{*}}{\bz_{n}}{n} \right) \|_{2}^{2} \nonumber \\
&\leq& 2 \mathbb{E}\| \sgrad{\bx^{*}}{\bz_{n}}{n} \|_{2}^{2} + 2 \mathbb{E}\|  \sgrad{\bx}{\bz_{n}}{n} - \sgrad{\bx^{*}}{\bz_{n}}{n} \|_{2}^{2}. \nonumber
\end{eqnarray}
}
For the first term, we have
\begin{eqnarray}
\mathbb{E}\| \sgrad{\bx_{n}^{*}}{\bz_{n}}{n} \|_{2}^{2} &=& \mathbb{E}\| e_{n} \bw_{n} \|_{2}^{2} \nonumber \\
&=& \mathbb{E}[e_{n}^{2}] \mathbb{E}\|\bw_{n}\|_{2}^{2} \nonumber \\
&=& \sigma_{e}^{2} \sigma_{\bw}^{2}  \nonumber
\end{eqnarray}
and for the second term, we have
\iftoggle{useTwoColumn}{
\begin{align}
\mathbb{E}&\| \sgrad{\bx}{\bz_{n}}{n} - \sgrad{\bx^{*}}{\bz_{n}}{n} \|^{2} \nonumber \\
&= \mathbb{E}\left[ \| \bw_{n} \|^{2} (\bx - \bx^{*}) \bw_{n} \bw_{n}^{\top} (\bx - \bx^{*}) \right] \nonumber \\
&\leq \mathbb{E}\left[ \| \bw_{n} \|^{4} \right] \| \bx - \bx^{*}) \|^{2} \nonumber \\
&\leq 3 \sigma_{w}^{4} \| \bx - \bx^{*} \|^{2}. \nonumber
\end{align} 
}{
\begin{eqnarray}
\mathbb{E}\| \sgrad{\bx}{\bz_{n}}{n} - \sgrad{\bx^{*}}{\bz_{n}}{n} \|^{2} &=& \mathbb{E}\left[ \| \bw_{n} \|^{2} (\bx - \bx^{*}) \bw_{n} \bw_{n}^{\top} (\bx - \bx^{*}) \right] \nonumber \\
&\leq& \mathbb{E}\left[ \| \bw_{n} \|^{4} \right] \| \bx - \bx^{*}) \|^{2} \nonumber \\
&\leq& 3 \sigma_{w}^{4} \| \bx - \bx^{*} \|^{2}. \nonumber
\end{eqnarray} 
}

The last inequality follows since $\bw_{n}$ is a centered Gaussian and therefore, it holds that $\mathbb{E}|\bw_{n}|^{4} \leq 3 \mathbb{E}|\bw_{n}|^{2}$ 
This implies that
\[
\mathbb{E}\| \sgrad{\bx}{\bz_{n}}{n} \|_{2}^{2} \leq 2 \sigma_{e}^{2} \sigma_{\bw}^{2} +  6 \sigma_{w}^{4} \| \bx - \bx^{*} \|_{2}^{2}.
\]
Therefore, for assumption~\ref{probState:assump5}, we can set
\begin{eqnarray}
A &=& 2 \sigma_{e}^{2} \sigma_{\bw}^{2} \nonumber \\
B &=& 6 \sigma_{w}^{4}. \nonumber
\end{eqnarray}

Putting it together, we have the parameters summarized in \Cref{experiment:paramTable}.
\begin{table}[!ht]
	\centering
	\begin{tabular}{|c|c|}
		\hline
		Parameter & Value \\
		\hline
		$m$ & $\sigma_{\bw}^{2}/d$ \\
		$A$ & $2 \sigma_{e}^{2} \sigma_{\bw}^{2}$\\
		$B$ & $6\sigma_{\bw}^{4}$\\
		\hline
	\end{tabular}
	\caption{Parameter Table}
	\label{experiment:paramTable}
\end{table}

\noindent For this simulation, we choose $d = 2$, $\sigma_{w}^{2} = 0.5$, $\sigma_{e}^{2} = 0.5$, and $\rho = 1$.

\subsection{Mean Tracking Criterion}
First, we assume that $\rho$ and all the parameters in \Cref{experiment:paramTable} are known. We focus on the mean tracking criterion in \cref{prob_form:L2_crit}. \Cref{fig:b_eps_vs_N_tradeoff} shows the trade-off for the optimal $\epsilon$ versus $\numIter^{*}$ defined in \cref{K_with_rho_known}. Any pair $(\epsilon,\numIter)$ located above this curve can be achieved, in the sense that by setting $\numIter_{n} = \numIter$, we achieve
\[
\limsup_{n \to \infty} \left( \mathbb{E}[f_{n}(\bx_{n})] - f_{n}(\bx_{n}^{*}) \right) \leq \epsilon.
\]
\begin{figure}[!h]
	\normalsize
	\centering
	\includegraphics[width=\myfiguresize]{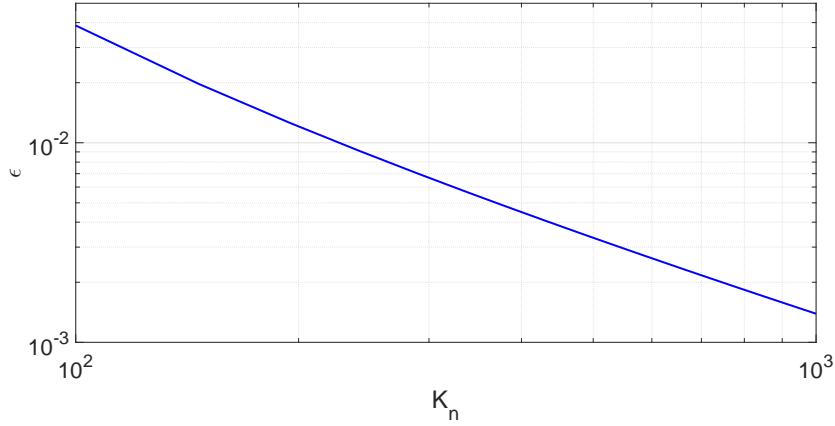}
	\caption{Mean tracking accuracy ($\epsilon$) vs. Number of samples ($\numIter_{n}$)}
	\label{fig:b_eps_vs_N_tradeoff} \medskip
\end{figure}

Next, we examine the case where $\rho$ and the parameters in \Cref{experiment:paramTable} are unknown. We estimate $\rho$ using the techniques introduced in \Cref{tracking_est_rho}, specifically \cref{ineqCond:basicEst}, select $\numIter_{n}$ using the rule in \cref{withRhoUnknown:KnChoice}, and estimate the parameters using the techniques in \Cref{parameterEstimation}\iftoggle{useArxiv}{}{of the extended version of this paper \cite{Wilson2016a}}. We target several different values of the mean tracking accuracy $\epsilon$ from \cref{prob_form:L2_crit}, including $0.001$, $0.01$, and $0.03$. For the problem in this section, we can compute an estimate of the mean tracking accuracy to evaluate our methods. First, we have $f_{n}(\bx_{n}^{*}) = \frac{1}{2} \sigma_{e}^{2}$. Second, for the sake of evaluation, we draw additional samples $\{\tilde{\bz}_{n}(\iterIndex)\}_{\iterIndex=1}^{T_{n}} \overset{\text{iid}}{\sim} p_{n}$ and compute
\[
\frac{1}{T_{n}} \frac{1}{2} \sum_{\iterIndex=1}^{T_{n}} \left( \tilde{y}_{n}(\iterIndex) - \bx_{n}^{\top} \tilde{\bw}_{n}(\iterIndex) \right)^{2}
\]
to estimate $f_{n}(\bx_{n})$. With these two pieces, we can estimate the tracking accuracy by computing
\begin{equation}
\label{experiment:meanCritEstTech}
\frac{1}{T_{n}} \frac{1}{2} \sum_{\iterIndex=1}^{T_{n}} \left( \tilde{y}_{n}(\iterIndex) - \bx_{n}^{\top} \tilde{\bw}_{n}(\iterIndex) \right)^{2} - \frac{1}{2} \sigma_{e}^{2}.
\end{equation}
\Cref{experiment:estMeanCrit} shows an estimate of the actual achieved mean tacking accuracy for three different $\epsilon$ \meangap{} targets averaged over $n=1$ to $100$. In all cases, we meet our \meangap{} target on average.

\begin{table}[!ht]
	\centering
	\begin{tabular}{|c|c|}
		\hline
		$\epsilon$ target & $\epsilon$ Estimate \\
		\hline
		0.001 & $0.0008 \pm 0.0002$ \\
		\hline
		0.01 & $0.0073 \pm 0.0012$ \\
		\hline
		0.03 & $0.022 \pm 0.0022$ \\
		\hline
	\end{tabular}	
	\caption{Estimate of \meangap{}}
	\label{experiment:estMeanCrit}
\end{table}

\Cref{experiment:KnSel} shows the selected number of samples $\numIter_{n}$ for each mean tracking error target $\epsilon$. In all cases we start from an insufficient number of samples $\numIter_{1} = \numIter_{2} = 50$. Due to the guarantees of \Cref{withRhoUnknown:meanGapRhoKnownLemma}, eventually we compensate for this initial bad choice to select $\numIter_{n}$ large enough. This process can be seen in \Cref{experiment:KnSel} by the spikes in $\numIter_{n}$ for small $n$ to ``catch up" to the correct $\numIter^{*}$. For larger $n$, the choice of $\numIter_{n}$ settles down and does not vary greatly. Finally, as expected, for smaller choices of $\epsilon$, $\numIter_{n}$ is larger.

\Cref{experiment:rhoEst} shows the estimate of $\rho$. Our estimates of $\rho$ upper bound the true value of $\rho = 1$ as desired. The initial spike in the estimates of $\rho$ may be due to form of the $\hat{h}_{W}$ function in \cref{estRho:euclid:hFunc} with $W=4$. Before we have four one step estimates of $\rho$ to plug in to $\hat{h}_{4}$, we use $\hat{h}_{1},\hat{h}_{2},\hat{h}_{3}$ per \eqref{estRho:euclid:hFunc}. The scaling factor in front of the maximum for these functions is $2, \frac{3}{2},\frac{4}{3}$ before settling down to $\frac{5}{4}$. The larger scaling factors combined with the small number of one step $\rho$ estimates results in an initial spike in the estimate of $\rho$. With more one step estimates of $\rho$, the overall estimate settles down. Finally, with a smaller mean tracking error target, we produce a tighter estimate of $\rho$. 

\Cref{experiment:meanTracking} shows the estimate $\hat{\epsilon}_{i,n}$ of the mean tracking error achieved computed by updating the past. As mentioned above, we have an insufficient initial choice of $\numIter_{1}$ and $\numIter_{2}$, which causes initial spikes in the estimate of mean tracking error. Our choice of $\numIter_{n}$ drive these mean tracking error estimates below their target values of $\epsilon$.
\begin{figure}[!h]
	\centering
	\includegraphics[width=\myfiguresize]{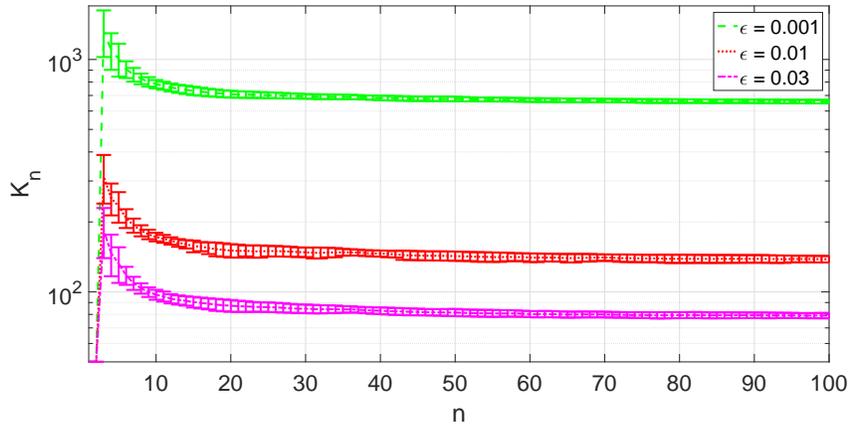}
	\caption{Number of samples ($\numIter_{n}$)}
	\label{experiment:KnSel} \medskip
\end{figure}

\begin{figure}[!h]
	\centering
	\includegraphics[width=\myfiguresize]{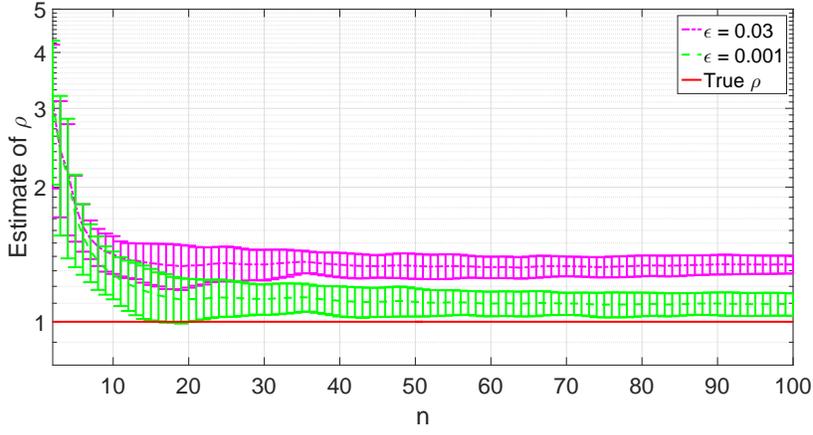}
	\caption{Estimate of change in minimizers ($\rho$)}
	\label{experiment:rhoEst} \medskip
\end{figure}

\begin{figure}[!h]
	\centering
	\includegraphics[width=\myfiguresize]{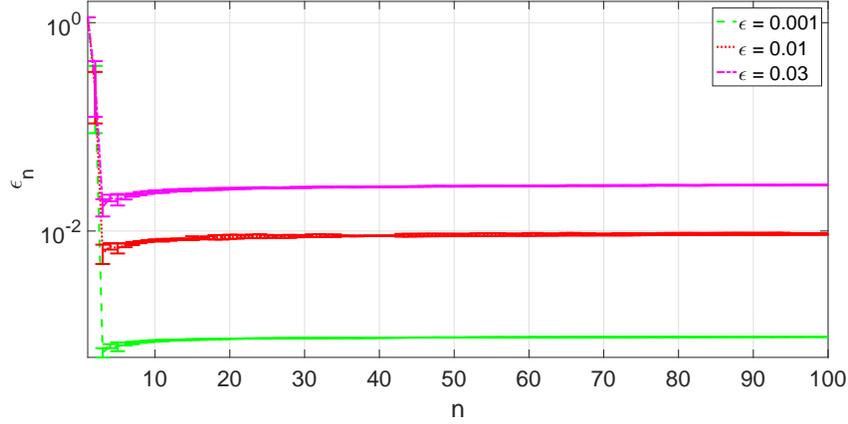}
	\caption{Estimate of mean tracking accuracy}
	\label{experiment:meanTracking} \medskip
\end{figure}

\subsection{IHP Tracking Criterion}
\label{ihpExperiment}

\Cref{fig:r_vs_eps_ihp} plots $r$ vs $\epsilon$ for several values of $\numIter$ by applying the IHP algorithm. The IHP bounds appear to be loose in general as we need fairly large values of $\numIter$ to get non-trivial bounds for reasonable $\epsilon$ and small $r$. The looseness of these bounds is not surprising, since we are only using the first moment of the tracking error to bound.

\begin{figure}[!h]
	\normalsize
	\centering
	\includegraphics[width=\myfiguresize]{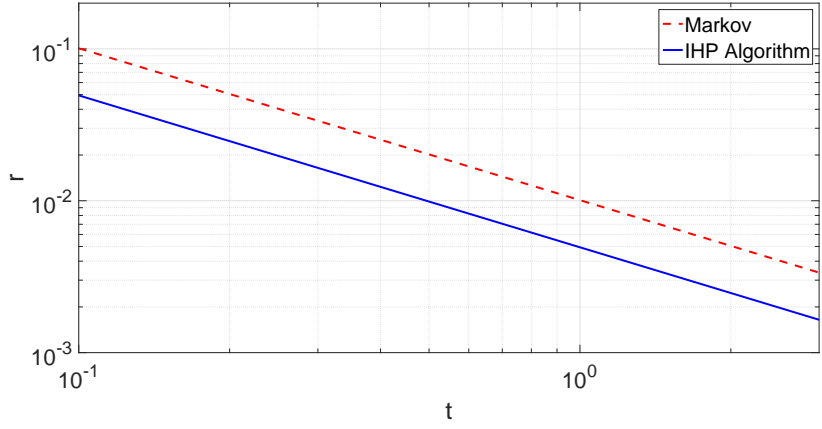}
	\caption{In high probability criterion: $r$ vs t}
	\label{fig:r_vs_eps_ihp} \medskip
\end{figure} 

We choose $\numIter_{n}$ by targeting $t = 0.1$ and $r = 0.25$. \Cref{fig:ihp_test} shows the empirical probability that $f_{n}(\bx_{n}) - f_{n}(\bx_{n}^{*}) > t$. As mentioned above, we can compute $f_{n}(\bx_{n}) - f_{n}(\bx_{n}^{*})$ exactly, so we can calculate the fraction of the time that the loss violates the $t = 0.1$ constraint. The empirical probability that $f_{n}(\bx_{n}) - f_{n}(\bx_{n}^{*}) > t$
satisfies our target value of $r = 0.25$ to within the error bars.
\begin{figure}[!h]
	\normalsize
	\centering
	\includegraphics[width=\myfiguresize]{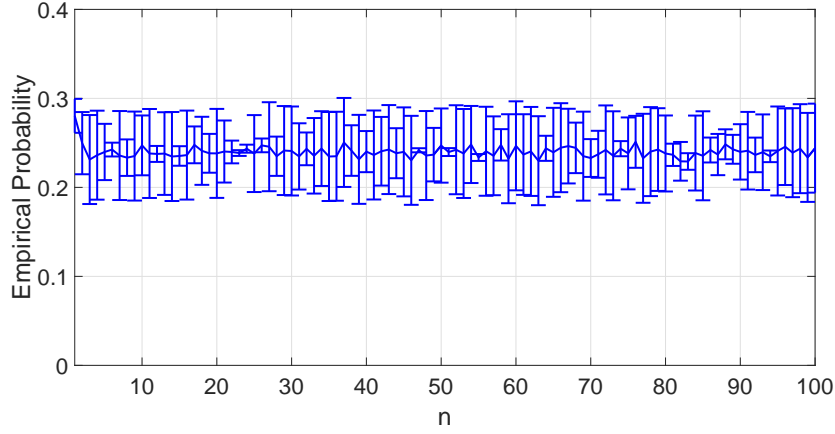}
	\caption{Empirical Probability}
	\label{fig:ihp_test} \medskip
\end{figure}

\subsection{Kalman Filter Comparison}
\label{kalmanComparison}

We now consider a slight modification of our model, so that we can apply the Kalman filter. As mentioned above, since we assume that $\fRV_{n}$ is generated as a deterministic sequence, we cannot apply the Kalman filter. In this section, we instead assume that
\[
\fRV_{n} - \fRV_{n-1} \sim \mathcal{N}(\bm{0},\sigma^{2}\bm{I})
\]
and $\fRV_{1}$ is fixed. Then it holds that
\[
\| \fRV_{n} - \fRV_{n-1} \|_{L_{2}} \leq \sigma \triangleq \rho.
\]
We satisfy \eqref{prob_form:opt_change_L2} and use the estimate of $\rho$ in \eqref{estRho:combEq:L2}.

To apply the Kalman filter, we take $\fRV_{n}$ to be the state of the system. The state evolution equation is given by
\[
\fRV_{n}(\iterIndex) = \begin{cases}
\fRV_{1}(1), & \text{fixed} \\
\fRV_{n-1}(\numIter_{n-1}) + \zeta_{n}, & \iterIndex = 1 \\
\fRV_{n}(\iterIndex-1), & 1 < \iterIndex \leq \numIter_{n}
\end{cases} 
\]
with $\zeta_{n} \sim \mathcal{N}(\bm{0},\sigma^{2} \bm{I})$. The observation equation is given by the pair $(\bw_{n}(\iterIndex),y_{n}(\iterIndex))$ with
\[
y_{n}(\iterIndex) = \fRV_{n}^{\top}(\iterIndex) \bw_{n}(\iterIndex) + e_{n}(\iterIndex).
\]
Let $\hat{\fRV}_{n}(\iterIndex|\tilde{\iterIndex})$ be the estimate of $\fRV_{n}$ at time $\iterIndex$ of epoch $n$ given all the information up to $\tilde{\iterIndex}$ with $\iterIndex \geq \tilde{\iterIndex}$. Let $P_{n}(\iterIndex|\tilde{\iterIndex})$ be the estimate of the covariance. The prediction equations for the state estimate and covariance estimate are given by \cite{Haykin2002}
\begin{eqnarray}
\hat{\fRV}_{n}(\iterIndex|\iterIndex-1) &=& \hat{\fRV}_{n}(\iterIndex-1|\iterIndex-1) \nonumber \\
P_{n}(\iterIndex|\iterIndex - 1) &=& P_{n}(\iterIndex - 1|\iterIndex - 1) + \sigma^{2} \bm{I} \mathbbm{1}_{\{\iterIndex = 1\}}.
\end{eqnarray}
The update equations are given by
\begin{eqnarray}
\hat{\fRV}_{n}(\iterIndex|\iterIndex) &=& \hat{\fRV}_{n}(\iterIndex|\iterIndex-1) + \bm{G}_{n}(\iterIndex) (y_{n}(\iterIndex) - \hat{\fRV}_{n}^{\top}(\iterIndex|\iterIndex-1) \bw_{n}(\iterIndex)) \nonumber \\
P_{n}(\iterIndex|\iterIndex) &=& (\bm{I} - G_{n}(\iterIndex) \bw_{n}^{\top}(\iterIndex)) P_{n}(\iterIndex|\iterIndex-1) \nonumber \\
G_{n}(\iterIndex) \!\!\! &=& \!\!\! P_{n}(\iterIndex|\iterIndex-1) \bw_{n}(\iterIndex) \left( \sigma_{e}^{2} + \bw_{n}^{\top}(\iterIndex) P_{n}(\iterIndex|\iterIndex-1) \bw_{n}(\iterIndex)  \right)^{-1} \nonumber
\end{eqnarray}
where $G_{n}(\iterIndex)$ is the Kalman gain. We have the initial conditions
\begin{eqnarray}
\hat{\fRV}_{n}(1|0) &=& \hat{\fRV}_{n-1}(\numIter_{n-1}|\numIter_{n-1}) \nonumber \\
P_{n}(1|0) &=& P_{n-1}(\numIter_{n-1}|\numIter_{n-1}). \nonumber
\end{eqnarray}

\figurename{}~\ref{experiment:kalmanComparison} shows a comparison of the Kalman filter against our SGD based approach both with exact and mismatched parameters for the Kalman filter. Table~\ref{experiment:kalmanCompare} uses the technique from \eqref{experiment:meanCritEstTech} to estimate the \meangap{} for all three methods. The Kalman filter receives the number of samples $\numIter_{n}$ chosen by the SGD approach. With correct parameters for the Kalman filter, both methods achieve similar performance, but the SGD method is able to control its desired accuracy. With incorrect parameters, the Kalman filter's performance is considerably worse.

\begin{table}[!ht]
	\centering
	\begin{tabular}{|c|c|}
		\hline
		Method & \MeanGap{} Estimate \\
		\hline
		Direct Estimate & $1.9 \times 10^{-2} \pm 1.1 \times 10^{-3}$ \\
		\hline
		Kalman Filter & $1.8 \times 10^{-2} \pm 5.7 \times 10^{-3}$ \\
		\hline
		Kalman Filter - Mismatch & $9.4 \times 10^{-2} \pm 5.3 \times 10^{-2}$ \\
		\hline
	\end{tabular}	
	\caption{Kalman Filter Comparison}
	\label{experiment:kalmanCompare}
\end{table}

\begin{figure}[!h]
	\normalsize
	\centering
	\includegraphics[width=\myfiguresize]{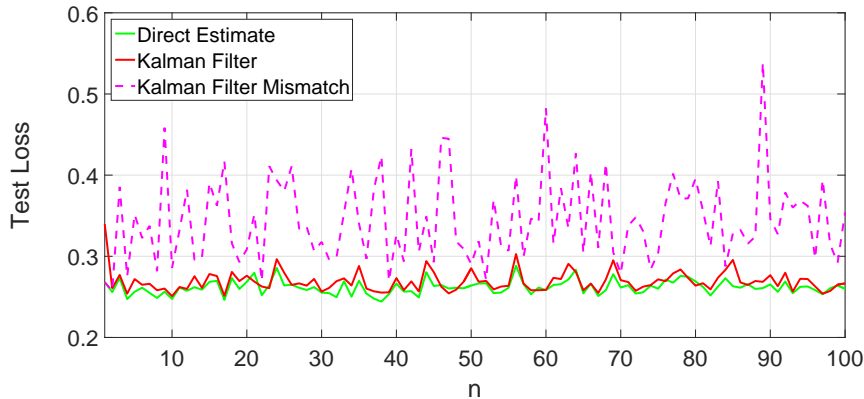}
	\caption{Comparison of Kalman filter and our approach}
	\label{experiment:kalmanComparison} \medskip
\end{figure} 

\section{Conclusion}
We have developed a framework for solving a sequence of slowly changing stochastic optimization problems to within a target accuracy in a mean sense at each time step. In an extended version of this paper \cite{Wilson2016a}, we also consider meeting the target in a high probability sense. We used an estimate of the change in the minimizers, combined with properties of the chosen optimization algorithm, to select the number of samples needed to meet a given criterion. We demonstrated through simulations that our approach works well. 

There are a number of avenues for further research in this area, including finding alternative estimation schemes for $\rho$, allowing for occasional abrupt changes in the optimizers, and incorporating a cost budget for samples used in the stochastic optimization. 

\bibliographystyle{IEEEtran}
\bibliography{TAC-paper}

\begin{thebibliography}{10}
\providecommand{\url}[1]{#1}
\csname url@samestyle\endcsname
\providecommand{\newblock}{\relax}
\providecommand{\bibinfo}[2]{#2}
\providecommand{\BIBentrySTDinterwordspacing}{\spaceskip=0pt\relax}
\providecommand{\BIBentryALTinterwordstretchfactor}{4}
\providecommand{\BIBentryALTinterwordspacing}{\spaceskip=\fontdimen2\font plus
\BIBentryALTinterwordstretchfactor\fontdimen3\font minus
  \fontdimen4\font\relax}
\providecommand{\BIBforeignlanguage}[2]{{%
\expandafter\ifx\csname l@#1\endcsname\relax
\typeout{** WARNING: IEEEtran.bst: No hyphenation pattern has been}%
\typeout{** loaded for the language `#1'. Using the pattern for}%
\typeout{** the default language instead.}%
\else
\language=\csname l@#1\endcsname
\fi
#2}}
\providecommand{\BIBdecl}{\relax}
\BIBdecl

\bibitem{Wilson2014}
C.~Wilson, V.~Veeravalli, and A.~Nedi\'c, ``Dynamic stochastic optimization,''
  in \emph{IEEE Conference on Decision and Control}, Los Angeles, USA, Dec.
  2014, pp. 173--178.

\bibitem{Wilson2015a}
C.~Wilson and V.~Veeravalli, ``Adaptive sequential optimization with
  applications to machine learning,'' in \emph{IEEE International Conference on
  Acoustics, Speech and Signal Processing}, Shanghai, China, Mar. 2016, pp.
  2642--2646.

\bibitem{Zhu16}
J.~Zhu and J.~C. Spall, ``Tracking capability of stochastic gradient algorithm
  with constant gain,'' in \emph{IEEE Conference on Decision and Control}, Las
  Vegas, USA, Dec. 2016, pp. 4522--4527.

\bibitem{Cesa2006}
N.~Cesa-Bianchi and G.~Lugosi, \emph{Prediction, {L}earning, and
  {G}ames}.\hskip 1em plus 0.5em minus 0.4em\relax New York, N.Y., USA:
  Cambridge University Press, 2006.

\bibitem{Duchi2011}
J.~Duchi, E.~Hazan, and Y.~Singer, ``Adaptive subgradient methods for online
  learning and stochastic optimization,'' \emph{Journal of Machine Learning
  Research}, vol.~12, pp. 2121--2159, Jul 2011.

\bibitem{Duchi2009}
J.~Duchi and Y.~Singer, ``Efficient online and batch learning using forward
  backward splitting,'' \emph{Journal of Machine Learning Research}, vol.~10,
  pp. 2899--2934, Dec 2009.

\bibitem{Hazan2007}
E.~Hazan, A.~Agarwal, and S.~Kale, ``Logarithmic regret algorithms for online
  convex optimization,'' \emph{Machine {L}earning}, vol.~69, no. 2--3, pp.
  169--192, Dec 2007.

\bibitem{Bartlett2008}
E.~H. P.~Bartlett and A.~Rakhlin, ``Adaptive online gradient descent,'' in
  \emph{Conference on Neural Information Processing Systems (NIPS)}, vol.~20,
  Vancouver, B.C., Canada, Dec. 2007, pp. 65--72.

\bibitem{Shwartz2009}
S.~Shalev-Shwartz and S.~Kakade, ``Mind the duality gap: Logarithmic regret
  algorithms for online optimization,'' in \emph{Conference on Neural
  Information Processing Systems (NIPS)}, vol.~21, Vancouver, B.C., Canada,
  Dec. 2009, pp. 1457--1464.

\bibitem{Shwartz2006}
S.~Shalev-Shwartz and Y.~Singer, ``Convex repeated games and {F}enchel
  duality,'' in \emph{Conference on Neural Information Processing Systems
  (NIPS)}, vol.~19, Vancouver, B.C., Canada, Dec. 2006, pp. 1265--1271.

\bibitem{Shwartz2007}
------, ``Logarithmic regret algorithms for strongly convex repeated games,''
  The Hebrew University, Tech. Rep., May 2007.

\bibitem{Xiao2010}
L.~Xiao, ``Dual averaging methods for regularized stochastic learning and
  online optimization,'' \emph{Journal of Machine Learning Research}, vol.~11,
  pp. 2543--2596, Mar. 2010.

\bibitem{Zinkevich2003}
M.~Zinkevich, ``Online convex programming and generalized infinitesimal
  gradient ascent,'' in \emph{International Conference on Machine Learning},
  Washington D.C., USA, Aug. 2003, pp. 928--936.

\bibitem{RakhlinSridharan2012}
A.~{Rakhlin} and K.~{Sridharan}, ``{Online Learning with Predictable
  Sequences},'' \emph{arXiv:1208.3728}, Aug. 2012.

\bibitem{Yang2012}
C.-K. Chiang, T.~Yang, C.-J. Lee, M.~Mahdavi, C.-J. Lu, R.~Jin, and S.~Zhu,
  ``Online optimization with gradual variations,'' in \emph{Conference on
  Learning Theory}, vol.~23, Edinburgh, Scotland, Jun. 2012, pp. 6.1--6.20.

\bibitem{Dontchev2009}
A.~Dontchev and R.~Rockafellar, \emph{Implicit Functions and Solution Mappings:
  A View from Variational Analysis}.\hskip 1em plus 0.5em minus 0.4em\relax New
  York, New York: Springer, 2009.

\bibitem{TYS2010}
N.~Takahashi, I.~Yamada, and A.~Sayed, ``Diffusion least-mean squares with
  adaptive combiners: {F}ormulation and performance analysis,'' \emph{IEEE
  Transactions on Signal Processing}, vol.~58, no.~9, pp. 4795--4810, Jun.
  2010.

\bibitem{YO2005}
I.~Yamada and N.~Ogura, ``Adaptive projected subgradient method for asymptotic
  minimization of sequence of nonnegative convex functions,'' \emph{Numerical
  Functional Analysis and Optimization}, vol.~25, no. 7--8, pp. 593--617, Aug.
  2005.

\bibitem{SYO2006}
K.~Slavakis, I.~Yamada, and N.~Ogura, ``The adaptive projected subgradient
  method over the fixed point set of strongly attracting nonexpansive
  mappings,'' \emph{Numerical functional analysis and optimization}, vol.~27,
  no. 7--8, pp. 905--930, Nov. 2006.

\bibitem{Kushner1994}
H.~Kushner and J.~Yang, ``Analysis of adaptive step size {SA} algorithms for
  parameter tracking,'' in \emph{IEEE Conference on Decision and Control}, Lake
  Buena Vista, USA, Dec. 1994, pp. 730--737.

\bibitem{Solo1995}
V.~Solo and X.~Kong, \emph{Adaptive Signal Processing Algorithms: Stability and
  Performance}.\hskip 1em plus 0.5em minus 0.4em\relax Englewood Cliffs, NJ:
  Prentice Hall, 1995.

\bibitem{Sayed2008}
A.~Sayed, \emph{Adaptive Filters}.\hskip 1em plus 0.5em minus 0.4em\relax
  Hoboken, New Jersey, USA: Wiley \& Sons, Inc., 2008.

\bibitem{Doucet2009}
A.~Doucet and A.~M. Johansen, ``A tutorial on particle filtering and smoothing:
  Fifteen years later,'' \emph{Handbook of nonlinear filtering}, vol.~12,
  no.~3, pp. 656--704, 2009.

\bibitem{Mueller1997}
A.~M\"{u}ller, ``{Integral Probability Metrics and Their Generating Classes of
  Functions},'' \emph{Advances in Applied Probability}, vol.~29, no.~2, pp.
  429--443, Jul. 1997.

\bibitem{Rudin1964}
W.~Rudin, \emph{Principles of Mathematical Analysis}.\hskip 1em plus 0.5em
  minus 0.4em\relax McGraw-Hill New York, 1964.

\bibitem{Boyd2004}
S.~Boyd and L.~Vandenberghe, \emph{Convex Optimization}.\hskip 1em plus 0.5em
  minus 0.4em\relax New York, NY, USA: Cambridge University Press, 2004.

\bibitem{Antonini2005}
R.~Antonini and Y.~Kozachenko, ``A note on the asymptotic behavior of sequences
  of generalized {sub-Gaussian} random vectors,'' \emph{Random Op. and Stoch.
  Equ.}, vol.~13, no.~1, pp. 39--52, Jan. 2005.

\bibitem{Wilson2016a}
C.~Wilson, V.~Veeravalli, and A.~Nedi\'c, ``Adaptive sequential stochastic
  optimization,'' \emph{arXiv:1610.01970}, Oct. 2016.

\bibitem{David03}
H.~A. David, \emph{Order Statistics}, 3rd~ed.\hskip 1em plus 0.5em minus
  0.4em\relax Wiley, 2003.

\bibitem{Haykin2002}
S.~Haykin, \emph{Adaptive Filter Theory}.\hskip 1em plus 0.5em minus
  0.4em\relax Springer, 2002.

\bibitem{Boucheron13}
S.~Boucheron, G.~Lugosi, and P.~Massart, \emph{Concentration Inequalities: A
  Nonasymptotic Theory of Independence}.\hskip 1em plus 0.5em minus 0.4em\relax
  Oxford University Press, 2013.

\bibitem{Kennan2001}
J.~Kennan, ``Uniqueness of positive fixed points for increasing concave
  functions on {R}n: An elementary result,'' \emph{Review of Economic
  Dynamics}, vol.~4, no.~4, pp. 893--899, Oct. 2001.

\bibitem{Granas2003}
A.~Granas and J.~Dugundji, \emph{Fixed Point Theory}.\hskip 1em plus 0.5em
  minus 0.4em\relax Springer-Verlag, 2003.

\bibitem{Nemirovski2009}
A.~Nemirovski, A.~Juditsky, G.~Lan, and A.~Shapiro, ``Stochastic approximation
  approach to stochastic programming,'' \emph{SIAM Journal on Optimization},
  vol.~19, pp. 1574--1609, 2009.

\bibitem{BachMoulines2011}
F.~Bach and E.~Moulines, ``\BIBforeignlanguage{English}{{Non-Asymptotic
  Analysis of Stochastic Approximation Algorithms for Machine Learning}},'' in
  \emph{\BIBforeignlanguage{English}{{Advances in Neural Information Processing
  Systems (NIPS)}}}, Spain, 2011.

\bibitem{Bertsekas1999}
D.~Bertsekas, \emph{Nonlinear Programming}.\hskip 1em plus 0.5em minus
  0.4em\relax Athena Scientific, 1999.

\bibitem{NedicLee12}
A.~Nedic and S.~Lee, ``Analysis of mirror descent for strongly convex
  functions,'' \emph{ArXiV}, 2013.

\bibitem{NesterovBook2004}
Y.~Nesterov, \emph{Introductory Lectures on Convex Optimization: A Basic
  Course}.\hskip 1em plus 0.5em minus 0.4em\relax Norwell, Massachusetts, USA:
  Kluwer Academic Publishers, 2004.

\end{thebibliography}

\section{Proofs for Estimates of Change in Minimizers}
\label{usefulConcIneq} 

For our analysis of minimizer change estimation, we need to introduce a few results for sub-Gaussian random variables including the following key technical lemma from \cite{Antonini2005}. This lemma controls the concentration of sums of random variables that are sub-Gaussian conditioned on a particular filtration $\{\mathcal{F}_{i}\}_{i=0}^{n}$. Such a collection of random variables is referred to as a \emph{sub-Gaussian martingale sequence}.

\begin{lem}[Theorem 7.5 of \cite{Antonini2005}]
	\label{subgauss:subgaussDepLem}
	Suppose we have a collection of random variables $\{V_{i}\}_{i=1}^{n}$ and a filtration $\{\mathcal{F}_{i}\}_{i=0}^{n}$ such that for each random variable $V_{i}$ it holds that
	\begin{enumerate}
		\item $\mathbb{E}\left[ \exp\left\{ s \left( V_{i} - \mathbb{E}\left[ V_{i}  \;\big|\; \mathcal{F}_{i-1}  \right]\right) \right\} \;\big|\; \mathcal{F}_{i-1}  \right] \leq e^{\frac{1}{2}\sigma_{i}^{2}s^{2}}$ with $\sigma_{i}^{2}$ a constant.
		\item $V_{i}$ is $\mathcal{F}_{i}$-measurable.
	\end{enumerate}
	Then for every $\bm{a} \in \mathbb{R}^{n}$ it holds that
	\[
	\mathbb{P}\left\{ \sum_{i=1}^{n} a_{i} V_{i} > \sum_{i=1}^{n} a_{i} \mathbb{E}\left[ V_{i}  \;\big|\; \mathcal{F}_{i-1}  \right] + t \right\} \leq \exp\left\{ - \frac{t^{2}}{2 \nu} \right\}
	\]
	with $\nu = \sum_{i=1}^{n} \sigma_{i}^{2} a_{i}^{2}$. The other tail is similarly bounded.
\end{lem}

If we can upper bound the conditional expectations $\mathbb{E}\left[ V_{i}  \;\big|\; \mathcal{F}_{i-1}  \right] \leq C_{i}$ by \\$\mathcal{F}_{i-1}$-measurable random variables $C_{i}$, then we have
\iftoggle{useTwoColumn}{	
	\begin{align}
	\mathbb{P}\left\{ \sum_{i=1}^{n} a_{i} V_{i} > \sum_{i=1}^{n} a_{i} C_{i} + t \right\} \leq \exp\left\{ - \frac{t^{2}}{2 \nu} \right\}. \nonumber
	\end{align}	
}{
\begin{equation*}
\mathbb{P}\left\{ \sum_{i=1}^{n} a_{i} V_{i} > \sum_{i=1}^{n} a_{i} C_{i} + t \right\}  \leq \mathbb{P}\left\{ \sum_{i=1}^{n} a_{i} V_{i} > \sum_{i=1}^{n} a_{i} \mathbb{E}\left[ V_{i}  \;\big|\; \mathcal{F}_{i-1}  \right] + t \right\} \leq \exp\left\{ - \frac{t^{2}}{2 \nu} \right\}.
\end{equation*}
}
For our analysis, we generally cannot compute $ \mathbb{E}\left[ V_{i}  \;\big|\; \mathcal{F}_{i-1}  \right]$, but we can find ``nice'' $C_{i}$.

To find $\sigma^{2}_{i}$ for use in \cref{subgauss:subgaussDepLem}, we employ the following conditional version of Hoeffding's Lemma.

\begin{lem}[Conditional Hoeffding's Lemma]
	\label{estRho:condHoeffdingLemma}
	If a random variable $V$ and a sigma algebra $\mathcal{F}$ satisfy $a \leq V \leq b$ and $\mathbb{E}[V|\mathcal{F}] = 0$, then
	\[
	\mathbb{E}\left[ e^{sV} \;|\; \mathcal{F} \right] \leq \exp\left\{ \frac{1}{8} (b-a)^{2} s^{2} \right\}.
	\]
\end{lem}
\begin{proof}
	Follows from the standard proof of Hoeffding's Lemma from \cite{Boucheron13}.
\end{proof}

Using these tools, we can analyze averages of the direct estimate. We focus on the proof of \cref{rho_conc_eq} and \cref{rho_conc_ineq} as the proofs of \cref{rho_conc_eq_sq} and \cref{rho_conc_ineq_sq} are simple extensions.

\subsection{Euclidean Norm Condition}

As a reminder, we consider running our optimization algorithm used to generate $\bx_{i}$ again using independent samples $\{\tilde{\bz}_{i}(\iterIndex)\}_{\iterIndex=1}^{\numIter_{i}}$ to yield a second approximate minimizer $\tilde{\bx}_{i}$. For SGD, the process to do this is summarized in \cref{estRho:secondSampleSGD}. We connect $\tilde{\rho}_{i}$ to $\tilde{\rho}_{i}^{(2)}$  with $\tilde{\rho}_{i}^{(2)}$ defined in \cref{estRho:rhoTwoEstimate}.

\begin{proof}[Proof of \cref{rho_conc_eq}]
	To proceed, we compare the three single step estimates:
	\begin{enumerate}
		\item $\tilde{\rho}_{i} = \| \bx_{i} - \bx_{i-1} \|_{2} + \frac{1}{m} \| G_{i} \|_{2} + \frac{1}{m} \| G_{i-1}\|_{2}$
		\item $\tilde{\rho}_{i}^{(2)} = \| \tilde{\bx}_{i} - \tilde{\bx}_{i-1} \|_{2} + \frac{1}{m} \| \tilde{G}_{i} \|_{2} + \frac{1}{m} \| \tilde{G}_{i-1}\|_{2}$
		\item $\tilde{\rho}_{i}^{(3)} = \| \tilde{\bx}_{i} - \tilde{\bx}_{i-1} \|_{2} + \frac{1}{m} \| \nabla f_{i}(\tilde{\bx}_{i}) \|_{2} + \frac{1}{m} \| \nabla f_{i-1}(\tilde{\bx}_{i-1}) \|_{2}$
	\end{enumerate}
	where
	\[
	\hat{G}_{i} = \frac{1}{\numIter_{i}} \sum_{\iterIndex=1}^{\numIter_{i}} \nabla_{\bx} \lossFunc(\bx_{i},\bz_{i}(\iterIndex))
	\]
	and
	\[
	\tilde{G}_{i} = \frac{1}{\numIter_{i}} \sum_{\iterIndex=1}^{\numIter_{i}} \nabla_{\bx} \lossFunc(\tilde{\bx}_{i},\bz_{i}(\iterIndex)).
	\]
	Define $\hat{\rho}_{n}^{(2)}$ and $\hat{\rho}_{n}^{(3)}$ analogously to $\hat{\rho}_{n}$ as an average of the relevant single step estimates. 

	Using the triangle inequality and the reverse triangle inequality, it holds that
	\iftoggle{useTwoColumn}{
		\begin{align}
		&| \hat{\rho}_{n} - \hat{\rho}_{n}^{(3)}| \nonumber \\
		&\quad = | \hat{\rho}_{n} - \hat{\rho}_{n}^{(2)} + \hat{\rho}_{n}^{(2)} - \hat{\rho}_{n}^{(3)}| \nonumber \\
		&\quad \leq | \hat{\rho}_{n} - \hat{\rho}_{n}^{(2)}| + |\hat{\rho}_{n}^{(2)} - \hat{\rho}_{n}^{(3)}| \nonumber \\
		&\quad \leq \frac{1}{n-1} \sum_{i=2}^{n} \left( \| \bx_{i} - \tilde{\bx}_{i}\|_{2} + \| \bx_{i-1} - \tilde{\bx}_{i-1}\|_{2} + \frac{1}{m} \| \hat{G}_{i} - \tilde{G}_{i}\|_{2} \right. \nonumber \\
		&\qquad \qquad \qquad \qquad \qquad \left. + \frac{1}{m} \| \hat{G}_{i-1} - \tilde{G}_{i-1}\|_{2}  \right) \nonumber \\
		&\quad \quad  + \frac{1}{n-1} \sum_{i=2}^{n} \left( \frac{1}{m} \| \tilde{G}_{i} - \nabla f_{i}(\tilde{\bx}_{i})\|_{2} \right. \nonumber \\
		&\left. \quad \quad \quad \quad \quad \quad  + \frac{1}{m} \| \tilde{G}_{i-1} - \nabla f_{i-1}(\tilde{\bx}_{i-1})\|_{2} \right) \nonumber \\
		&\quad \leq \frac{1}{n-1} \left( \| \bx_{1} - \tilde{\bx}_{1}\|_{2} + 2 \sum_{i=2}^{n-1}  \| \bx_{i} - \tilde{\bx}_{i}\|_{2} + \| \bx_{n} - \tilde{\bx}_{n}\|_{2} \right) \nonumber \\
		&\quad \quad \quad  + \frac{1}{m(n-1)} \left( \| \hat{G}_{1} - \tilde{G}_{1} \|_{2} \right. \nonumber \\
		&\left. \quad \quad \quad \quad \quad + 2 \sum_{i=2}^{n-1}  \| \hat{G}_{i} - \tilde{G}_{i} \|_{2} + \|\hat{G}_{n} - \tilde{G}_{n}\|_{2} \right) \nonumber \\
		&\quad \quad \quad  + \frac{1}{m(n-1)} \left( \| \tilde{G}_{1} - \nabla f_{1}(\tilde{\bx}_{1}) \|_{2} + 2 \sum_{i=2}^{n-1}  \| \tilde{G}_{i} - \nabla f_{i}(\tilde{\bx}_{i}) \|_{2} \right. \nonumber \\
		&\qquad \qquad \qquad \qquad \qquad \left. + \|\tilde{G}_{n} - \nabla f_{n}(\tilde{\bx}_{n})\|_{2} \right). \nonumber
		\end{align}	
	}{
	\begin{eqnarray}
	| \hat{\rho}_{n} - \hat{\rho}_{n}^{(3)}| &=& | \hat{\rho}_{n} - \hat{\rho}_{n}^{(2)} + \hat{\rho}_{n}^{(2)} - \hat{\rho}_{n}^{(3)}| \nonumber \\
	&\leq& | \hat{\rho}_{n} - \hat{\rho}_{n}^{(2)}| + |\hat{\rho}_{n}^{(2)} - \hat{\rho}_{n}^{(3)}| \nonumber \\
	&\leq& \frac{1}{n-1} \sum_{i=2}^{n} \left( \| \bx_{i} - \tilde{\bx}_{i}\|_{2} + \| \bx_{i-1} - \tilde{\bx}_{i-1}\|_{2} + \frac{1}{m} \| \hat{G}_{i} - \tilde{G}_{i}\|_{2} + \frac{1}{m} \| \hat{G}_{i-1} - \tilde{G}_{i-1}\|_{2}   \right) \nonumber \\
	&& \;\;\;\;\; + \frac{1}{n-1} \sum_{i=2}^{n} \left( \frac{1}{m} \| \tilde{G}_{i} - \nabla f_{i}(\tilde{\bx}_{i})\|_{2} + \frac{1}{m} \| \tilde{G}_{i-1} - \nabla f_{i-1}(\tilde{\bx}_{i-1})\|_{2} \right) \nonumber \\
	&\leq& \frac{1}{n-1} \left( \| \bx_{1} - \tilde{\bx}_{1}\|_{2} + 2 \sum_{i=2}^{n-1}  \| \bx_{i} - \tilde{\bx}_{i}\|_{2} + \| \bx_{n} - \tilde{\bx}_{n}\|_{2} \right) \nonumber \\
	&& \;\;\;\;\;\;\;\;\;\;\;\;\;\; + \frac{1}{m(n-1)} \left( \| \hat{G}_{1} - \tilde{G}_{1} \|_{2} + 2 \sum_{i=2}^{n-1}  \| \hat{G}_{i} - \tilde{G}_{i} \|_{2} + \|\hat{G}_{n} - \tilde{G}_{n}\|_{2} \right) \nonumber \\
	&& \;\;\;\;\;\;\;\;\;\;\;\;\;\; + \frac{1}{m(n-1)} \left( \| \tilde{G}_{1} - \nabla f_{1}(\tilde{\bx}_{1}) \|_{2} + 2 \sum_{i=2}^{n-1}  \| \tilde{G}_{i} - \nabla f_{i}(\tilde{\bx}_{i}) \|_{2} + \|\tilde{G}_{n} - \nabla f_{n}(\tilde{\bx}_{n})\|_{2} \right). \nonumber
	\end{eqnarray}
}
Define 
\[
U_{n} = \frac{1}{n-1} \left( \| \bx_{1} - \tilde{\bx}_{1}\|_{2} + 2 \sum_{i=2}^{n-1}  \| \bx_{i} - \tilde{\bx}_{i}\|_{2} + \| \bx_{n} - \tilde{\bx}_{n}\|_{2} \right)
\]
and
\[
V_{n} = \frac{1}{m(n-1)} \left( \| \hat{G}_{1} - \tilde{G}_{1} \|_{2} + 2 \sum_{i=2}^{n-1}  \| \hat{G}_{i} - \tilde{G}_{i} \|_{2} + \|\hat{G}_{n} - \tilde{G}_{n}\|_{2} \right)
\]
and
\iftoggle{useTwoColumn}{
	\begin{align}
	W_{n} = \frac{1}{m(n-1)} &\left( \| \tilde{G}_{1} - \nabla f_{1}(\tilde{\bx}_{1}) \|_{2} + 2 \sum_{i=2}^{n-1}  \| \tilde{G}_{i} - \nabla f_{i}(\tilde{\bx}_{i}) \|_{2} \right. \nonumber \\
	&\qquad \qquad \qquad \left. + \|\tilde{G}_{n} - \nabla f_{n}(\tilde{\bx}_{n})\|_{2} \right). \nonumber
	\end{align}
}{
\[
W_{n} = \frac{1}{m(n-1)} \left( \| \tilde{G}_{1} - \nabla f_{1}(\tilde{\bx}_{1}) \|_{2} + 2 \sum_{i=2}^{n-1}  \| \tilde{G}_{i} - \nabla f_{i}(\tilde{\bx}_{i}) \|_{2} + \|\tilde{G}_{n} - \nabla f_{n}(\tilde{\bx}_{n})\|_{2} \right).
\]
}
Then it holds that
\[
| \hat{\rho}_{n} - \hat{\rho}_{n}^{(3)}| \leq U_{n} + V_{n} + W_{n}.
\]
Now, we look at bounding $\mathbb{E}[\|\bx_{i} - \bx_{i-1}\|_{2} \;|\; \mathcal{F}_{i-1}]$, $\mathbb{E}[\|\hat{G}_{i} - \tilde{G}_{i}\|_{2} \;|\; \mathcal{F}_{i-1}]$, \\and $\mathbb{E}[\|\tilde{G}_{i} - \nabla f_{i}(\tilde{\bx})_{i}\|_{2} \;|\; \mathcal{F}_{i-1}]$. First, by assumption~\ref{probState:assumpB1}, it holds that
\[
\mathbb{E}[\|\bx_{i} - \tilde{\bx}_{i}\|_{2} \;|\; \mathcal{F}_{i-1}] \leq C_{i}.
\]
Second, it holds that
\iftoggle{useTwoColumn}{
	\begin{align}
	\mathbb{E}&\left[ \|\hat{G}_{i} - \tilde{G}_{i}\|_{2} \;|\; \mathcal{F}_{i-1} \right] \nonumber \\
	&\quad \leq \frac{1}{\numIter_{i}}\sum_{\iterIndex=1}^{\numIter_{i}} \mathbb{E}\left[ \| \sgrad{\bx_{i}}{\bz_{i}(\iterIndex)}{i} - \sgrad{\tilde{\bx}_{i}}{\bz_{i}(\iterIndex)}{i} \|_{2} \;|\; \mathcal{F}_{i-1} \right] \nonumber \\
	&\quad = \mathbb{E}\left[ \| \sgrad{\bx_{i}}{\bz_{i}(1)}{i} - \sgrad{\tilde{\bx}_{i}}{\bz_{i}(1)}{i} \|_{2} \;|\; \mathcal{F}_{i-1} \right] \nonumber \\
	&\quad \leq M \mathbb{E}\left[ \| \bx_{i} - \tilde{\bx}_{i} \|_{2} \;|\; \mathcal{F}_{i-1}  \right]  \nonumber \\
	&\quad \leq M C_{i}. \nonumber
	\end{align}	
}{
\begin{eqnarray}
\mathbb{E}\left[ \|\hat{G}_{i} - \tilde{G}_{i}\|_{2} \;|\; \mathcal{F}_{i-1} \right] &\leq& \frac{1}{\numIter_{i}}\sum_{\iterIndex=1}^{\numIter_{i}} \mathbb{E}\left[ \| \sgrad{\bx_{i}}{\bz_{i}(\iterIndex)}{i} - \sgrad{\tilde{\bx}_{i}}{\bz_{i}(\iterIndex)}{i} \|_{2} \;|\; \mathcal{F}_{i-1} \right] \nonumber \\
&=& \mathbb{E}\left[ \| \sgrad{\bx_{i}}{\bz_{i}(1)}{i} - \sgrad{\tilde{\bx}_{i}}{\bz_{i}(1)}{i} \|_{2} \;|\; \mathcal{F}_{i-1} \right] \nonumber \\
&\leq& M \mathbb{E}\left[ \| \bx_{i} - \tilde{\bx}_{i} \|_{2} \;|\; \mathcal{F}_{i-1}  \right]  \nonumber \\
&\leq& M C_{i}. \nonumber
\end{eqnarray}
}
Third, it holds that
\iftoggle{useTwoColumn}{
	\begin{align}
	\mathbb{E}&\left[  \| \tilde{G}_{i} - \nabla f_{i}(\tilde{\bx}_{i}) \|_{2} \;|\; \mathcal{F}_{i-1}  \right] \nonumber \\
	&\quad \leq \left( \mathbb{E}\left[ \Bigg\| \frac{1}{\numIter_{i}} \sum_{\iterIndex=1}^{\numIter_{i}} \left( \sgrad{\tilde{\bx}_{i}}{\bz_{i}(\iterIndex)}{i} - \nabla f_{i}(\tilde{\bx}_{i})  \right)    \Bigg\|_{2}^{2} \;\Bigg|\; \mathcal{F}_{i-1} \right] \right)^{1/2} \nonumber \\
	&\quad \leq \left( \mathbb{E}\left[ \frac{1}{\numIter_{i}^{2}} \sum_{\iterIndex=1}^{\numIter_{i}} \| \sgrad{\tilde{\bx}_{i}}{\bz_{i}(\iterIndex)}{i} - \nabla f_{i}(\tilde{\bx}_{i}) \|_{2}^{2}   \;\Bigg|\; \mathcal{F}_{i-1} \right] \right)^{1/2} \nonumber \\
	&\quad \leq \left( \frac{\sigma}{\numIter_{i}} \right)^{1/2}. \nonumber
	\end{align}	
}{
\begin{eqnarray}
\mathbb{E}\left[  \| \tilde{G}_{i} - \nabla f_{i}(\tilde{\bx}_{i}) \|_{2} \;|\; \mathcal{F}_{i-1}  \right] &\leq& \left( \mathbb{E}\left[ \Bigg\| \frac{1}{\numIter_{i}} \sum_{\iterIndex=1}^{\numIter_{i}} \left( \sgrad{\tilde{\bx}_{i}}{\bz_{i}(\iterIndex)}{i} - \nabla f_{i}(\tilde{\bx}_{i})  \right)    \Bigg\|_{2}^{2} \;\Bigg|\; \mathcal{F}_{i-1} \right] \right)^{1/2} \nonumber \\
&\leq& \left( \mathbb{E}\left[ \frac{1}{\numIter_{i}^{2}} \sum_{\iterIndex=1}^{\numIter_{i}} \| \sgrad{\tilde{\bx}_{i}}{\bz_{i}(\iterIndex)}{i} - \nabla f_{i}(\tilde{\bx}_{i}) \|_{2}^{2}   \;\Bigg|\; \mathcal{F}_{i-1} \right] \right)^{1/2} \nonumber \\
&\leq& \left( \frac{\sigma}{\numIter_{i}} \right)^{1/2}. \nonumber
\end{eqnarray}
}
The resulting bounds on the expectation of $U_{n}$, $V_{n}$, and $W_{n}$ denoted $\bar{U}_{n}$, $\bar{V}_{n}$, and $\bar{W}_{n}$ are as follows:
\begin{enumerate}
	\item $\bar{U}_{n} = \frac{1}{n-1} \left( C_{1} + 2 \sum_{i=2}^{n-1}  C_{i} + C_{n} \right)
	$
	\item $\bar{V}_{n} = \frac{M}{m(n-1)} \left( C_{1}  + 2 \sum_{i=2}^{n-1}  C_{i} + C_{n} \right)$
	\item $\bar{W}_{n} = \frac{1}{m(n-1)} \left( \left( \frac{\sigma}{\numIter_{1}} \right)^{1/2} + 2 \sum_{i=2}^{n-1}  \left( \frac{\sigma}{\numIter_{1}} \right)^{1/2} + \left( \frac{\sigma}{\numIter_{n}} \right)^{1/2} \right)$.
\end{enumerate}

Then it holds that
\iftoggle{useTwoColumn} {
	\begin{align}
	\mathbb{P}&\left\{ |\hat{\rho}_{n} - \hat{\rho}_{n}^{(3)}| > (\bar{U}_{n} + \bar{V}_{n} + \bar{W}_{n}) + t_{n} \right\} \nonumber \\
	&\quad \leq \mathbb{P}\left\{ U_{n} + V_{n} + W_{n} > (\bar{U}_{n} + \bar{V}_{n} + \bar{W}_{n}) + t_{n} \right\} \nonumber \\
	&\quad \leq \mathbb{P}\left\{ U_{n} > \bar{U}_{n} + \frac{1}{3}t_{n} \right\} + \mathbb{P}\left\{  V_{n} > \bar{V}_{n} + \frac{1}{3}t_{n} \right\} \nonumber \\
	&\qquad \qquad \qquad \qquad + \mathbb{P}\left\{ W_{n} > \bar{W}_{n} + \frac{1}{3}t_{n} \right\}. \nonumber
	\end{align}	
}{
\begin{eqnarray}
\mathbb{P}\left\{ \hat{|\rho}_{n} - \hat{\rho}_{n}^{(3)}| > (\bar{U}_{n} + \bar{V}_{n} + \bar{W}_{n}) + t_{n} \right\} &\leq& \mathbb{P}\left\{ U_{n} + V_{n} + W_{n} > (\bar{U}_{n} + \bar{V}_{n} + \bar{W}_{n}) + t_{n} \right\} \nonumber \\
&\leq& \mathbb{P}\left\{ U_{n} > \bar{U}_{n} + \frac{1}{3}t_{n} \right\} + \mathbb{P}\left\{  V_{n} > \bar{V}_{n} + \frac{1}{3}t_{n} \right\} + \mathbb{P}\left\{ W_{n} > \bar{W}_{n} + \frac{1}{3}t_{n} \right\}. \nonumber
\end{eqnarray}
}
Now, we bound each of these three probabilities using \cref{subgauss:subgaussDepLem}. First, we have 
\[
0 \leq \|\bx_{i} - \tilde{\bx}_{i}\|_{2} \leq \text{diam}(\xSp)
\]
so applying \cref{estRho:condHoeffdingLemma} and \cref{subgauss:subgaussDepLem} with $\sigma_{i}^{2} = \frac{1}{4}\text{diam}^{2}(\xSp)$ and 
\begin{eqnarray}
a_{1} &=& a_{n} = \frac{1}{n-1} \nonumber \\
a_{2} &=& \cdots = a_{n-2} = \frac{2}{n-1} \nonumber 
\end{eqnarray}
yields
\begin{eqnarray}
\nu_{U} &=& \frac{1}{4} \text{diam}^{2}(\xSp) \sum_{i=1}^{n} a_{i}^{2} \nonumber \\
&=& \frac{1}{4} \text{diam}^{2}(\xSp) \left(\frac{1}{n-1}\right)^{2} + \sum_{i=2}^{n-1} \left(\frac{2}{n-1}\right)^{2} + \left(\frac{1}{n-1}\right)^{2} \nonumber \\
&\leq& \frac{1}{n-1}\text{diam}^{2}(\xSp).
\end{eqnarray}
Therefore, it holds that
\begin{eqnarray}
\mathbb{P}\left\{ U_{n} > \bar{U}_{n} + \frac{1}{3}t_{n} \right\} &\leq& \exp\left\{ -\frac{(t_{n}/3)^{2}}{2\nu_{U}}  \right\} \nonumber \\
&=& \exp\left\{ - \frac{(n-1)t_{n}^{2}}{18 \text{diam}^{2}(\xSp)} \right\}. \nonumber
\end{eqnarray}
Since
\[
0 \leq \|\hat{G}_{i} - \tilde{G}_{i}\|_{2} \leq 2G
\]
and
\[
0 \leq \| \tilde{G}_{i} - \nabla f_{i}(\tilde{\bx}_{i})\|_{2} \leq 2G
\]
we can apply \cref{estRho:condHoeffdingLemma} and \cref{subgauss:subgaussDepLem} to $V_{n}$ and $W_{n}$ to yield
\iftoggle{useTwoColumn}{
\begin{align}
\mathbb{P}\left\{ V_{n} > \bar{V}_{n} + \frac{1}{3}t_{n} \right\} &\leq \exp\left\{ -\frac{(t_{n}/3)^{2}}{2\nu_{V}}  \right\} \nonumber \\
& = \exp\left\{ - \frac{m^2(n-1)t_{n}^{2}}{72 G^{2}} \right\} \nonumber
\end{align}
}{
\begin{equation*}
\mathbb{P}\left\{ V_{n} > \bar{V}_{n} + \frac{1}{3}t_{n} \right\} \leq \exp\left\{ -\frac{(t_{n}/3)^{2}}{2\nu_{V}}  \right\} = \exp\left\{ - \frac{m^2(n-1)t_{n}^{2}}{72 G^{2}} \right\}
\end{equation*}
}
and similarly
\begin{equation*}
\mathbb{P}\left\{ W_{n} > \bar{W}_{n} + \frac{1}{3}t_{n} \right\} \leq \exp\left\{ - \frac{m^2(n-1)t_{n}^{2}}{72 G^{2}} \right\}.
\end{equation*}
Define
\begin{eqnarray}
D_{n} &=& \bar{U}_{n} + \bar{V}_{n} + \bar{W}_{n} \nonumber
\end{eqnarray}
which is the definition in \cref{rho_conc_eq:Dn}. It follows that
\iftoggle{useTwoColumn}{
	\begin{align}
	\mathbb{P}&\left\{ \hat{\rho}_{n} < \hat{\rho}_{n}^{(3)} - D_{n} - t_{n} \right\} \nonumber \\
	&\qquad \leq \exp\left\{ - \frac{(n-1)t_{n}^{2}}{18 \text{diam}^{2}(\xSp)}\right\} + 2\exp\left\{ - \frac{m^2(n-1)t_{n}^{2}}{72 G^{2}} \right\}.
	\end{align}
}{
\[
\mathbb{P}\left\{ \hat{\rho}_{n} < \hat{\rho}_{n}^{(3)} - D_{n} - t_{n} \right\} \leq \exp\left\{ - \frac{(n-1)t_{n}^{2}}{18 \text{diam}^{2}(\xSp)}\right\} + 2\exp\left\{ - \frac{m^2(n-1)t_{n}^{2}}{72 G^{2}} \right\}.
\]
}
Then it follows that
	\begin{align}
	\sum_{n=2}^{\infty} \mathbb{P}&\left\{ \hat{\rho}_{n} < \hat{\rho}_{n}^{(3)} - D_{n} - t_{n} \right\} \nonumber \\
	&\;\;  \leq\sum_{n=2}^{\infty} \left( \exp\left\{ - \frac{(n-1)t_{n}^{2}}{18 \text{diam}^{2}(\xSp)}\right\} \right. \nonumber \\
	&\qquad\qquad \left. + 2\exp\left\{ - \frac{m^2(n-1)t_{n}^{2}}{72 G^{2}} \right\} \right) \nonumber \\
	&\;\;  < +\infty. \nonumber
	\end{align}
Therefore, by the Borel-Cantelli Lemma, for all $n$ large enough it holds that
\[
\hat{\rho}_{n} + D_{n} + t_{n} \geq \hat{\rho}_{n}^{(3)}
\]
almost surely. Finally, by \Cref{dir_est_deriv}, it holds that \an{$\hat{\rho}_{n}^{(3)} \geq\rho$}, which proves the result.
\end{proof}

Looking at the form of $D_{i}$, it follows that in this case
\[
D_{n} = \mathcal{O}\left(  \frac{1}{n-1} \sum_{i=1}^{n} \frac{1}{\sqrt{\numIter_{i}}}  \right).
\]
In the case where $\numIter_{i} = \numIter^{*}$, this implies that
\[
D_{n} = \mathcal{O}\left(  \frac{1}{\sqrt{\numIter^{*}}} \right).
\]

In \Cref{usefulConcIneqPartTwo}\iftoggle{useArxiv}{}{ of an extended version of this paper \cite{Wilson2016a}}, we finish the proofs for the inequality condition on $\rho$ and the $L_{2}$ condition for equality and inequality. The arguments are similar to those employed here.

\subsection{Proofs for Estimates of Change in Minimizers Under The Inequality Condition}
\label{usefulConcIneqPartTwo}

We can also prove the result for the inequality constraint of \cref{prob_form:opt_change_L2} using similar techniques in \cref{rho_conc_ineq}.
\begin{proof}[Proof of \cref{rho_conc_ineq}]
	Define $\bar{\rho}_{i}^{(2)}$ and $\bar{\rho}_{i}^{(3)}$ analogous to the equality case proof and the pair $\hat{\rho}_{i}^{(2)}$ and $\hat{\rho}_{i}^{(3)}$ of the form in \cref{ineqCond:basicEst}. First, we have by assumptions~\ref{probState:assumpB4}-\ref{probState:assumpB5}
	\iftoggle{useTwoColumn}{
		\begin{align}
		&|\hat{\rho}_{n} - \hat{\rho}_{n}^{(3)}| \nonumber \\
		&\quad \leq \frac{1}{n-W} \sum_{i=W+1}^{n} |\bar{\rho}^{(i)} - \bar{\rho}^{(i)}_{3}| \nonumber \\
		&\quad \leq \frac{1}{n-W} \sum_{i=W+1}^{n} \sum_{j = i-W+1}^{i} a_{j} |\tilde{\rho}_{j} - \tilde{\rho}_{j}^{(3)}| \nonumber \\
		&\quad \leq \frac{1}{n-W} \sum_{i=W+1}^{n} \sum_{j = i-W+1}^{i} a_{j} \left( |\tilde{\rho}_{j} - \tilde{\rho}_{j}^{(2)}| +  |\tilde{\rho}_{j}^{(2)} - \tilde{\rho}_{j}^{(3)}| \right) \nonumber  \\
		&\quad \leq \frac{\sum_{j = 1}^{W} a_{j}}{n-W} \sum_{i=2}^{n}  \left( |\tilde{\rho}_{i} - \tilde{\rho}_{i}^{(2)}| +  |\tilde{\rho}_{i}^{(2)} - \tilde{\rho}_{i}^{(3)}| \right) \nonumber \\
		&\quad \leq \left( \frac{n-1}{n-W} \sum_{j = 1}^{W} a_{j} \right) \frac{1}{n-1} \sum_{i=2}^{n}  \left( |\tilde{\rho}_{i} - \tilde{\rho}_{i}^{(2)}| +  |\tilde{\rho}_{i}^{(2)} - \tilde{\rho}_{i}^{(3)}| \right). \nonumber 
		\end{align}	
	}{
	\begin{eqnarray}
	|\hat{\rho}_{n} - \hat{\rho}_{n}^{(3)}| &\leq& \frac{1}{n-W} \sum_{i=W+1}^{n} |\bar{\rho}^{(i)} - \bar{\rho}^{(i)}_{3}| \nonumber \\
	&\leq& \frac{1}{n-W} \sum_{i=W+1}^{n} \sum_{j = i-W+1}^{i} a_{j} |\tilde{\rho}_{j} - \tilde{\rho}_{j}^{(3)}| \nonumber \\
	&\leq& \frac{1}{n-W} \sum_{i=W+1}^{n} \sum_{j = i-W+1}^{i} a_{j} \left( |\tilde{\rho}_{j} - \tilde{\rho}_{j}^{(2)}| +  |\tilde{\rho}_{j}^{(2)} - \tilde{\rho}_{j}^{(3)}| \right) \nonumber  \\
	&\leq& \frac{\sum_{j = 1}^{W} a_{j}}{n-W} \sum_{i=2}^{n}  \left( |\tilde{\rho}_{i} - \tilde{\rho}_{i}^{(2)}| +  |\tilde{\rho}_{i}^{(2)} - \tilde{\rho}_{i}^{(3)}| \right) \nonumber \\
	&\leq& \left( \frac{n-1}{n-W} \sum_{j = 1}^{W} a_{j} \right) \frac{1}{n-1} \sum_{i=2}^{n}  \left( |\tilde{\rho}_{i} - \tilde{\rho}_{i}^{(2)}| +  |\tilde{\rho}_{i}^{(2)} - \tilde{\rho}_{i}^{(3)}| \right). \nonumber 
	\end{eqnarray}
}
In comparison, we looked at controlling 
\[
\frac{1}{n-1} \sum_{i=2}^{n}  \left( |\tilde{\rho}_{i} - \tilde{\rho}_{i}^{(2)}| +  |\tilde{\rho}_{i}^{(2)} - \tilde{\rho}_{i}^{(3)}| \right)
\]
in the proof of \cref{rho_conc_eq}. The quantity of interest here is the same scaled by
\[
\frac{n-1}{n-W} \sum_{j = 1}^{W} a_{j}.
\]

By construction, we always have $\tilde{\rho}_{i}^{(3)} \geq \rho_{i}$. Therefore, by assumptions~\ref{probState:assumpB4}-\ref{probState:assumpB5}, it follows that
\[
\mathbb{E}\left[ h(\tilde{\rho}_{i}^{(3)},\tilde{\rho}_{i-1}^{(3)},\ldots,\tilde{\rho}_{i-W+1}^{(3)}) \;|\; \mathcal{F}_{i-1} \right] \geq \rho_{i}
\]
and $\mathbb{E}[\hat{\rho}_{n}^{(3)} \;|\; \mathcal{F}_{i-1}] \geq \rho$. Therefore, by applying \cref{subgauss:subgaussDepLem} and the Borel-Cantelli lemma, it follows for all $n$ large enough
\[
\hat{\rho}_{n}^{(3)} + t_{n} \geq \rho.
\]
This observation combined with a nearly identical proof to equality case shows that for all $n$ large enough and appropriate $\{t_{n}\}$
\[
\hat{\rho}_{n} + \left( \frac{n-1}{n-W} \sum_{j = 1}^{W} a_{j} \right) D_{n} + t_{n} \geq \rho
\]
almost surely.
\end{proof}

\subsubsection{$L_{2}$ Norm Condition}

Now, we look at analyzing $\hat{\rho}_{n}$ from \cref{estRho:combEq:L2} under the $L_{2}$ condition. First, we consider the condition in \cref{prob_form:opt_change_eq_L2}. Define the averaged estimate 
\[
\left(\hat{\rho}_{n}^{(3)}\right)^{2} \triangleq \frac{1}{n-1} \sum_{i=2}^{n} \left(\tilde{\rho}_{i}^{(3)}\right)^{2}
\]
and analogously $\left(\hat{\rho}_{n}^{(2)}\right)^{2}$. The following lemma shows that $\hat{\rho}_{n}^{(3)}$ upper bounds $\rho$ eventually.
\begin{lem}
	\label{estRho:rho3SqLemma}
	For all sequences $\{t_{n}\}$ such that
	\[
	\sum_{n=2}^{\infty} \exp\left\{ - \frac{2(n-1)t_{n}^{2}}{\text{diam}^{2}(\xSp)} \right\} < +\infty
	\]
	it holds that for all $n$ large enough
	\[
	\sqrt{\left(\hat{\rho}_{n}^{(3)}\right)^{2} + t_{n} } \geq \rho
	\]
	almost surely.
\end{lem} 
\begin{proof}
	First, for all $i$ it holds that
	\[
	\tilde{\rho}_{i}^{(3)} \geq \| \bx_{i}^{*} - \bx_{i-1}^{*}\|.
	\]
	This in turn implies that
	\[
	\mathbb{E}\left[\left( \tilde{\rho}_{i}^{(3)} \right)^{2} \;\Bigg|\; \mathcal{F}_{i-1} \right] \geq \mathbb{E}\left[ \| \bx_{i}^{*} - \bx_{i-1}^{*}\|^{2} \;|\; \mathcal{F}_{i-1} \right]  = \rho^{2}.
	\]
	Second, it holds that $0 \leq \left( \tilde{\rho}_{i}^{(3)} \right)^{2} \leq \text{diam}^{2}(\xSp)$. Applying \cref{estRho:condHoeffdingLemma} and \cref{subgauss:subgaussDepLem} yields
	\iftoggle{useTwoColumn} {
		\begin{align}
		\mathbb{P}&\left\{ \left( \hat{\rho}_{n}^{(3)} \right)^{2} < \rho^{2} - t_{n} \right\} \nonumber \\
		&\leq \mathbb{P}\left\{ \left( \hat{\rho}_{n}^{(3)} \right)^{2} < \frac{1}{n-1} \sum_{i=2}^{n} \mathbb{E}\left[ \left( \tilde{\rho}_{i}^{(3)} \right)^{2} \;\Bigg|\; \mathcal{F}_{i} \right] - t_{n} \right\} \nonumber \\
		&\leq \exp\left\{ - \frac{2(n-1)t_{n}^{2}}{\text{diam}^{2}(\xSp)} \right\}. \nonumber
		\end{align}	
	}{
	\begin{eqnarray}
	\mathbb{P}\left\{ \left( \hat{\rho}_{n}^{(3)} \right)^{2} < \rho^{2} - t_{n} \right\} &\leq& \mathbb{P}\left\{ \left( \hat{\rho}_{n}^{(3)} \right)^{2} < \frac{1}{n-1} \sum_{i=2}^{n} \mathbb{E}\left[ \left( \tilde{\rho}_{i}^{(3)} \right)^{2} \;\Bigg|\; \mathcal{F}_{i} \right] - t_{n} \right\} \nonumber \\
	&\leq& \exp\left\{ - \frac{2(n-1)t_{n}^{2}}{\text{diam}^{4}(\xSp)} \right\}. \nonumber
	\end{eqnarray}
}
By the Borel-Cantelli lemma, this in turn implies that for $n$ sufficiently large
\[
\sqrt{\left( \hat{\rho}_{n}^{(3)} \right)^{2} + t_{n}} \geq \rho.
\]
\end{proof}

We can now follow the proof technique of \cref{rho_conc_eq} and \cref{estRho:rho3SqLemma} to prove \cref{rho_conc_eq_sq}.
\begin{proof}[Proof of \cref{rho_conc_eq_sq}]
	This is a straightforward extension of the proof of \cref{rho_conc_eq} using the observation that
	we have
	\iftoggle{useTwoColumn} {
		\begin{align}
		&| (\hat{\rho}_{n})^{2} - (\hat{\rho}_{n}^{(3)})^{2}| \nonumber \\
		&\quad \leq | (\hat{\rho}_{n})^{2} - (\hat{\rho}_{n}^{(2)})^{2}| + |(\hat{\rho}_{n}^{(2)})^{2} - (\hat{\rho}_{n}^{(3)})^{2}| \nonumber \\
		&\quad \leq | \hat{\rho}_{n} + \hat{\rho}_{n}^{(2)}| | \hat{\rho}_{n} - \hat{\rho}_{n}^{(2)}| + |\hat{\rho}_{n}^{(2)} + \hat{\rho}_{n}^{(3)}| |\hat{\rho}_{n}^{(2)} - \hat{\rho}_{n}^{(3)}| \nonumber \\
		&\quad \leq 2 \text{diam}(\xSp) \left(  | \hat{\rho}_{n} - \hat{\rho}_{n}^{(2)}| + |\hat{\rho}_{n}^{(2)} - \hat{\rho}_{n}^{(3)}|  \right). \nonumber
		\end{align}	
	}{
	\begin{align}
	| (\hat{\rho}_{n})^{2} - (\hat{\rho}_{n}^{(3)})^{2}| & \leq | (\hat{\rho}_{n})^{2} - (\hat{\rho}_{n}^{(2)})^{2}| + |(\hat{\rho}_{n}^{(2)})^{2} - (\hat{\rho}_{n}^{(3)})^{2}| \nonumber \\
	& \leq | \hat{\rho}_{n} + \hat{\rho}_{n}^{(2)}| | \hat{\rho}_{n} - \hat{\rho}_{n}^{(2)}| + |\hat{\rho}_{n}^{(2)} + \hat{\rho}_{n}^{(3)}| |\hat{\rho}_{n}^{(2)} - \hat{\rho}_{n}^{(3)}| \nonumber \\
	& \leq 2 \text{diam}(\xSp) \left(  | \hat{\rho}_{n} - \hat{\rho}_{n}^{(2)}| + |\hat{\rho}_{n}^{(2)} - \hat{\rho}_{n}^{(3)}|  \right). \nonumber
	\end{align}	
}
We can now follow the proof technique of \cref{rho_conc_eq}.
\end{proof}

\begin{proof}[Proof of \cref{rho_conc_ineq_sq}]
	This is a straightforward extension of the proof of \cref{rho_conc_ineq} along the lines of \cref{rho_conc_eq_sq}.
\end{proof}

\section{Proofs for Analysis with Change in Minimizers Unknown}
\label{appendix_proofs_rho_unknown}
We prove a general result showing that for any choice of $\numIter_{n}$ such that $\numIter_{n} \geq \numIter^{*}$ for all n large enough with $\numIter^{*}$ from \cref{K_with_rho_known}, the \meangap{} is controlled in the sense that
\[
\limsup_{n \to \infty} \left( \mathbb{E}[f_{n}(\bx_{n})] - f_{n}(\bx_{n}^{*}) \right) \leq \epsilon.
\]

Consider the function
\begin{equation}
\label{withRhoUnknown:phiDefKey}
\phi_{\numIter}(v) = \alpha(\numIter) \left( \sqrt{\frac{2}{m}v} + \rho \right)^{2} + \beta(\numIter) = b\left(\sqrt{\frac{2}{m}v} + \rho , \numIter \right)
\end{equation}
from assumption~\ref{probState:assumpC2}. Note that as a function of $v$, $\phi_{\numIter}(v)$ is clearly increasing and strictly concave. If we select $\numIter^{*}$ defined in \cref{K_with_rho_known}, then by definition it holds that
\begin{equation}
\label{withRhoUnknown:keyPhiIneq}
\phi_{\numIter^{*}}(\epsilon) \leq \epsilon.
\end{equation}
First, we study fixed points of the function $\phi_{\numIter^{*}}(v)$. We need Theorem 3.3 of \cite{Kennan2001} to proceed.
\begin{lem}[Theorem 3.3 of \cite{Kennan2001}]
	\label{withRhoUnknown:kennanTheorem}
	Suppose that $f$ is an increasing and strictly concave 
	\an{mapping} from $\mathbb{R}$ to $\mathbb{R}$ such that $f(0) \geq 0$ and there exist points $0 < a < b$ such that $f(a) > a$ and $f(b) < b$. Then $f$ has unique positive fixed point.
\end{lem}
\begin{proof}
	See \cite{Kennan2001} for the proof.
\end{proof}
We consider the fixed points of the function $\phi_{\numIter^{*},\rho}(\nu) + \delta$ with $\delta \geq 0$. We add the term $\delta$ for reasons that will become clear later in the proof of \cref{withRhoUnknown:meanGapRhoKnownLemma}.
\begin{lem}
	\label{withRhoUnknown:fixedPointBasic}
	Provided that $\alpha(\numIter) > 0$ for all $\numIter > 0$, $\rho > 0$, and $\delta \geq 0$, the function $\phi_{\numIter^{*},\rho}(v) + \delta$ has a unique positive fixed point $\bar{v}_{\delta}$ with the following properties:
	\begin{enumerate}
		\item $\bar{\nu}_{0} = \phi_{\numIter^{*},\rho}(\bar{v}_{0}) \leq \epsilon$.
		\item $\phi'_{\numIter^{*},\rho}(\bar{v}_{\delta}) < 1$.
		\item $\bar{\nu}_{\delta}$ is non-decreasing in $\delta$ and 
		\[
		\lim_{\delta \searrow 0} \bar{\nu}_{\delta} = \bar{\nu}_{0}.
		\]
	\end{enumerate}
\end{lem}
\begin{proof}
	We have
	\[
	\phi_{\numIter^{*}}(0) + \delta = \alpha(\numIter^{*})\rho^{2} + \beta(\numIter^{*}) + \delta >0.
	\]
	Since 
	\[
	\lim_{v \to 0} \left( \phi_{\numIter^{*}}(v)  + \delta \right) = \phi_{\numIter^{*}}(0) + \delta
	\]
	and $\phi_{\numIter^{*}}(0) > 0$, for all $\delta \geq 0$, there exists a positive $a$ sufficiently small that
	\[
	\phi_{\numIter^{*}}(a) + \delta > a.
	\]
	Next, expanding $\phi_{\numIter}(v)$ yields
	\[
	\phi_{\numIter}(v) =  \frac{2}{m} \alpha(\numIter) v + 2 \alpha(\numIter) \rho \sqrt{\frac{2}{m}} \sqrt{v} + \alpha(\numIter) \rho^{2}  + \beta(\numIter).
	\]
	Since $\phi_{\numIter^{*}}(\epsilon) \leq \epsilon$, we obviously must have $\frac{2}{m} \alpha(\numIter^{*})  \leq 1$. Suppose that
	\[
	\frac{2}{m} \alpha(\numIter^{*})  = 1.
	\]
	Then it holds that
	\[
	\phi_{\numIter^{*}}(\epsilon) =  \epsilon + \sqrt{2m} \rho \sqrt{\epsilon} + \frac{m}{2} \rho^{2}  + \beta(\numIter) > \epsilon.
	\]
	This contradicts \cref{withRhoUnknown:keyPhiIneq}, so it holds that
	\[
	\frac{2}{m} \alpha(\numIter^{*}) < 1.
	\]
	It is thus readily apparent that
	\[
	v - \left( \phi_{\numIter^{*}}(v) + \delta \right) \to \infty
	\]
	as $v \to \infty$. Therefore, there exists a point $b > a$ such that
	\[
	\phi_{\numIter^{*}}(b) + \delta < b .
	\]
	In addition, it is easy to check that $\phi_{\numIter^{*}}(v) + \delta$ is increasing and strictly concave. Therefore, we can apply \cref{withRhoUnknown:kennanTheorem} from \cite{Kennan2001} to conclude that there exists a unique, positive fixed point $\bar{\nu}_{\delta}$ of $\phi_{\numIter^{*}}(v) + \delta$.
	
	Next, suppose that $\phi'_{\numIter^{*}}(\bar{\nu}_{\delta}) > 1$. Then by continuity for $v > \bar{\nu}_{\delta}$ sufficiently close to $\bar{\nu}_{\delta}$, we have
	\[
	\phi_{\numIter^{*}}(v) + \delta > v.
	\]
	However, we know that as $v \to \infty$, it holds that $v - \left( \phi_{\numIter^{*}}(v)  + \delta \right) \to \infty$. By the Intermediate Value Theorem, this implies that there is another fixed point on $[v,b]$. This is a contradiction, since $\bar{\nu}_{\delta}$ is the unique, positive fixed point. Therefore, it holds that $\phi'_{\numIter^{*}}(\bar{\nu}_{\delta}) \leq 1$. Now, suppose that $\phi'_{\numIter^{*}}(\bar{\nu}_{\delta}) = 1$. Since $\phi_{\numIter^{*}}(v)$ is strictly concave, its derivative is decreasing \cite{Boyd2004}. Therefore, on $[0,\bar{\nu}_{\delta})$, it holds that
	\[
	\phi'_{\numIter^{*}}(v) \geq 1.
	\]
	This implies that
	\begin{eqnarray}
	\bar{\nu}_{\delta} &=& \phi_{\numIter^{*}}(\bar{\nu}_{\delta}) + \delta \nonumber \\
	&=& \phi_{\numIter^{*}}(0) + \int_{0}^{\bar{\nu}_{\delta}} \phi'_{\numIter^{*}}(v) dx  + \delta \nonumber \\
	&\geq& \phi_{\numIter^{*}}(0) + \delta + \bar{\nu}_{\delta} \nonumber \\
	&>& \bar{\nu}_{\delta}. \nonumber
	\end{eqnarray}
	This is a contradiction, so it must be that $\phi'_{\numIter^{*}}(\bar{\nu}_{\delta}) < 1$. 
	
	Since there is a unique positive fixed point $\bar{\nu}_{\delta}$ and $v - \left( \phi_{\numIter^{*}}(v) + \delta \right) \to \infty$, it must hold that $\phi_{\numIter^{*}}(x) + \delta \leq x$ iff $x \geq \bar{\nu}_{\delta}$. Since $\phi_{\numIter^{*}}(\epsilon) \leq \epsilon$, it holds that $\bar{\nu}_{0} \leq \epsilon$.
	
	Finally, for $\delta' \geq \delta$, it holds that
	\begin{align}
	\bar{\nu}_{\delta} &= \phi_{\numIter^{*}}(\bar{\nu}_{\delta}) + \delta \nonumber \\
	&= \phi_{\numIter^{*}}(\bar{\nu}_{\delta}) + \delta' + \underbrace{(\delta - \delta')}_{< 0} \nonumber \\
	&< \phi_{\numIter^{*}}(\bar{\nu}_{\delta}) + \delta'.
	\end{align}
	By the observation above, we then have $\bar{\nu}_{\delta} \leq \bar{\nu}_{\delta'}$. This monotonicity in turn implies that
	\[
	\lim_{\delta \searrow 0} \bar{\nu}_{\delta} = \bar{\nu}_{0}.
	\]
\end{proof}
As a simple consequence of the concavity of $\phi_{\numIter^{*}}(v)$, we can study a fixed point iteration involving $\phi_{\numIter}(v)$. Define the $n$-fold composition mapping
\[
(\phi_{\numIter}+\delta)^{(n)}(v) \triangleq \left( (\phi_{\numIter}+\delta) \circ \cdots \circ (\phi_{\numIter}+\delta) \right)(v).
\]
\begin{lem}
	\label{withRhoUnknown:fixedPointIter}
	For any $v > 0$, it holds that
	\[
	\lim_{n \to \infty} (\phi_{\numIter^{*}}+\delta)^{(n)}(v) = \bar{\nu}_{\delta}.
	\]
\end{lem}
\begin{proof}
	Following \cite{Granas2003}, for any fixed point $\bar{\nu}$, it holds that 
	\[
	|\phi_{\numIter^{*}}(v) + \delta - \bar{\nu}_{\delta}| \leq \phi'_{\numIter^{*}}(\bar{\nu})|v - \bar{\nu}_{\delta}|.
	\]
	Therefore, applying the fixed point property repeatedly yields
	\[
	|(\phi_{\numIter^{*}}+\delta)^{(n)}(v) - \bar{\nu}_{\delta}| \leq (\phi'_{\numIter^{*}}(\bar{\nu}))^{n}|v - \bar{\nu}_{\delta}|.
	\]
	By \cref{withRhoUnknown:fixedPointBasic}, it holds that
	\[
	\phi'_{\numIter^{*}}(\bar{\nu}) < 1
	\]
	and so the result follows.
\end{proof}
This implies that if we select $\numIter^{*}$ stochastic gradients at every time instant, and we start from any $\nu$, then it holds that
\[
\phi^{(n)}_{\numIter^{*},\rho}(\nu) \to \bar{\nu}_{0}
\]
with $\bar{\nu}_{0} \leq \epsilon$.

Now, we show that we control the \meangap{} defined in \cref{prob_form:L2_crit} when we estimate $\rho$. In \Cref{tracking_crit_rho_known:meanAnalysis}, we pick a deterministic choice of $\numIter_{n} = \numIter^{*}$ and proceed with the analysis. Then it holds that
\iftoggle{useTwoColumn}{
\begin{align}
\mathbb{E}&[f_{n}(\bx_{n})] - f_{n}(\bx_{n}^{*}) \nonumber \\
&\leq \mathbb{E}\left[ b\left(  \sqrt{\frac{2}{m}(f_{n-1}(\bx_{n-1}) - f_{n-1}(\bx_{n-1}^{*}))} + \rho , \numIter_{n} \right) \right] \nonumber \\
\label{withRhoUnknown:factorFailPrev}
&= \mathbb{E}\left[ \alpha(\numIter_{n}) \left( \sqrt{\frac{2}{m}(f_{n-1}(\bx_{n-1}) - f_{n-1}(\bx_{n-1}^{*}))} + \rho \right)^{2} \right. \nonumber \\
&\left. \;\;\;\;\;\;\;\;\;\;\; + \beta(\numIter_{n}) \right] \\
&= \alpha(\numIter^{*}) \mathbb{E}\left[  \left( \sqrt{\frac{2}{m}(f_{n-1}(\bx_{n-1}) - f_{n-1}(\bx_{n-1}^{*}))} + \rho \right)^{2} \right] \nonumber \\
\label{withRhoUnknown:factorFail}
&\;\;\;\;\;\;\;\;\;\;\;  + \beta(\numIter^{*}).
\end{align}
}{ 
\begin{eqnarray}
\mathbb{E}[f_{n}(\bx_{n})] - f_{n}(\bx_{n}^{*}) &\leq& \mathbb{E}\left[ b\left(  \sqrt{\frac{2}{m}(f_{n-1}(\bx_{n-1}) - f_{n-1}(\bx_{n-1}^{*}))} + \rho , \numIter_{n} \right) \right] \nonumber \\
\label{withRhoUnknown:factorFailPrev}
&=& \mathbb{E}\left[ \alpha(\numIter_{n}) \left( \sqrt{\frac{2}{m}(f_{n-1}(\bx_{n-1}) - f_{n-1}(\bx_{n-1}^{*}))} + \rho \right)^{2} + \beta(\numIter_{n}) \right] \\
\label{withRhoUnknown:factorFail}
&=& \alpha(\numIter_{n}) \mathbb{E}\left[  \left( \sqrt{\frac{2}{m}(f_{n-1}(\bx_{n-1}) - f_{n-1}(\bx_{n-1}^{*}))} + \rho \right)^{2} \right]  + \beta(\numIter_{n}) \\
&=& \alpha(\numIter^{*}) \mathbb{E}\left[  \left( \sqrt{\frac{2}{m}(f_{n-1}(\bx_{n-1}) - f_{n-1}(\bx_{n-1}^{*}))} + \rho \right)^{2} \right]  + \beta(\numIter^{*}). \nonumber \\
\end{eqnarray}
}
We can bound 
\[
\mathbb{E}\left[  \left( \sqrt{\frac{2}{m}(f_{n-1}(\bx_{n-1}) - f_{n-1}(\bx_{n-1}^{*}))} + \rho \right)^{2} \right]
\]
using \cref{sol_sens:L2_d0} and recover \cref{withRhoUnknown:phiDefKey}. However, in this paper, $\numIter_{n}$ and $\bx_{n-1}$ are dependent random variables, so \eqref{withRhoUnknown:factorFail} does not hold in general. Instead, only \eqref{withRhoUnknown:factorFailPrev} holds. To get around this issue, we need a more sophisticated analysis using the observation that $\numIter_{n} \geq \numIter^{*}$ for all $n$ large enough. This property implies that $\numIter_{n}$ behaves like a constant for $n$ large enough and the analysis in \Cref{tracking_crit_rho_known:meanAnalysis} nearly applies. 

\begin{proof}[Proof of \cref{withRhoUnknown:meanGapRhoKnownLemma}]
	We know that for all $n$ large enough that we pick $\an{\numIter_n} \geq \numIter^{*}$ almost surely. This in turn implies that there exists a finite almost surely random variable $\tilde{N}$ such that
	\[
	n \geq \tilde{N} \;\;\Rightarrow\;\; \numIter_{n} \geq \numIter^{*}.
	\]
	Since $\tilde{N}$ is finite almost surely, we know that 
	\[
	\lim_{n \to \infty} \mathbb{P}\left\{ \tilde{N} > n \right\} = 0.
	\]
	By the compactness of $\xSp$, it follows that there is a constant $C > 0$ such that
	\[
	\max_{\bx \in \xSp} \phi_{\numIter^{*},\rho} \left( f_{n}(\bx) - f_{n}(\bx_{n}^{*}) \right) \leq C \;\; \forall n \geq 1.
	\]
	Then it follows that
	\iftoggle{useTwoColumn} {
		\begin{align}
		\mathbb{E}&[f_{n}(\bx_{n})] - f_{n}(\bx_{n}^{*}) \nonumber \\
		&= \mathbb{E}\left[ \phi_{\numIter_{n},\rho}\left( f_{n-1}(\bx_{n-1}) - f_{n-1}(\bx_{n-1}^{*})  \right)   \right] \nonumber \\
		&= \mathbb{E}\left[ \phi_{\numIter_{n},\rho}\left( f_{n-1}(\bx_{n-1}) - f_{n-1}(\bx_{n-1}^{*})  \right) \mathbbm{1}_{\{ n \geq \tilde{N} \}}   \right] \nonumber \\
		&\qquad + \mathbb{E}\left[ \phi_{\numIter_{n},\rho}\left( f_{n-1}(\bx_{n-1}) - f_{n-1}(\bx_{n-1}^{*})  \right) \mathbbm{1}_{\{ n < \tilde{N} \}}   \right] \nonumber \\
		&\leq \mathbb{E}\left[ \phi_{\numIter^{*},\rho}\left( f_{n-1}(\bx_{n-1}) - f_{n-1}(\bx_{n-1}^{*})  \right) \mathbbm{1}_{\{ n \geq \tilde{N} \}}   \right] \nonumber \\
		&\qquad + C \mathbb{P}\left\{ \tilde{N} > n \right\} \nonumber \\
		&\leq \phi_{\numIter^{*},\rho}\left( \mathbb{E}[f_{n-1}(\bx_{n-1})] - f_{n-1}(\bx_{n-1}^{*})  \right) + C \mathbb{P}\left\{ \tilde{N} > n \right\}. \nonumber
		\end{align}
	}{
	\begin{align}
	\mathbb{E}[f_{n}(\bx_{n})] - f_{n}(\bx_{n}^{*}) &= \mathbb{E}\left[ \phi_{\numIter_{n},\rho}\left( f_{n-1}(\bx_{n-1}) - f_{n-1}(\bx_{n-1}^{*})  \right)   \right] \nonumber \\
	&= \mathbb{E}\left[ \phi_{\numIter_{n},\rho}\left( f_{n-1}(\bx_{n-1}) - f_{n-1}(\bx_{n-1}^{*})  \right) \mathbbm{1}_{\{ n \geq \tilde{N} \}}   \right] \nonumber \\
	&\qquad + \mathbb{E}\left[ \phi_{\numIter_{n},\rho}\left( f_{n-1}(\bx_{n-1}) - f_{n-1}(\bx_{n-1}^{*})  \right) \mathbbm{1}_{\{ n < \tilde{N} \}}   \right] \nonumber \\
	&\leq \mathbb{E}\left[ \phi_{\numIter^{*},\rho}\left( f_{n-1}(\bx_{n-1}) - f_{n-1}(\bx_{n-1}^{*})  \right) \mathbbm{1}_{\{ n \geq \tilde{N} \}}   \right] + C \mathbb{P}\left\{ \tilde{N} > n \right\} \nonumber \\
	&\leq \phi_{\numIter^{*},\rho}\left( \mathbb{E}[f_{n-1}(\bx_{n-1})] - f_{n-1}(\bx_{n-1}^{*})  \right) + C \mathbb{P}\left\{ \tilde{N} > n \right\}. \nonumber
	\end{align}
}
To bound the \meangap{}, we consider the recursion
\begin{equation}
\label{withRhoUnknown:KeyThm:firstRecur}
\epsilon_{n} = \phi_{\numIter^{*},\rho}\left( \epsilon_{n-1} \right) + C \mathbb{P}\left\{ \tilde{N} > n \right\}  \qquad \forall n \geq \tilde{N}
\end{equation}
which satisfies
\[
\mathbb{E}[f_{n}(\bx_{n})] - f_{n}(\bx_{n}^{*}) \leq \epsilon_{n} \qquad \forall n \geq \tilde{N}.
\]
By assumption, we know that as $n \to \infty$
\[
C \mathbb{P}\left\{ \tilde{N} > n \right\} \to 0.
\]
Fix $\delta > 0$. Then there exists a random variable $\tilde{N}_{\delta} \geq \tilde{N}$ such that
\[
n \geq \tilde{N}_{\delta} \;\;\Rightarrow\;\; C \mathbb{P}\left\{ \tilde{N} > n \right\} \leq \delta.
\]
Then we consider the recursion
\begin{equation}
\label{withRhoUnknown:KeyThm:secondRecur}
\begin{aligned}
\tilde{\epsilon}_{n} &= \phi_{\numIter^{*},\rho}\left( \tilde{\epsilon}_{n-1} \right) + \delta \\
\tilde{\epsilon}_{\tilde{N}_{\delta}} &= \epsilon_{\tilde{N}_{\delta}}
\end{aligned}
\qquad \forall n \geq \tilde{N}_{\delta}.
\end{equation}
By construction, we have $\epsilon_{n} \leq \tilde{\epsilon}_{n}$ for all $n \geq \tilde{N}_{\delta}$. As a consequence of \cref{withRhoUnknown:fixedPointBasic} and \cref{withRhoUnknown:fixedPointIter}, we have
\begin{align}
\limsup_{n \to \infty}\left( \mathbb{E}[f_{n}(\bx_{n})] - f_{n}(\bx_{n}^{*}) \right) &\leq \limsup_{n \to \infty} \epsilon_{n} \nonumber \\
&\leq \limsup_{n \to \infty} \tilde{\epsilon}_{n} \nonumber \\
&\leq \bar{\nu}_{\delta}. \nonumber 
\end{align}
Since $\delta > 0$ was arbitrary and $\bar{\nu}_{\delta} \searrow \an{\bar{\nu}_0}$ as $\delta \searrow 0$ from \cref{withRhoUnknown:fixedPointBasic}, it follows that
\[
\limsup_{n \to \infty}\left( \mathbb{E}[f_{n}(\bx_{n})] - f_{n}(\bx_{n}^{*}) \right)  \leq \bar{\nu}_{0} \leq \epsilon.
\]
\end{proof}

\section{Examples of $b(d_{0},\numIter)$ Bound for SGD}
\label{bBounds}

We examine bounds $b(d_{0},\numIter)$ satisfying assumption~\ref{probState:assump4} for SGD \Cref{prob_form:sa_alg}. We form a convex combination of the iterates to yield a final approximate minimizer 
\begin{eqnarray}
\bar{\bx}(\numIter) &=& \sum_{\iterIndex=0}^{\numIter} \lambda(\iterIndex) \bx(\iterIndex). \nonumber
\end{eqnarray} 
Note that this includes the case where $\bar{\bx}(\numIter) = \bx(\numIter)$ by selecting $\lambda(\numIter) = 1$ and $\lambda(0) = \cdots = \lambda(\numIter-1) = 0$. 

Define
\begin{equation}
\label{bBounds:dDef}
d(\iterIndex) \triangleq \| \bx(\iterIndex) - \bx^{*} \|_{2}.
\end{equation}
First, we bound $\mathbb{E}[d(\iterIndex)]$ in \cref{bBounds:dBound}, which follows the classic Lyapunov function analysis of SGD \cite{Nemirovski2009}.
\begin{lem}
	\label{bBounds:dBound}
	It holds that
	\iftoggle{useTwoColumn}{
	\begin{align}
	\mathbb{E}&[d(\iterIndex)] \nonumber \\
	&\leq \prod_{\iterIndex = 1}^{\numIter} (1 - 2 m \mu(\iterIndex) + B \mu^{2}(\iterIndex))d^{2}(0) \nonumber \\
	& + A \sum_{\iterIndex = 1}^{\numIter} \prod_{i=\iterIndex+1}^{\numIter} (1 - 2 m \mu(i) + B \mu^{2}(i)) \mu^{2}(\iterIndex).
	\end{align}
	}{
	\[
	\mathbb{E}[d(\iterIndex)] \leq \prod_{\iterIndex = 1}^{\numIter} (1 - 2 m \mu(\iterIndex) + B \mu^{2}(\iterIndex))d^{2}(0) + A \sum_{\iterIndex = 1}^{\numIter} \prod_{i=\iterIndex+1}^{\numIter} (1 - 2 m \mu(i) + B \mu^{2}(i)) \mu^{2}(\iterIndex).
	\]
	}
\end{lem}
\begin{proof}
	See \cite{Nemirovski2009}.
\end{proof}
It is possible to further upper bound the bound in \cref{bBounds:dBound} to yield a closed form given in \cite{BachMoulines2011}; however, the bound in \cref{bBounds:dBound} is generally tighter. Next, we apply \cref{bBounds:dBound} along with a Lipschitz gradient assumption on $f(\bx)$ to produce a simple $b(d_{0},\numIter)$ bound.
\begin{lem}
	\label{bBounds:fBoundBasic}
	With arbitrary step sizes, assuming that $f(\bx)$ has Lipschitz continuous gradients with modulus $M$, and $\lambda(\numIter) = 1$, it holds that
	\[
	\mathbb{E}[f(\bar{\bx}(\numIter))] - f(\bx^{*}) \leq \frac{1}{2} M \mathbb{E}[d^{2}(\numIter)]
	\]
	and therefore, it holds that
	\iftoggle{useTwoColumn}{
	\begin{align}
	b(d_{0},\numIter) = \frac{1}{2} M &\left( \prod_{\iterIndex = 1}^{\numIter} (1 - 2 m \mu(\iterIndex) + B \mu^{2}(\iterIndex)) d_{0}^{2} \right. \nonumber \\
	&\left. + A \sum_{\iterIndex = 1}^{\numIter} \prod_{i=\iterIndex+1}^{\numIter} (1 - 2 m \mu(i) + B \mu^{2}(i)) \mu^{2}(\iterIndex) \right)
	\end{align}
	}{
	\[
	b(d_{0},\numIter) = \frac{1}{2} M \left( \prod_{\iterIndex = 1}^{\numIter} (1 - 2 m \mu(\iterIndex) + B \mu^{2}(\iterIndex)) d_{0}^{2} + A \sum_{\iterIndex = 1}^{\numIter} \prod_{i=\iterIndex+1}^{\numIter} (1 - 2 m \mu(i) + B \mu^{2}(i)) \mu^{2}(\iterIndex) \right)
	\]
	}
	satisfies the requirements of assumption~\ref{probState:assump4}.
\end{lem}
\begin{proof}
	Using the descent lemma from \cite{Bertsekas1999}, it holds that \\$\mathbb{E}[f(\bx)] - f(\bx^{*}) \leq \frac{1}{2} M \mathbb{E}[d(\numIter)]$. Plugging in the bound from \cref{bBounds:dBound} yields the bound $b(d_{0},\numIter)$.
\end{proof}

Next, we consider an extension of the averaging scheme derived with $B = 0$ in \cite{NedicLee12} to the case with $B > 0$ using the bounds in \cref{bBounds:dBound}. This averaging scheme puts weight
\[
\lambda(\iterIndex) = \frac{\frac{1}{\mu(\iterIndex)}}{\sum_{j=1}^{\numIter} \frac{1}{\mu(j)}}
\]
on the iterate $\bx(\iterIndex)$ with step size $\mu(\iterIndex) = \mathcal{O}(\iterIndex^{-1})$. Therefore, this averaging puts increasing weight on later iterates.

\begin{lem}
	\label{bBounds:fBoundNedLeeAve}
	With the choice of step sizes given by 
	\[
	\mu(\iterIndex) = \frac{1}{m (\iterIndex + 1)}  \;\;\;\; \forall \iterIndex \geq 1
	\]
	it holds that
	\[
	b(d_{0},\numIter) = \frac{(1+B) d_{0}^{2} +  B \sum_{\iterIndex=1}^{\numIter} \gamma(\iterIndex) + (\numIter+1) A}{ m (\numIter+1)(\numIter+4) }
	\]
	satisfies assumption~\ref{probState:assump4} where $\mathbb{E}[d(\iterIndex)] \leq \gamma(\iterIndex)$.
\end{lem}
\begin{proof}
	This proof is a straightforward extension of the proof in \cite{NedicLee12}. We have using standard analysis of SGD (see \cite{Nemirovski2009} for example)
	\iftoggle{useTwoColumn}{
	\begin{align}
	\mathbb{E}&[d^{2}(\iterIndex)] \nonumber \\
	&\leq (1 - 2 m \mu(\iterIndex) + B \mu^{2}(\iterIndex)) \mathbb{E}[d^{2}(\iterIndex-1)] \nonumber \\
	&\qquad - 2 \mu(\iterIndex) ( \mathbb{E}[f(\bx(\iterIndex - 1))] - f(\bx^{*})) + A \mu^{2}(\iterIndex).
	\end{align}
	}{
	\[
	\mathbb{E}[d^{2}(\iterIndex)] \leq (1 - 2 m \mu(\iterIndex) + B \mu^{2}(\iterIndex)) \mathbb{E}[d^{2}(\iterIndex-1)] - 2 \mu(\iterIndex) ( \mathbb{E}[f(\bx(\iterIndex - 1))] - f(\bx^{*})) + A \mu^{2}(\iterIndex).
	\]
	}
	Then dividing by $\mu^{2}(\iterIndex)$, we have
	\iftoggle{useTwoColumn}{
	\begin{align}
	\frac{1}{\mu^{2}(\iterIndex)} \mathbb{E}&[d^{2}(\iterIndex)] \nonumber \\
	& \leq \left( \frac{1 - 2 m \mu(\iterIndex)}{\mu^{2}(\iterIndex)} + B  \right) \mathbb{E}[d^{2}(\iterIndex-1)] \nonumber \\
	& - \frac{2}{\mu(\iterIndex)} ( \mathbb{E}[f(\bx(\iterIndex - 1))] - f(\bx^{*})) + A.
	\end{align}
	}{
	\[
	\frac{1}{\mu^{2}(\iterIndex)} \mathbb{E}[d^{2}(\iterIndex)] \leq \left( \frac{1 - 2 m \mu(\iterIndex)}{\mu^{2}(\iterIndex)} + B  \right) \mathbb{E}[d^{2}(\iterIndex-1)] - \frac{2}{\mu(\iterIndex)} ( \mathbb{E}[f(\bx(\iterIndex - 1))] - f(\bx^{*})) + A.
	\]
	}
	It holds that
	\begin{equation*}
	\frac{1 - 2 m \mu(\iterIndex)}{\mu^{2}(\iterIndex)} \leq \frac{1}{\mu^{2}(\iterIndex-1)}.
	\end{equation*}
	This implies that
	\iftoggle{useTwoColumn}{
	\begin{align}
	\frac{1}{\mu^{2}(\iterIndex)} &\mathbb{E}[d(\iterIndex)] - \frac{1}{\mu^{2}(\iterIndex-1)} \mathbb{E}[d^{2}(\iterIndex-1)] \nonumber \\
	&\leq  B \mathbb{E}[d^{2}(\iterIndex-1)] - \frac{2}{\mu(\iterIndex)} ( \mathbb{E}[f(\bx(\iterIndex - 1))] - f(\bx^{*})) + A. \nonumber
	\end{align}
	}{
	\[
	\frac{1}{\mu^{2}(\iterIndex)} \mathbb{E}[d(\iterIndex)] - \frac{1}{\mu^{2}(\iterIndex-1)} \mathbb{E}[d^{2}(\iterIndex-1)] \leq  B \mathbb{E}[d^{2}(\iterIndex-1)] - \frac{2}{\mu(\iterIndex)} ( \mathbb{E}[f(\bx(\iterIndex - 1))] - f(\bx^{*})) + A. 
	\]
	}
	Summing from $\iterIndex=1$ to $\numIter+1$ and rearranging yields
	\iftoggle{useTwoColumn}{
	\begin{align}
	\sum_{\iterIndex=0}^{\numIter} \frac{1}{\mu(\iterIndex + 1)} &\left( \mathbb{E}[f(\bx(\iterIndex))] - f(\bx^{*}) \right) \nonumber \\
	& \leq \frac{1}{2} (1 + B) d_{0}^{2} + \frac{1}{2} B \sum_{\iterIndex=1}^{\numIter} \mathbb{E}[d(\iterIndex)] + \frac{1}{2} (\numIter+1) A. \nonumber
	\end{align}
	}{
	\[
	\sum_{\iterIndex=0}^{\numIter} \frac{1}{\mu(\iterIndex + 1)} \left( \mathbb{E}[f(\bx(\iterIndex))] - f(\bx^{*}) \right) \leq \frac{1}{2} (1 + B) d_{0}^{2} + \frac{1}{2} B \sum_{\iterIndex=1}^{\numIter} \mathbb{E}[d(\iterIndex)] + \frac{1}{2} (\numIter+1) A.
	\]
	}
	
	With the weights
	\[
	\lambda(\iterIndex) = \frac{\frac{1}{\mu(\iterIndex + 1)}}{\sum_{\tau=0}^{\numIter} \frac{1}{\mu(j + 1)}}
	\]
	we have
	\iftoggle{useTwoColumn}{
	\begin{align}
	\mathbb{E}&[f(\bar{\bx}(\numIter))] - f(\bx^{*}) \nonumber \\
	&\leq \frac{\frac{1}{2} (1+B) d(0) + \frac{1}{2} B \sum_{\iterIndex=1}^{\numIter} \mathbb{E}[d^{2}(\iterIndex)] +  \frac{1}{2} (\numIter+1) A}{\sum_{\tau=0}^{\numIter} \frac{1}{\mu(\tau)}}. \nonumber
	\end{align}
	}{
	\[
	\mathbb{E}[f(\bar{\bx}(\numIter))] - f(\bx^{*}) \leq \frac{\frac{1}{2} (1+B) d(0) + \frac{1}{2} B \sum_{\iterIndex=1}^{\numIter} \mathbb{E}[d^{2}(\iterIndex)] +  \frac{1}{2} (\numIter+1) A}{\sum_{\tau=0}^{\numIter} \frac{1}{\mu(\tau)}}.
	\]
	}
	Then it holds that
	\[
	\sum_{\tau = 0}^{\numIter} \frac{1}{\mu(\tau+1)} = \sum_{\tau=0}^{\numIter} m (\tau + 2) = \frac{1}{2}m(\numIter+1)(\numIter+4) 
	\]
	so
	\iftoggle{useTwoColumn}{
	\begin{align}
	\mathbb{E}&[f(\bar{\bx}(\numIter))] - f(\bx^{*}) \nonumber \\
	& \quad \leq \frac{(1+B) d(0) +  B \sum_{\iterIndex=1}^{\numIter} \mathbb{E}[d(\iterIndex)] + (\numIter+1) A}{ m (\numIter+1)(\numIter+4) } \nonumber \\
	& \quad \leq \frac{(1+B) d_{0}^{2} +  B \sum_{\iterIndex=1}^{\numIter} \gamma(\iterIndex) + (\numIter+1) A}{ m (\numIter+1)(\numIter+4) }. \nonumber
	\end{align}
	}{
	\begin{eqnarray}
	\mathbb{E}[f(\bar{\bx}(\numIter))] - f(\bx^{*}) &\leq& \frac{(1+B) d(0) +  B \sum_{\iterIndex=1}^{\numIter} \mathbb{E}[d(\iterIndex)] + (\numIter+1) A}{ m (\numIter+1)(\numIter+4) } \nonumber \\
	&\leq& \frac{(1+B) d_{0}^{2} +  B \sum_{\iterIndex=1}^{\numIter} \gamma(\iterIndex) + (\numIter+1) A}{ m (\numIter+1)(\numIter+4) }. \nonumber
	\end{eqnarray}
	}
\end{proof}
To get the required $\gamma(\iterIndex)$ bounds, we use 
\cref{bBounds:dBound}. For the choice of step sizes in \cref{bBounds:fBoundNedLeeAve} from \cref{bBounds:dBound}, it holds that $\mathbb{E}[d(\iterIndex)] = \mathcal{O}\left( \frac{1}{\iterIndex} \right)$. Since
\[
\sum_{\iterIndex=1}^{\numIter} \frac{1}{\iterIndex} = \mathcal{O}\left( \log \numIter \right)
\]
it holds that
\[
\mathbb{E}[f(\bar{\bx}(\numIter))] - f(\bx^{*}) = \mathcal{O}\left( \frac{d_{0}^{2}}{\numIter^{2}} + \frac{\log(\numIter)}{\numIter^{2}} + \frac{1}{\numIter}  \right).
\]
The $\mathcal{O}(\frac{1}{\numIter})$ rate is minimax optimal for stochastic minimization of a strongly convex function \cite{NesterovBook2004}.

Next, we look at a special case of averaging from \cite{BachMoulines2011} for stochastic gradients such that
\[
\mathbb{E}\| \sgrad{\bx}{\bz}{}  - \sgrad{\tilde{\bx}}{\bz}{} - \shess{\bx}{\bz}{} \left( \bx - \tilde{\bx} \right) \|_{2}^{2} = 0 \;\;\;\; \forall \bx,\tilde{\bx} \in \xSp
\]
where $\shess{\bx}{\bz}{}$ is an unbiased stochastic second derivative with respect to $\bx$. Quadratic objectives satisfy this condition.
\begin{lem}
	\label{bBounds:fBoundBachAve}
	Assuming that
	\begin{enumerate}
		\item $\mathbb{E}\| \sgrad{\bx}{\bz}{}  - \sgrad{\tilde{\bx}}{\bz}{} - \shess{\bx}{\bz}{} \left( \bx - \tilde{\bx} \right) \|_{2}^{2} = 0 \;\;\;\; \forall \bx,\tilde{\bx} \in \xSp$
		\item $\mu(\iterIndex) = C \iterIndex^{-\alpha}$ with $\alpha \geq 1/2$
		\item $\lambda(0) = 0$ and $\lambda(\iterIndex) = 1/\numIter$ for $1 \leq \iterIndex \leq \numIter$
		\item $\mathbb{E}[d^{2}(\iterIndex)] \leq \gamma(\iterIndex)$
	\end{enumerate}
	it holds that
	\iftoggle{useTwoColumn}{
		\begin{align}
		\left(\mathbb{E}[\bar{d}(\numIter)] \right)^{1/2} &\leq  \frac{1}{m^{1/2}} \sum_{\iterIndex=1}^{\numIter-1} \bigg| \frac{1}{\mu(\iterIndex+1)}  - \frac{1}{\mu(\iterIndex)} \bigg| \left( \gamma(\iterIndex) \right)^{1/2}  \nonumber \\
		&\;\;\;\;\;\;\;\;\;\;\;\;  + \frac{1}{m^{1/2} \mu(1)}  \left( d_{0}^{2} \right)^{1/2} + \frac{1}{m^{1/2} \mu(\numIter)} \left( \gamma(\numIter) \right)^{1/2}  \nonumber \\
		&\;\;\;\;\;\;\;\;\;\;\;\;  + \sqrt{\frac{A}{ m \numIter}} +  \sqrt{\frac{2B}{m \numIter^{2}} \sum_{\iterIndex=1}^{\numIter} \mathbb{E}[d^{2}(\iterIndex-1)]} \nonumber
		\end{align}
	}{
	\begin{align}
	\left(\mathbb{E}[\bar{d}(\numIter)] \right)^{1/2} &\leq \frac{1}{m^{1/2}} \sum_{\iterIndex=1}^{\numIter-1} \bigg| \frac{1}{\mu(\iterIndex+1)}  - \frac{1}{\mu(\iterIndex)} \bigg| \left( \gamma(\iterIndex) \right)^{1/2} + \frac{1}{m^{1/2} \mu(1)}  \left( d_{0}^{2} \right)^{1/2} \nonumber \\
	&\;\;\;\;\;\;\;\;\;\;\;\; + \frac{1}{m^{1/2} \mu(\numIter)} \left( \gamma(\numIter) \right)^{1/2} + \sqrt{\frac{A}{ m \numIter}} +  \sqrt{\frac{2B}{m \numIter^{2}} \sum_{\iterIndex=1}^{\numIter} \mathbb{E}[d^{2}(\iterIndex-1)]} \nonumber
	\end{align}
}
with $\bar{d}(\numIter) = \| \bar{\bx}(\numIter) - \bx^{*}\|_{2}^{2}$. If $f$ has Lipschitz continuous gradients with modulus $M$, then it holds that
\iftoggle{useTwoColumn}{
\begin{align}
b(d_{0},\numIter) &= \frac{M}{2} \left( \frac{1}{m^{1/2}} \sum_{\iterIndex=1}^{\numIter-1} \bigg| \frac{1}{\mu(\iterIndex+1)}  - \frac{1}{\mu(\iterIndex)} \bigg| \left( \gamma(\iterIndex) \right)^{1/2} \right. \nonumber \\
& \qquad \qquad \left. + \frac{1}{m^{1/2} \mu(1)}  \left( d(0) \right)^{1/2}  \right. \nonumber \\
&\qquad \qquad \left. + \frac{1}{m^{1/2} \mu(\numIter)} \left( \gamma(\numIter) \right)^{1/2} + \sqrt{\frac{A}{ m \numIter}} \right. \nonumber \\
&\qquad \qquad \left. +  \sqrt{\frac{2B}{m \numIter^{2}} \sum_{\iterIndex=1}^{\numIter} \mathbb{E}[d(\iterIndex-1)]} \right)^{2} \nonumber
\end{align}
}{
\begin{align}
b(d_{0},\numIter) &= \frac{M}{2} \left( \frac{1}{m^{1/2}} \sum_{\iterIndex=1}^{\numIter-1} \bigg| \frac{1}{\mu(\iterIndex+1)}  - \frac{1}{\mu(\iterIndex)} \bigg| \left( \gamma(\iterIndex) \right)^{1/2} + \frac{1}{m^{1/2} \mu(1)}  \left( d(0) \right)^{1/2}  \right. \nonumber \\
&\;\;\;\;\;\;\;\;\;\;\;\;\;\;\;\;\;\;\; \left. + \frac{1}{m^{1/2} \mu(\numIter)} \left( \gamma(\numIter) \right)^{1/2} + \sqrt{\frac{A}{ m \numIter}} +  \sqrt{\frac{2B}{m \numIter^{2}} \sum_{\iterIndex=1}^{\numIter} \mathbb{E}[d(\iterIndex-1)]} \right)^{2} \nonumber
\end{align}
}
satisfies assumption~\ref{probState:assump4}.
\end{lem}
\begin{proof}
	See \cite{BachMoulines2011} for the proof
\end{proof}
This decays at rate $\mathcal{O}\left(\frac{1}{\numIter}\right)$ as long as $\mu(\iterIndex) = C\iterIndex^{-\alpha}$ with $\frac{1}{2} \leq \alpha \leq 1$. To get the bounds $\gamma(\iterIndex)$, we can again apply \cref{bBounds:dBound}.

\section{Parameter Estimation}
\label{parameterEstimation}

We may need to estimate parameters of the functions $\{f_{n}\}$ such as the strong convexity parameter $m$ to compute the bound $b(d_{0},\numIter)$ from assumption~\ref{probState:assump4}. In this section, we assume that the bound $b(d_{0},\numIter,\psi)$ is parameterized by $\psi$, which depends on properties of the functions $f_{n}(\bx)$. In most cases, we have parameters
\[
\psi = \left[ \begin{array}{cccc}
1/m & A & B
\end{array} \right]^{\top}
\]
where $m$ is the parameter of strong convexity, and the pair $(A,B)$ controls gradient growth as in assumption~\ref{probState:assump5}, i.e.,
\[
\mathbb{E}\| \sgrad{\bx}{\bz}{} \|_{2}^{2} \leq A +  B \| \bx-  \bx^{*} \|_{2}^{2}.
\]
We parameterize using $1/m$, since smaller $m$ increase the bound $b(d_{0},\numIter,\psi)$. Our goal is to produce an estimate $\hat{\psi}$ such that $\hat{\psi} \geq \psi^{*}$ with $\psi^{*}$ the true parameters. We present several general methods for estimating these parameters..

Similar to estimating $\rho$, we produce one time instant estimates $\tilde{m}_{i}$, $\tilde{A}_{i}$, and $\tilde{B}_{i}$ at time $i$ and combine them by averaging to yield
\begin{enumerate}
	\item $\hat{m}_{n} = \frac{1}{n} \sum_{i=1}^{n} \tilde{m}_{i}$
	\item $\hat{A}_{n} = \frac{1}{n} \sum_{i=1}^{n} \tilde{A}_{i}$
	\item $\hat{B}_{n} = \frac{1}{n} \sum_{i=1}^{n} \tilde{B}_{i}$.
\end{enumerate}

We make the following assumptions for our analysis:
\begin{enumerate}[label=\textbf{D.\arabic*}]
	\item \label{probState:assumpD1} The parameters $\psi \in \mathcal{P}$ with $\mathcal{P}$ compact and there exists a true set of parameters $\psi^{*}$.
	
	\item \label{probState:assumpD2}  The bound $b(d_{0},\numIter,\tilde{\psi})$ is non-decreasing in $\psi$, i.e.,
	\[
	\psi \leq \tilde{\psi} \;\;\Rightarrow\;\; b(d_{0},\numIter,\psi) \leq b(d_{0},\numIter,\tilde{\psi}).
	\]
	
	\item \label{probState:assumpD3} $\nabla f_{n}(\bx_{n})$ has Lipschitz continuous gradients with modulus $M$.
	
	\item \label{probState:assumpD4} $f_{n}(\bx)$ is twice differentiable and there exist stochastic second derivatives with respect to $\bx$, $\shess{\bx}{\bz}{}$, such  that
	\[
	\mathbb{E}_{\bz_{n} \sim p_{n}}\left[ \shess{\bx}{\bz_{n}}{n} \;|\; \bx  \right] = \nabla_{\bx\bx}^{2} f_{n}(\bx).
	\]
	\item \label{probState:assumpD5} The space $\mathcal{Z}$ is compact and there exists a constant $G$ such that $\| \sgrad{\bx}{\bz}{n}\|_{2} \leq G \;\;\;\; \forall \bx \; \forall \bz \; \forall n$.
	
	\item \label{probState:assumpD6} We have access to stochastic functions $\sfunc{\bx}{\bz}{n}$ such that
	\[
	\mathbb{E}[ \sfunc{\bx}{\bz}{n} \;|\; \bx ] = f_{n}(\bx).
	\]
\end{enumerate}

\subsection{Estimating the Strong Convexity Parameter}

We seek one step estimates $\tilde{m}_{i}$ of the parameter of the strong convexity such that
\[
\mathbb{E}[\tilde{m}_{i} \;|\; \mathcal{F}_{i-1}  ] \leq m.
\]

For any two points $\bx$ and $\tilde{\bx}$, by strong convexity we have
\[
f_{i}(\tilde{\bx}) \geq f_{i}(\bx) + \inprod{\nabla f_{i}(\bx)}{\tilde{\bx} -\bx} + \frac{1}{2} m \|\tilde{\bx} - \bx\|_{2}^{2} \;\;\;\;\; \forall \bx,\tilde{\bx} \in \xSp
\]
which implies that
\[
m \leq \frac{f_{i}(\tilde{\bx}) - f_{i}(\bx)-\inprod{\nabla f_{i}(\bx)}{\tilde{\bx} - \bx}}{\frac{1}{2}\|\tilde{\bx} - \bx\|_{2}^{2}}.
\]
We suppose that for all $n$
\[
m = \min_{\bx,\tilde{\bx} \in \xSp} \frac{f_{i}(\tilde{\bx}) - f_{i}(\bx)-\inprod{\nabla f_{i}(\bx)}{\tilde{\bx} - \bx}}{\frac{1}{2}\|\tilde{\bx} - \bx\|_{2}^{2}}.
\]
This is not restrictive since any $m > 0$ that satisfies
\[
m \leq \min_{\bx,\tilde{\bx} \in \xSp} \frac{f_{i}(\tilde{\bx}) - f_{i}(\bx)-\inprod{\nabla f_{i}(\bx)}{\tilde{\bx} - \bx}}{\frac{1}{2}\|\tilde{\bx} - \bx\|_{2}^{2}}
\]
can be taken as a parameter of strong convexity for the class of functions $f_{i}(\bx)$. We estimate this quantity for fixed $\bx$ and $\tilde{\bx}$ using the plug in approximation in \cref{paramEst:ratioDef}.
\iftoggle{useTwoColumn}{
\begin{figure*}[!t]
\normalsize
\begin{equation}
\label{paramEst:ratioDef}
r(\bx,\tilde{\bx}) \triangleq \frac{\frac{1}{\numIter_{i}} \sum_{\iterIndex=1}^{\numIter_{i}} \sfunc{\tilde{\bx}}{\bz_{i}(\iterIndex)}{i} - \frac{1}{\numIter_{i}} \sum_{\iterIndex=1}^{\numIter_{i}} \sfunc{\bx}{\bz_{i}(\iterIndex)}{i} - \inprod{\frac{1}{\numIter_{i}} \sum_{\iterIndex=1}^{\numIter_{i}}\sgrad{\bx}{\bz_{i}(\iterIndex)}{i}}{\tilde{\bx} - \bx}}{\frac{1}{2}\|\tilde{\bx} - \bx\|_{2}^{2}}
\end{equation}
\hrulefill
\vspace*{4pt}
\end{figure*}
}{
\begin{equation}
\label{paramEst:ratioDef}
r(\bx,\tilde{\bx}) \triangleq \frac{\frac{1}{\numIter_{i}} \sum_{\iterIndex=1}^{\numIter_{i}} \sfunc{\tilde{\bx}}{\bz_{i}(\iterIndex)}{i} - \frac{1}{\numIter_{i}} \sum_{\iterIndex=1}^{\numIter_{i}} \sfunc{\bx}{\bz_{i}(\iterIndex)}{i} - \inprod{\frac{1}{\numIter_{i}} \sum_{\iterIndex=1}^{\numIter_{i}}\sgrad{\bx}{\bz_{i}(\iterIndex)}{i}}{\tilde{\bx} - \bx}}{\frac{1}{2}\|\tilde{\bx} - \bx\|_{2}^{2}}.
\end{equation}
}
Then consider the following estimate of $m$:
\begin{equation}
\label{paramEst:strCvxRatioExact}
\tilde{m}_{i} \triangleq \min_{\bx,\tilde{\bx} \in \xSp} r(\bx,\tilde{\bx}).
\end{equation}
This estimate satisfies
\begin{eqnarray}
\mathbb{E}[\tilde{m}_{i} \;|\; \mathcal{F}_{i-1}] &=& \mathbb{E}\left[ \min_{\bx,\tilde{\bx} \in \xSp} r(\bx,\tilde{\bx}) \;|\; \mathcal{F}_{i-1}  \right] \nonumber \\
&\leq& \min_{\bx,\tilde{\bx} \in \xSp} \mathbb{E}\left[r(\bx,\tilde{\bx})  \right] \nonumber \\
&=& \min_{\bx,\tilde{\bx} \in \xSp} \frac{f_{i}(\tilde{\bx}) - f_{i}(\bx)-\inprod{\nabla f_{i}(\bx)}{\tilde{\bx} - \bx}}{\frac{1}{2}\|\tilde{\bx} - \bx\|_{2}^{2}} \nonumber \\
&=& m. \nonumber
\end{eqnarray}

Since computing the minimum here is difficult and is generally a non-convex problem, we can instead look at an approximate method. Suppose that we have $N$ points\\ $\bx(1),\ldots,\bx(N)$. Then for any two distinct points $\bx_{i}$ and $\bx_{j}$, we have
\[
m \leq \frac{f_{i}(\bx(i)) - f_{i}(\bx(j)) - \inprod{\nabla f_{i}(\bx(j))}{\bx(i) - \bx(j)}}{\frac{1}{2} \| \bx(i) - \bx(j) \|_{2}^{2}}.
\]
This suggests the estimate
\begin{equation}
\label{paramEst:strCvxOne}
\hat{m}_{i} \triangleq \min_{i \neq j} r(\bx(i),\bx(j))
\end{equation}
for the strong convexity parameter. Then we have
\begin{align}
\mathbb{E}&[\hat{m}_{i} \;|\; \mathcal{F}_{i-1}] \nonumber \\
&= \mathbb{E}\left[ \min_{i \neq j} r(\bx(i),\bx(j)) \;\Bigg|\; \mathcal{F}_{i-1}  \right] \nonumber \\
&\leq \min_{i \neq j} \mathbb{E}\left[ r(\bx(i),\bx(j)) \;\Bigg|\; \mathcal{F}_{i-1} \right] \nonumber \\
&\leq \min_{i \neq j} \frac{f_{i}(\bx(i)) - f_{i}(\bx(j)) - \inprod{\nabla f_{i}(\bx(j))}{\bx(i) - \bx(j)}}{\frac{1}{2} \| \bx(i) - \bx(j) \|_{2}^{2}}. \nonumber
\end{align}
It is difficult to compare this estimate to $m$ exactly. All we can say is that
\[
m \leq \min_{i \neq j} \frac{f_{i}(\bx(i)) - f_{i}(\bx(j)) - \inprod{\nabla f_{i}(\bx(j))}{\bx(i) - \bx(j)}}{\frac{1}{2} \| \bx(i) - \bx(j) \|_{2}^{2}}
\]
as well.

\subsection{Estimating Gradient Parameters}

We seek $(A,B)$ such that
\[
\mathbb{E}\| \sgrad{\bx}{\bz}{} \|_{2}^{2} \leq A + B \| \bx - \bx^{*}\|_{2}^{2}.
\]
Suppose that our functions have Lipschitz continuous gradients with modulus $M$, and we construct estimates of the modulus $\hat{M}_{i}$ analogous to \cref{paramEst:strCvxRatioExact} or \cref{paramEst:strCvxOne} by replacing the min with a max. Suppose that we select $N$ points $\bx(1),\ldots,\bx(N) \in \xSp$. We want to find $A$ and $B$ such that
\[
\mathbb{E}\| \sgrad{\bx(j)}{\bz}{} \|_{2}^{2} \leq A + B \| \bx(j) - \bx^{*}\|_{2}^{2}.
\]
By the Lipschitz gradient assumption, we have
\[
\| \bx - \bx^{*}\|_{2} \geq \frac{1}{M} \| \nabla f(\bx) \|_{2}.
\]
Therefore, the following implication holds
\iftoggle{useTwoColumn}{
	\begin{align}
	\mathbb{E}&\| \sgrad{\bx(j)}{\bz}{} \|_{2}^{2} \leq A + \frac{B}{M^{2}} \| \nabla f(\bx(j)) \|_{2}^{2} \nonumber \\
	& \qquad \qquad \Rightarrow \;\; \mathbb{E}\| \sgrad{\bx(j)}{\bz}{} \|_{2}^{2} \leq A + B \| \bx(j) - \bx^{*}\|_{2}^{2}.
	\end{align}
}{
\[
\mathbb{E}\| \sgrad{\bx(j)}{\bz}{} \|_{2}^{2} \leq A + \frac{B}{M^{2}} \| \nabla f(\bx(j)) \|_{2}^{2} \;\; \Rightarrow \;\; \mathbb{E}\| \sgrad{\bx(j)}{\bz}{} \|_{2}^{2} \leq A + B \| \bx(j) - \bx^{*}\|_{2}^{2}.
\]
}
We look for $(A,B)$ such that
\[
\mathbb{E}\| \sgrad{\bx(j)}{\bz}{} \|_{2}^{2} \leq A + \frac{B}{M^{2}} \| \nabla f(\bx()) \|_{2}^{2} \;\;\;\; j=1,\ldots,N.
\]
Define
\[
s_{i}(j) \triangleq \frac{1}{\numIter_{i}} \sum_{\iterIndex=1}^{\numIter_{i}} \| \sgrad{\bx(j)}{\bz_{i}(\iterIndex)}{i} \|_{2}^{2}
\]
and
\iftoggle{useTwoColumn}{
	$d_{i}(j)$ equal to
	\begin{align}
	g_{i}(j) - \frac{1}{\numIter_{i}-1} \sum_{\iterIndex=1}^{\numIter_{i}} \Bigg\| \sgrad{\bx(j)}{\bz_{i}(\iterIndex)}{i} - \frac{1}{\numIter_{i}} \sum_{p=1}^{\numIter_{i}} \sgrad{\bx(p)}{\bz_{i}(p)}{i} \Bigg\|_{2}^{2}. \nonumber
	\end{align}
}{
\[
d_{i}(j) \triangleq g_{i}(j) - \frac{1}{\numIter_{i}-1} \sum_{\iterIndex=1}^{\numIter_{i}} \Bigg\| \sgrad{\bx(j)}{\bz_{i}(\iterIndex)}{i} - \frac{1}{\numIter_{i}} \sum_{p=1}^{\numIter_{i}} \sgrad{\bx(p)}{\bz_{i}(p)}{i} \Bigg\|_{2}^{2}.
\]
}
We want to find $(A,B)$ such that
\[
s_{i}(j) \leq A + \frac{B}{(\hat{M}_{i-1} + t_{i-1})^{2}} d_{i}(j) \;\;\;\; j=1,\ldots,N.
\]

Suppose that we are given a function $\phi(A,B)$ that controls the size of $(A,B)$. For example, we may have $\phi(A,B) = \frac{1}{2}A^{2} + \frac{1}{2}B^{2}$ or $\phi(A,B) = \lambda A^{2} + (1-\lambda)B^{2}$ with $0 \leq \lambda \leq 1$. We solve
\begin{equation}
\label{paramEst:abApproxEst}
\begin{aligned}
& \underset{\tilde{A}_{i},\tilde{B}_{i}}{\text{minimize}}
& & \phi(\tilde{A}_{i},\tilde{B}_{i})  \\
& \text{subject to}
& & s_{i}(j) \leq \tilde{A}_{i} + \frac{\tilde{B}_{i}}{(\hat{M}_{i-1} + t_{i-1})^{2}} d_{i}(j), \;\;\; j = 1, \ldots, N \\
& & & \tilde{A}_{i} \geq 0 \;,\; \tilde{B}_{i} \geq 0
\end{aligned}
\end{equation}
to generate approximate $(\tilde{A}_{i},\tilde{B}_{i})$.

\subsection{Combining One Step Estimates}
One issue in parameter estimation is that there may be some dependencies among the various estimates which need to be accounted for. For example, the estimates for $(A,B)$ in \cref{paramEst:abApproxEst} depend on estimates for the Lipschitz modulus $M$. We show that this does not impact our estimation process using \cref{paramEst:firstEst} and \cref{paramEst:secondEst}. First, we present a result showing that if we plug in the true parameters that our estimates work.
\begin{lem}
	\label{paramEst:firstEst}
	Suppose that we estimate $\phi^{*}$ by averaging the estimates $\phi_{i}(\pi^{*})$ where $\pi^{*}$ are the true parameters on which the estimate $\phi_{i}$ depends and the following conditions hold:
	\begin{enumerate}
		\item $|\phi_{i}(\pi^{*})| \leq C$
		\item $\mathbb{E}[\phi_{i}(\pi^{*}) \;|\; \mathcal{F}_{i-1}] \geq \phi^{*}$
		\item $\sum_{n=1}^{\infty} \exp\left\{ -\frac{2n t_{n}^{2}}{C^{2}} \right\} < +\infty$
	\end{enumerate}
	Then for all $n$ large enough, it holds that
	\[
	\frac{1}{n} \sum_{i=1}^{n} \phi_{i}(\pi^{*}) + t_{n} \geq \phi^{*}
	\]
	almost surely.
\end{lem}
\begin{proof}
	Since $|\phi_{i}(\pi^{*})| \leq C$, by applying \cref{estRho:condHoeffdingLemma} it holds that
	\[
	\mathbb{E}\left[ e^{ s\left( \phi_{i}(\pi^{*}) - \mathbb{E}[\phi_{i}(\pi^{*}) \;|\; \mathcal{F}_{i-1}]  \right) } \;|\; \mathcal{F}_{i-1}  \right] \leq \exp\left\{ \frac{1}{2} \frac{C^2}{4} s^{2} \right\}.
	\]
	Then by \cref{subgauss:subgaussDepLem}, it holds that
	\iftoggle{useTwoColumn}{
		\begin{align}
		\mathbb{P}&\left\{ \frac{1}{n} \sum_{i=1}^{n} \phi_{i}(\pi^{*}) < \phi^{*} - t_{n} \right\} \nonumber \\
		&\quad = \mathbb{P}\left\{ \frac{1}{n} \sum_{i=1}^{n} \phi_{i}(\pi^{*}) < \frac{1}{n} \sum_{i=1}^{n} \mathbb{E}\left[ \phi_{i}(\pi^{*}) \;|\; \mathcal{F}_{i-1} \right] - t_{n} \right\} \nonumber \\
		&\quad \leq \exp\left\{ -\frac{2n t_{n}^{2}}{C^{2}} \right\}. \nonumber
		\end{align}
	}{
	\begin{eqnarray}
	\mathbb{P}\left\{ \frac{1}{n} \sum_{i=1}^{n} \phi_{i}(\pi^{*}) < \phi^{*} - t_{n} \right\} &=& \mathbb{P}\left\{ \frac{1}{n} \sum_{i=1}^{n} \phi_{i}(\pi^{*}) < \frac{1}{n} \sum_{i=1}^{n} \mathbb{E}\left[ \phi_{i}(\pi^{*}) \;|\; \mathcal{F}_{i-1} \right] - t_{n} \right\} \nonumber \\
	&\leq& \exp\left\{ -\frac{2n t_{n}^{2}}{C^{2}} \right\}. \nonumber
	\end{eqnarray}
}
Since it holds that
\begin{eqnarray}
\sum_{n=1}^{\infty} \mathbb{P}\left\{ \frac{1}{n} \sum_{i=1}^{n} \phi_{i}(\pi^{*}) < \phi^{*} - t_{n} \right\} &\leq& \sum_{n=1}^{\infty} \exp\left\{ -\frac{2n t_{n}^{2}}{C^{2}} \right\}  < +\infty \nonumber
\end{eqnarray}
by the Borel-Cantelli lemma, it follows that for all $n$ large enough
\[
\frac{1}{n} \sum_{i=1}^{n} \phi_{i}(\pi^{*}) + t_{n} \geq \phi^{*}.
\]
\end{proof}

Since our estimates of $m$ do not depend on any parameters $\pi$, this lemma shows that the averaged estimate eventually lower bounds $m$. Similar reasoning holds for the Lipschitz modulus $M$. \Cref{paramEst:firstEst} also shows that estimates of $(A,B)$ upper bound the true quantities provided that the true value of $m$ and $M$ are plugged in. We bootstrap from this result to show that the estimates of $A$ and $B$ upper bound the exact quantities using \cref{paramEst:secondEst}. Before proceeding, note that random variables $X_{n}$ are $o_{\mathbb{P}}(1)$ if
\[
\lim_{n \to \infty} \mathbb{P}\left\{ |X_{n}| \geq t \right\} = 0 \;\;\;\; \forall t > 0.
\]

\begin{lem}
	\label{paramEst:secondEst}
	Suppose that we estimate $\phi^{*}$ by averaging the estimates $\phi_{i}(\pi_{i})$ where $\pi_{i}$ are the estimates of the parameters on which the estimate $\phi_{i}$ depends and the following hold:
	\begin{enumerate}
		\item $|\phi_{i}(\pi)| \leq C$
		\item For all $n$ large enough $\pi_{n} \geq \pi^{*}$ almost surely
		\item $\pi \leq \tilde{\pi} \;\;\Rightarrow\;\; \phi_{i}(\pi) \leq \phi_{i}(\tilde{\pi})$
		\item For appropriate sequences $t_{n}$, $\frac{1}{n} \sum_{i=1}^{n} \phi_{i}(\pi^{*}) + t_{n} \geq \phi^{*}$
	\end{enumerate}
	Then for all $n$ large enough, it holds that
	\[
	\frac{1}{n} \sum_{i=1}^{n} \phi_{i}(\pi_{i}) + t_{n} \geq \phi^{*} + o_{\mathbb{P}}(1)
	\]
	almost surely.
\end{lem}
\begin{proof}
	There exists a finite almost surely random variable $\tilde{N}$ such that
	\[
	n \geq \tilde{N} \;\;\Rightarrow\;\; \pi_{i} \geq \pi^{*}.
	\]
	It holds that
	\begin{eqnarray}
	\frac{1}{n} \sum_{i=1}^{n} \phi_{i}(\pi_{i}) &=& \frac{1}{n} \sum_{i=1}^{\tilde{N}-1} \phi_{i}(\pi_{i}) + \frac{1}{n} \sum_{i = \tilde{N}}^{n} \phi_{i}(\pi_{i}) \nonumber \\
	&\geq& \frac{1}{n} \sum_{i=1}^{\tilde{N}-1} \phi_{i}(\pi_{i}) + \frac{1}{n} \sum_{i = \tilde{N}}^{n} \phi_{i}(\pi^{*}) \nonumber \\
	&=& \frac{1}{n} \sum_{i=1}^{\tilde{N}-1} \left( \phi_{i}(\pi_{i}) - \phi_{i}(\pi^{*}) \right) + \frac{1}{n} \sum_{i = 1}^{n} \phi_{i}(\pi^{*}). \nonumber
	\end{eqnarray}
	By the boundedness of $\phi(\pi)$, this implies that
	\begin{eqnarray}
	\frac{1}{n} \sum_{i=1}^{n} \phi_{i}(\pi_{i}) + t_{n} &=& \left( \frac{1}{n} \sum_{i = 1}^{n} \phi_{i}(\pi^{*}) + t_{n} \right) + o_{\mathbb{P}}(1) \nonumber \\
	&\geq& \phi^{*} +  o_{\mathbb{P}}(1). \nonumber
	\end{eqnarray}
\end{proof}
This result proves that estimates for $(A,B)$ work as a result of the estimates for $m$ and $M$ working.

\end{document}